\documentclass[letterpaper,11pt,reqno]{amsart}
\makeatletter
\providecommand\abx@aux@refcontext[1]{}
\providecommand\abx@aux@cite[1]{}
\providecommand\abx@aux@segm[3]{}
\makeatother

\usepackage[margin=1in]{geometry}

\usepackage{hyperref}
\hypersetup{
     colorlinks   = true,
     citecolor    = black,
     linkcolor    = blue,
     urlcolor     = black
}

\usepackage{graphicx,amsmath,amssymb,amsthm,paralist,color,tikz-cd}
\theoremstyle{plain}
\newtheorem{thm}{Theorem}[section]
\newtheorem{lem}[thm]{Lemma}
\newtheorem{prop}[thm]{Proposition}
\newtheorem{Proposition}[thm]{Proposition}

\theoremstyle{definition}
\newtheorem{definition}[thm]{Definition}
\newtheorem{Definition}[thm]{Definition}

\theoremstyle{remark}
\newtheorem*{remark}{Remark}

\usepackage{mathrsfs}
\usepackage{mathtools}

\usepackage[utf8]{inputenc} 
\usepackage[T1]{fontenc}

\usepackage{enumitem}

\usepackage{stmaryrd}
\usepackage{comment}

\usepackage{dirtytalk}

\usepackage{xcolor}

\newtheoremstyle{case}{}{}{}{}{}{:}{ }{}
\theoremstyle{case}

\newcommand{\N}{\mathbb{N}}

\newcommand{\R}{\mathbb{R}}
\newcommand{\C}{\mathbb{C}}

\newcommand{\mbf}{\mathbf}

\renewcommand{\a}{\alpha}
\renewcommand{\b}{\beta}

\def\multiset#1#2{\ensuremath{\left(\kern-.3em\left(\genfrac{}{}{0pt}{}{#1}{#2}\right)\kern-.3em\right)}}

\newcommand{\cM}{\mathcal{M}}

\newcommand{\cL}{\mathcal{L}}

\newcommand{\cF}{\mathcal{F}}

\newcommand{\cW}{\mathcal{W}}

\newcommand{\cP}{\mathcal{P}}

\newcommand{\cA}{\mathcal{A}}

\numberwithin{equation}{section}

\newcommand{\A}{\mathbf{A}}
\newcommand{\B}{\mathbf{B}}
\renewcommand{\a}{\mathbf{a}}
\renewcommand{\b}{\mathbf{b}}
\renewcommand{\c}{\mathbf{c}}
\renewcommand{\v}{\mathbf{v}}
\newcommand{\bd}{\mbf{d}}
\newcommand{\be}{\mbf{e}}

\newcommand{\1}{\mathrm{\mathbf{1}}}
\newcommand{\eps}{\varepsilon}

\newcommand{\Hd}{\dim_\mathrm{H}}

\newcommand{\la}{\lambda}

\newcommand{\al}{\alpha}

\newcommand{\ga}{\gamma}

\newcommand{\si}{\sigma}
\newcommand{\Si}{\Sigma}

\newcommand{\ex}{\exists}
\newcommand{\fa}{\forall}
\newcommand{\Lo}{\Longleftrightarrow}
\newcommand{\im}{\implies}

\newcommand{\su}[1]{\left(#1\right)}

\newcommand{\mar}[1]{\mathrm{#1}}

\newcommand{\mal}[1]{\mathcal{#1}}

\newcommand{\no}[1]{\left|\left|#1\right|\right|}
\newcommand{\ab}[1]{\left|#1\right|}
\newcommand{\kk}[1]{\left[#1\right]}

\newcommand{\ka}[1]{\left[#1\right)}
\newcommand{\se}[1]{\left\{#1\right\}}

\renewcommand{\a}{\mathbf{a}}
\renewcommand{\b}{\mathbf{b}}
\renewcommand{\c}{\mathbf{c}}

\usepackage{amsfonts}

\title{Fourier decay in parabolic $C^{1+\alpha}$ systems with overlaps}

\author{Gaétan Leclerc}
\address{Department of Mathematics and Statistics, University of Helsinki, Finland}
\email{gaetan.leclerc@helsinki.fi}

\author{Sampo Paukkonen}
\address{Department of Mathematics and Statistics, University of Helsinki, Finland}
\email{sampo.paukkonen@helsinki.fi}

\author{Tuomas Sahlsten}
\address{Department of Mathematics and Statistics, University of Helsinki, Finland}
\email{tuomas.sahlsten@helsinki.fi}

\date{}

\thanks{The authors are all supported by the Research Council of Finland's Academy Research Fellowship \emph{``Quantum chaos of large and many body systems''}, grant Nos. 347365, 353738.}

\begin{document}

\begin{abstract}
We establish power Fourier decay for equilibrium states of parabolic $C^{1+\alpha}$ iterated function systems with overlaps satisfying a multiscale nonlinearity condition. This class includes the Lyons conductance measures $\nu_t$, $0<t<1$, associated to Galton-Watson trees with equal weights yielding advance towards a conjecture of Lyons on the absolute continuity of $\nu_t$ for small $t$. Further applications include Patterson-Sullivan measures for cusped hyperbolic surfaces, extending the work of Bourgain and Dyatlov to parabolic settings, conformal measures for Manneville-Pommeau and Lorenz-type maps, and the construction of the first genuinely $C^{1+\alpha}$ IFSs whose attractors have positive Fourier dimension but are not $C^1$-conjugate to linear IFSs. The proof combines the Bourgain-Dyatlov sum-product strategy with a multiscale induction approach that bypasses the use of spectral gaps for twisted transfer operators needed in several other works in the area.
\end{abstract}

\maketitle


\section{Introduction}

\subsection{Summary of the main results}

There has been substantial recent interest in the high-frequency behaviour of Fourier transforms of measures 
$\widehat{\mu}(\xi)=\int e^{-2\pi i \xi x}\,d\mu(x)$, arising from random constructions \cite{FalconerJin2019,SuomalaShmerkin,GaVargas,SuomalaChen,ChenHan,lai2025fourier,lee2025fourierdimensionfractionalbrownian}, dynamical systems and quantum chaos \cite{Kaufman2,JorSah,BourgainDyatlov,LiNaudPan,MosqueraShmerkin,AlgomHertzWang-Log,AlgomHertzWang-Normality,SahlstenStevens,BakerBanaji,BakerSahlsten,Leclerc-JuliaSets,leclerc2024nonlinearityfractalsfourierdecay,BakerKhalilSahlsten,LeclercPreprint,AvilaLyubichZhang}, and additive combinatorics and geometric measure theory \cite{orponenAdditive2023,orponen2025fourier,demeter2024szemer,yi2024bounded,Khalil-Mixing,fraserspectrum,AlgomKhalil,AlgomOrponen}; see also \cite{Sahlsten-survey}. The study of power Fourier decay is a signature of smoothness of the measure, and can be used to study equidistribution and Diophantine properties of numbers in the support \cite{DavenportErdosLeVeque,PVZZ} and it is also closely related to proving absolute continuity notably in Shmerkin’s work on Bernoulli convolutions \cite{Shmerkin-AbsContBernoulliConv} and higher dimensions by Lindenstrauss and Varj\'u \cite{LindenstraussVarju}. 

For dynamically defined measures, most literature concern uniformly hyperbolic systems with $C^2$ or analytic regularity and separation, where Fourier decay follows from spectral gaps for twisted transfer operators via Dolgopyat’s method \cite{JorSah,BourgainDyatlov,MosqueraShmerkin,AlgomHertzWang-Log,AlgomHertzWang-Normality,SahlstenStevens,BakerBanaji,BakerSahlsten,Leclerc-JuliaSets,leclerc2024nonlinearityfractalsfourierdecay}. In contrast, many natural measures arise from overlapping systems (Bernoulli convolutions \cite{Shmerkin-AbsContBernoulliConv}, Axiom $A$ diffeomorphisms \cite{LeclercPreprint}) or from parabolic dynamics (for example Lorenz-type systems \cite{Lorenz}), and fall outside this framework. An illustrative example is the family of Lyons conductance measures $\nu_t$, $0<t<1$, describing effective conductance in Galton-Watson trees with equal weights \cite{Lyons} arising as self-conformal measures for nonlinear and overlapping parabolic IFSs. Their smoothness properties such as absolute continuity remain difficult to access \cite{simon2001invariant,mihailescu2016random} and as with Bernoulli convolutions they are in fact conjectured to be absolutely continuous for all small $t$ by Lyons \cite{Lyons}. 

The purpose of this paper is to introduce a purely symbolic multiscale induction strategy to prove power Fourier decay for equilibrium states of overlapping parabolic $C^{1+\alpha}$ iterated function systems under a Quantitative Non-Linearity condition (Theorem~\ref{thm:mainCountable}), for which we provide broadly applicable criteria (Theorem~\ref{thm:QNL}). Our approach completely bypasses spectral gaps for twisted transfer operators \cite{LiENS,SahlstenStevens,AlgomRHWang-Polynomial,BakerSahlsten,BakerKhalilSahlsten}, allowing us to treat genuinely $C^{1+\alpha}$ and parabolic systems beyond the scope of current smooth hyperbolic methods. As an application, we prove power Fourier decay for the Lyons conductance measures $\nu_t$ for all $0<t<1$ (Theorem~\ref{thm:lyons}). We note the analogous problem for Bernoulli convolutions would be useful for proving absolute continuity \cite{Shmerkin-AbsContBernoulliConv} and our results highlights the strong effect of nonlinearity as an amplifying effect for smoothness.

Further applications of Theorem~\ref{thm:mainCountable} include power Fourier decay of Patterson--Sullivan measures for hyperbolic surfaces with parabolic elements, extending Bourgain–Dyatlov who used these in the study of scattering resonances in the convex co-compact setting \cite{BourgainDyatlov} (Theorem~\ref{thm:ps}), and all $\delta$-conformal measures for Manneville–Pomeau \cite{MP} and Lorenz-type maps arising from the Lorenz system \cite{Lorenz} (Theorem~\ref{thm:FourierMP}). The multiscale induction we develop is flexible enough to treat genuinely $C^{1+\alpha}$ systems where smooth tools are unavailable. In particular, Theorem~\ref{thm:mainCountable} yields the first examples of $C^{1+\alpha}$ IFSs not conjugate to linear systems whose equilibrium measures exhibit power Fourier decay under a condition we call Measure Non-Linearity (Theorem~\ref{thm:C1alpha}).

We now proceed to state the main Fourier decay theorem under QNL in the next section and criteria to check QNL, and then discuss in detail the applications above.

\subsection{QNL and power Fourier decay}\label{sec:mainthm}

 Let $\mathcal{A}$ be a finite or countable alphabet and let
$\Sigma^\N=\{(a_1,a_2,\ldots)\in\mathcal{A}^\N:\ 
A(a_i,a_{i+1})=1\}$ be a subshift of finite type with admissible words
$\Sigma^n$ and $\Sigma^*=\bigcup_{n\ge1}\Sigma^n$.
Let $\{I_a\}_{a\in\mathcal A}$ be open intervals in $[0,1]$ and for
each admissible pair $ab$ let
$g_{ab}:I_b\to I_a$ be a $C^{1+\alpha}$-diffeomorphic contraction, defining
$g_\mathbf a:=g_{a_1a_2}\circ\cdots\circ g_{a_{n-1}a_n}$ for
$\mathbf a\in\Sigma^n$. Assume uniform $\kappa_+$-contraction
$\kappa_-(ab)\le |g_{ab}'(x)|\le \kappa_+<1$ and bounded
$(C,\alpha)$-distortion
$|\log |g_{ab}'(x)|-\log |g_{ab}'(y)||\le C|x-y|^\alpha$ for all $x,y$. Then there exists attractor
$F=F_\Phi:=\bigcap_{n\ge1}\ \bigcup_{\mathbf a\in\Sigma^n}
\overline{g_\mathbf a(I_{a_n})}$.
Let $\phi$ be a normalised potential with summable variations and let
$\mu_\phi$ be the associated Gibbs measure. In the countable-branch
case we additionally assume the light-tail condition: there exists
$\gamma_0>0$ such that
\[
\sum_{ab\in\Sigma^2}\ \sup_{x\in I_b}
e^{\phi(g_{ab}(x))}\, |g_{ab}'(x)|^{-\gamma_0}<\infty.
\]
For $\mathbf a,\mathbf b\in\Sigma^n$ and $x,y\in F$ define
\[
\Delta_n((\mathbf a,x),(\mathbf b,y))
:=
(\log g_\mathbf a'(x)-\log g_\mathbf b'(x))
-
(\log g_\mathbf a'(y)-\log g_\mathbf b'(y)).
\]
We say that $\Phi$ satisfies \emph{(QNL)} if there exist
$\rho,\Theta\in(0,1)$ such that for any interval
$I\subset\R$ and $n\in\N$,
\[
\sum_{\mathbf a,\mathbf b,\mathbf c,\mathbf d\in\Sigma^n}
\mu(I_\mathbf a)\mu(I_\mathbf b)\mu(I_\mathbf c)\mu(I_\mathbf d)
\,
\mathbf 1_I\!\left(
\Delta_n((\mathbf a,x_\mathbf c),(\mathbf b,x_\mathbf d))
\right)
\lesssim
|I|^\Theta+\rho^n.
\]

\begin{thm}\label{thm:mainCountable}
Let $\Phi$ be a $C^{1+\alpha}$ IFS as above, satisfying uniform
$\kappa_+$-contraction and bounded $(C,\alpha)$-distortion. Let $\phi$
be a normalised potential with summable variations satisfying the
light-tail condition, and let $\mu_\phi$ be the associated Gibbs
measure without atoms. If \emph{(QNL)} holds, then there exists $\beta>0$ such that
\[
\Big|\int \chi(x,\xi)\, e^{i\psi(x,\xi)}\, d\mu_\phi(x)\Big|
\ \lesssim\ |\xi|^{-\beta}, \qquad |\xi|\to\infty,
\]
for every pair $\chi:I\times\R\to\R$ and $\psi:I\times\R\to\R$ with
\[
\|\chi(\cdot,\xi)\|_{C^\alpha}\lesssim |\xi|^{\varepsilon_\chi},\qquad
\|\log \partial_x\psi(\cdot,\xi)\|_{C^\alpha}\lesssim |\xi|^{1+\tilde\varepsilon_\psi},\qquad
\inf_x|\partial_x\psi(x,\xi)|\gtrsim |\xi|^{1-\varepsilon_\psi},
\]
for some sufficiently small
$\varepsilon_\chi,\varepsilon_\psi,\tilde\varepsilon_\psi>0$.
\end{thm}

To make Theorem~\ref{thm:mainCountable} usable in applications, we seek
checkable sufficient conditions for \emph{(QNL)}. In the $C^2$
uniformly hyperbolic setting, Fourier decay is typically obtained by
establishing spectral gaps for twisted transfer operators via
Dolgopyat-type estimates (e.g. \cite{LiNaudPan,SahlstenStevens,AlgomRHWang-Polynomial,BakerSahlsten,BakerKhalilSahlsten}). A dynamical formulation of the required
non-concentration is the \emph{Uniform Non-Integrability (UNI)}
condition introduced by Chernov and Dolgopyat \cite{Dolgopyat}. Defining cocycle
$\tau(g_\mathbf a(x)):=-\log|g_\mathbf a'(x)|$ and fluctuation
$X_\mathbf a(x,y):=\tau(g_\mathbf a(x))-\tau(g_\mathbf a(y))$, we say
that $\Phi$ satisfies \emph{(UNI)} if there exist $c_0>0$ and
infinite $N$ such that for every $c\in\mathcal A$ there exist
$\mathbf a,\mathbf b\in\Sigma^N$ with $b(\mathbf a)=b(\mathbf b)=c$ satisfying
\[
|X_\mathbf a(x,y)-X_\mathbf b(x,y)|\ge c_0|x-y|^\alpha
\quad\text{for all }x\ne y\in F\cap I_c.
\]
However, in genuinely $C^{1+\alpha}$ setting verifying UNI on the entire attractor becomes difficult leading to many difficulties e.g. in smooth dynamics \cite{TsujiiZhang,LeclercPreprint}. In
these cases one instead seeks a statistical form of non-concentration
along the equilibrium state. Letting $\delta=\dim_HF$ be the Hausdorff dimension and $\mu$ be the
equilibrium state for $\delta\tau$, we say that the IFS $\Phi$ satisfies the
\emph{Measure Non-Linearity (MNL)} condition if there exists
$C_{\mathrm{(MNL)}} > 0$ such that for all $N$ there are
$\mathbf a,\mathbf b\in\Sigma^N$ with
\[
\forall x_0\in F,\ \forall I\subset\R,\qquad
\big(X_\mathbf a(\cdot,x_0)-X_\mathbf b(\cdot,x_0)\big)_*\mu(I)
\ \le\ C_{\mathrm{(MNL)}}\,|I|.
\]
This replaces the uniform lower bound in \emph{(UNI)} by a
measure-theoretic non-concentration property along the symbolic tree,
in the spirit of the coupling of geometric and symbolic directions in
\cite{TsujiiZhang}.

Under UNI or MNL, we then have the needed QNL for Theorem \ref{thm:mainCountable}:

\begin{thm}\label{thm:QNL}
If infinite $C^{1+\alpha}$ IFS has
\emph{(UNI)}, or in finite branch case \emph{(MNL)}, then \emph{(QNL)} holds.
\end{thm}

Note here we allow arbitrary overlaps. The UNI part in finite branch case can be verified using spectral gaps for twisted transfer operators as in \cite{SahlstenStevens} and in overlapping case a co-cycle version of the Dolgopyat method
was needed and indeed developed in \cite{AlgomRHWang-Polynomial,BakerSahlsten} via average contraction of random Dolgopyat operators. Our strategy avoids these constructions and replaces it by the multiscale induction allowing us to access parabolic and $C^{1+\alpha}$ cases.

Let us now discuss applications of Theorem \ref{thm:mainCountable} and Theorem \ref{thm:QNL} reducing power Fourier decay to UNI or MNL type statements.

\subsection{Smoothness of effective conductance of Galton-Watson trees.} Let $(p_k)_{k\ge0}$ be a probability distribution on $\mathbb{N}\cup\{0\}$ and let $T$ be the random genealogical tree obtained from the associated \textit{Galton-Watson branching process} begun with one individual $\varnothing$, in which each vertex has $k$ children with probability $p_k$, independently of all other vertices. Adjoining a parent vertex $v$ to $\varnothing$ and connecting it to $\varnothing$ by an edge, we may view the resulting graph $T\cup\{v\}$ as an electrical network in which each edge has unit conductance. Writing $\Theta$ for the effective conductance of $T \cup \{\v\}$ from $v$ to infinity, we know that $\Theta$ equals to the probability of a simple random walk starting from $v$ and returning to $v$ \cite{Lyons_Peres_2017,lyons1995ergodic,curien2017harmonic}. Then the effective conductance is a random continued fraction $\Theta = \frac{1}{1+\frac{1}{X_1+\frac{1}{1+\cdots}}}$ for i.i.d. $X_1,X_2,\ldots$. In \cite{Lyons}, Lyons considered the special case where $X_i \in \{0,t\}$ i.i.d. with equal probabilities. Thus in this case $\Theta$ is obtained as the infinite random
composition $\Theta=\lim_{n \to \infty} f_{X_1}\circ f_{X_2}\circ\cdots f_{X_n}(0)$ for $f_0(x)=\frac{x}{1+x}$ and $f_t(x)=\frac{x+t}{1+x+t}$ so if we denote by $\nu_t$ the law of $\Theta$, it is an equilibrium state for the overlapping IFS $\{f_0,f_t\}$, which is parabolic as $f_0'(0) = 1$.

On the smoothness of $\nu_t$, Lyons \cite{Lyons} proved $\nu_t$ is singular for any $t > t_0 \approx 0.2689$ and conjectured $\nu_t$ should to be absolutely continuous for all small $t$. This was motivated by numerics and it being a nonlinear analogue of the Bernoulli convolution $\mu_\lambda$ (distribution of $\sum \pm \lambda^n$), where a similar conjecture is widely open despite much recent progress \cite{Shmerkin-AbsContBernoulliConv,Hochman-Annals,Varj__2019} and for higher dimensional analogues of Bernoulli convolutions \cite{Lindenstrauss_2016}. As with Bernoulli convolutions, absolute continuity and density information of $\nu_t$ are known for almost every $t$ by transversality method \cite{simon2001invariant,solomyak1998non,solomyak2001q,barany2022typical}. In \cite{Shmerkin-AbsContBernoulliConv} Shmerkin related the absolute continuity of $\mu_\lambda$ to the power Fourier decay of Bernoulli convolutions associated to powers $\lambda^k$, but it remains a very challenging problem prove this, where the Erd\"os-Kahane method can be applied for almost every $\lambda$ (see e.g. \cite{Solom,Shmerkin-AbsContBernoulliConv}). 

By inducing the parabolic Lyons IFS \cite{mihailescu2016random}, as a consequence of Theorem \ref{thm:mainCountable} we have:

\begin{thm}\label{thm:lyons}
For all $0<t<1$, the Lyons conductance measure $\nu_t$ has power Fourier decay.
\end{thm}

Now Shmerkin's strategy \cite{Shmerkin-AbsContBernoulliConv} cannot be directly used with Theorem \ref{thm:lyons} to prove Lyons's conjecture due to the lack of convolution structure of $\nu_t$. However, as in the work of Lindenstrauss and Varj\'u on the higher dimensional setting \cite{Lindenstrauss_2016}, which is very different from the one dimensional case due to non-Abelian nature of $SO(d)$ instead a spectral gap assumption on the underlying transfer operator was crucial in the proof of absolute continuity via a power Fourier decay property of the self-similar measure. Given that our proof gives a non-concentration for the Lyons IFS typically needed for spectral gaps \cite{BakerSahlsten}, one may try to adapt more directly the spectral gaps of twisted transfer operators associated to $\nu_t$ as in \cite{Lindenstrauss_2016} to try to upgrade Theorem \ref{thm:lyons} to a proof of Lyons conjecture.

\subsection{Patterson--Sullivan measures on parabolic limit sets.} Fourier decay for Patterson--Sullivan measures associated to convex co-compact hyperbolic surfaces was established by Bourgain and Dyatlov \cite{BourgainDyatlov} as a tool to bound the resonance free regions for scattering resonances by proving improved the Fractal Uncertainty Principle bounds. Extension of this to convex co-compact subgroups of $\mathrm{PSL}(2,\C)$ was done in \cite{LiNaudPan} using twisted transfer operators and to higher dimensions in \cite{BKS} adapting an $L^2$-flattening approach of Khalil \cite{Khalil-Mixing}. We will now see that Theorem \ref{thm:mainCountable} allows us to expand these to geometrically finite case where there are parabolic elements in the group. 

Let $X=\Gamma\backslash\mathbb{H}$ be a non-compact hyperbolic surface with cusps, where $\Gamma<\mathrm{PSL}(2,\mathbb{R})$ is torsion-free and geometrically finite with parabolic elements acting on $\mathbb{H}\cup\dot{\mathbb{R}}$ (the extended real line, identified with $\mathbb{S}^1$ in the disk model). Denote the limit set by $\Lambda_\Gamma\subset\dot{\mathbb{R}}$. Let $\mu_\Gamma$ be the Patterson–Sullivan measure centered at $i\in\mathbb{H}$. If $\delta>0$ is the critical exponent of $\Gamma$, then $\mu_\Gamma$ is $\delta$–conformal and satisfies, for all bounded Borel $f:\dot{\mathbb{R}}\to\mathbb{C}$ and all $\gamma\in\Gamma$,
\[
\int_{\Lambda_\Gamma} f(x)\, d\mu_\Gamma(x)
= \int_{\Lambda_\Gamma} f(\gamma(x))\, |\gamma'(x)|_{\mathbb{S}^1}^{\,\delta}\, d\mu_\Gamma(x),
\]
where $|\gamma'(x)|_{\mathbb{S}^1}=\frac{1+x^2}{1+\gamma(x)^2}\,|\gamma'(x)|$.

\begin{thm}\label{thm:ps}
Let $\Gamma<\mathrm{PSL}(2,\mathbb{R})$ be torsion-free, geometrically finite, with parabolic elements. Then there exists $\beta>0$ such that
\[
\Big|\int \chi(x,\xi)\, e^{i\psi(x,\xi)}\, d\mu_\Gamma(x)\Big|\ \lesssim\ |\xi|^{-\beta},\qquad |\xi|\to\infty,
\]
for any $\chi:I\times\mathbb{R}\to\mathbb{R}$ and $\psi:I\times\mathbb{R}\to\mathbb{R}$ with
\[
\|\chi(\cdot,\xi)\|_{C^\alpha}\lesssim |\xi|^{\varepsilon_\chi},\quad
\|\log\partial_x\psi(\cdot,\xi)\|_{C^\alpha}\lesssim |\xi|^{\tilde\varepsilon_\psi},\quad
\inf_x|\partial_x\psi(x,\xi)|\gtrsim |\xi|^{1-\varepsilon_\psi},
\]
for some sufficiently small $\varepsilon_\chi,\varepsilon_\psi,\tilde\varepsilon_\psi>0$.
\end{thm}

Bourgain–Dyatlov \cite{BourgainDyatlov} used Fourier decay to prove a Fractal Uncertainty Principle (FUP) for convex co-compact groups ($\delta_\Gamma<1/2$), with consequences for the distribution of scattering resonances. In the parabolic case one has $\delta_\Gamma\ge 1/2$, and the strategy in \cite{BourgainDyatlov} does not apply directly. It would be interesting to investigate whether Theorem~\ref{thm:ps} has implications for the meromorphic continuation of the resolvent on geometrically finite manifolds with cusps (cf.\ \cite{GuillarmouMazzeo}).

\subsection{Manneville--Pommeau and Lorenz maps.} 

Our axioms in Theorem~\ref{thm:mainCountable} were designed with Young tower models in mind, obtained by inducing parabolic systems such as Lorenz maps (arising from the Lorenz flow) \cite{Lorenz,LorenzYoungCoding} and, in higher dimensions, the Chernov–Markarian–Zhang structure for slowly mixing billiards (e.g. the Bunimovich stadium) \cite{ChernovZhang,Markarian,BMT}. We give two applications: Manneville–Pommeau maps and Lorenz maps; the assumptions can be verified in broader induced settings.

\begin{itemize}
\item[(a)] \textit{Manneville–Pommeau maps.}
We use the Liverani–Saussol–Vaienti parametrization \cite{LSV}:
\[
T_\alpha(x)=
\begin{cases}
x\big(1+2^\alpha x^\alpha\big), & x<\tfrac12,\\
2x-1,& x\ge \tfrac12,
\end{cases}
\]
with $T'_\alpha(0)=1$ (parabolic fixed point). Let $f_0,f_1$ be the inverse branches (with $f_1$ the left, nonlinear one, and $f_0(x)=(x+1)/2$), set $x_0=1$, $x_1=\tfrac12$, $x_{n+1}=f_1(x_n)$, and $I=[\tfrac12,1]$, $I_n=[x_{n+1},x_n)$. This gives a Markov partition of $[0,1]=T_\alpha(I)$ with $I_{n-1}=T_\alpha(I_n)$. Let $R:I\to\mathbb{N}$ be the first return time to $I$, and define the induced map $T_A:I\to I$ by $T_A(x)=T_\alpha^{R(x)}(x)$. Then $T_A$ is an infinitely branched expanding Markov map with inverse branches $g_a:I\to A_a$, $a\ge0$ (where $A_0=[3/4,1]$ and $A_a:=(T_\alpha^{-1}I_a)\cap I$).

\item[(b)] \textit{Lorenz maps.}
Consider the map $f:J\setminus\{0\}\to J$, $J=[-1/2,1/2]$, defined by \cite[§3]{AV}
\[
f(x)=
\begin{cases}
-a(-x)^\alpha+\tfrac12,& x<0,\\
a x^\alpha-\tfrac12,& x>0,
\end{cases}
\qquad \alpha:=-\lambda_3/\lambda_1\in(0,1),\ 0<a\le 2^\alpha,
\]
arising from the Poincaré map of the geometric Lorenz flow ($\lambda_2<\lambda_3<0<-\lambda_3<\lambda_1$). The derivative of $f$ blows up near $0$, and inverse branches have zero derivative there. Diaz-Ordaz \cite[Thm.~2]{LorenzYoungCoding} constructs an induced expanding Markov map $T_A:I\to I$ (for some neighborhood $I\subset[-1,1]$ of $0$) with bounded distortion and first return time $R:I\to\mathbb{N}$ having exponential tails, under uniform contraction (1) and bounded distortion (2).
\end{itemize}

Given a base potential $\phi_0$ for the original system (e.g.\ $\phi_0=-\delta\log|T'_\alpha|$ for Manneville–Pommeau or $\phi_0=-\delta\log|f'|$ for Lorenz), define the induced potential on $I$ by $\phi(x):=S_{R(x)}\phi_0(x)$. For our results we require \emph{summable variations} for $\phi$:
$
\sum_{n\ge1} V_n(\phi)<\infty$, where $V_n(\phi):=\sup_{\mathbf{a}\in\Sigma^n}\ \sup_{x,y\in I_\mathbf{a}}|\phi(x)-\phi(y)|,$
which ensures that the Sarig theory applies to the countable Markov map $T_A$ \cite{Sarig1,Sarig2}. A sufficient condition is $\sum_n n\, V_n(\phi_0)<\infty$ \cite[§2.4]{BruinTodd}. Let $\mu_\phi$ be a Gibbs measure for the induced map $T_A$. Define the projected measure on the parabolic system by
\[
\nu_\phi(E):=\sum_{a=0}^\infty\ \sum_{k=0}^a \mu_\phi\big(T_\alpha^{-k}(E)\cap A_a\big),\qquad E\subset I,
\]
and normalize by $\int R\, d\mu_\phi$ when finite. This is the standard way (e.g.\ for SRB/conformal measures) to transfer statistical properties from the Young tower to the base system; see \cite{BruinTodd,Todd}.

\begin{thm}\label{thm:FourierMP}
Assume $\phi:I\to\mathbb{R}$ has summable variations and, for some $\gamma_0>0$,
\[
\sum_{a\in\mathcal{A}}\ \sup_{x\in I} w_a(x)\, |g_a'(x)|^{-\gamma_0}<\infty,
\]
where $w_a(x)=e^{\phi(g_a(x))}$. Let $\mu_\phi$ be the Gibbs measure for $T_A$ arising from the Manneville–Pommeau map $T_\alpha$ or Lorenz map $f$. Then the projected parabolic measure $\nu_\phi$ for $T_\alpha$ (or Lorenz map $f$ respectively) satisfies, for some $\beta>0$,
\[
\Big|\int \chi(x)\, e^{i\xi \psi(x)}\, d\nu_\phi(x)\Big|\ \lesssim\ |\xi|^{-\beta},
\]
for any $\chi,\psi:I\to\mathbb{R}$ with $\operatorname{supp}\chi$ disjoint from the parabolic fixed point, $\|\chi\|_{C^\alpha}<\infty$, $\|\psi\|_{C^{1+\alpha}}<\infty$, and $\inf|\psi'|>0$.
\end{thm}

For Manneville–Pommeau maps $T_\alpha$ with $0<\delta<1$ and $0<\alpha<1$, the choice $\phi_0=-\delta\log T'_\alpha$ yields an induced potential $\phi=S_R\phi_0$ with summable variations and a light-tail condition for any $\gamma_0>0$ (see \cite{BruinTodd,Todd}); hence Theorem~\ref{thm:FourierMP} gives that the $\delta$–conformal measure of $T_\alpha$ has positive Fourier dimension. The same argument applies to Lorenz maps once (UNI) is verified for the induced model (see Section~\ref{sec:para} for examples of parameters, where this holds).

\begin{remark}
The assumption that $\operatorname{supp}\chi$ avoids the parabolic fixed point can be relaxed to a local decay condition; we keep the support condition for clarity of statement.
\end{remark}

An intriguing direction is to relax the light-tail assumption. Because Theorem~\ref{thm:mainCountable} is formulated in terms of exponential moment conditions, it is well-suited to being adapted to heavier tails, where one might expect sub-polynomial (e.g.\ polylogarithmic) Fourier decay. In the context of Furstenberg measures, such log–log decay has already been demonstrated in \cite{Dihn}. Extending our framework to this regime would provide a natural bridge between strong polynomial decay and much weaker forms of Fourier decay, potentially revealing new structures in dynamical systems with weaker hyperbolicity.

\subsection{$C^{1+\alpha}$ iterated function systems.} Finally, many dynamically defined measures of interest arise from systems that admit only $C^{1+\alpha}$ regularity, for instance due to rough holonomies in codings of hyperbolic diffeomorphisms or in nonlinear cocycle models such as circle extensions and the Fibonacci Hamiltonian. In such settings the associated temporal distance functions are typically non-smooth, and the classical Fourier decay mechanisms based on Dolgopyat-type spectral gap estimates are not available.

Our framework applies in this genuinely low-regularity regime. In particular, we construct examples of $C^{1+\alpha}$ iterated function systems whose conformal measures admit power Fourier decay despite not being $C^1$-conjugate to linear models, that is, there is no $C^1$ map $h$ such that $\{h f h^{-1} : f\in \Phi\}$ consists of maps with constant derivative.

\begin{thm}\label{thm:C1alpha}
There exist $C^{1+\alpha}$ iterated function systems with finitely many branches whose attractors support equilibrium states with power Fourier decay, and which are not $C^1$-conjugate to linear IFS. 
\end{thm}

 In particular, these examples provide conformal measures of positive Fourier dimension arising from genuinely $C^{1+\alpha}$ dynamics. The construction proceeds by coupling the symbolic and geometric oscillations of the distortion function through a novel fixed-point argument in the moduli space of IFS parametrised by equilibrium states. This allows us to verify the multiscale nonlinearity condition required by Theorem~\ref{thm:mainCountable} without relying on higher regularity or smooth symbolic holonomies.

The condition on MNL is quite restrictive as it requires oscillations at right dimension on the fractal. This suggests the possibility of developing a $C^{1+}$ theory of Fourier decay, spectral gaps and fractal uncertainty principles in settings where the distortion functions exhibit only autosimilar or fractal-type oscillations (e.g. Brownian/Weierstrass-type as in \cite{TsujiiZhang,LeclercPreprint}), rather than smooth curvature.

 Finally, let us remark on the opposite case when the system is conjugated to a linear IFS. The analytic and $C^2$ settings (e.g.\ \cite{AlgomRHWang-Polynomial,BakerSahlsten,BakerBanaji,AlgomChangWuWu,BanajiYu,AlgO}), power Fourier decay is typically classified through the presence of non-linear branches with non-vanishing second derivative or non-flatness condition of the map \cite{AlgomChangWuWu}. In our rough $C^{1+\alpha}$ framework, \emph{autosimilarity} of the distortion function might play the role of the second derivative giving the necessary non-flatness property needed in \cite{AlgomChangWuWu}. Establishing such a replacement would move the theory toward a genuine characterisation of power Fourier decay at the $C^{1+\alpha}$ regularity threshold.

\subsection*{Organisation of the article} In Section \ref{sec:moments} we first give the preliminary thermodynamic notations and derive the needed moment bounds and upper regularity using the light tail condition where the uniform contraction, bounded distortions and light tail conditions are crucially used. In Section \ref{sec:sumprod} we show how the moments arise from the discretised sum-product bounds if we allow the multiplicative random variables have unbounded support. In Section \ref{sec:reduction} we then begin the proof of Theorem \ref{thm:mainCountable} by reducing the Fourier transform of the equilibrium state $\mu$ to exponential sums of multiplicative random variables with moment bounds. To then apply the moment bounds and upper regularity derived in Section \ref{sec:moments} and the sum-product bound from Section \ref{sec:sumprod}, in Section \ref{sec:nonconc} we prove the needed non-concentration bound from (QNL). In section \ref{sec:nonconcentration}, we prove that (QNL) holds under UNI and MNL conditions. In Section \ref{sec:lyons} we check conditions for Theorem \ref{thm:lyons}. In section \ref{sec:ps} we prove Theorem \ref{thm:ps}. In section \ref{sec:para} we prove Theorem \ref{thm:FourierMP}. In section \ref{sec:IFSconstruct} we construct a $C^{1+}$ IFS satisfying (MNL) and prove Theorem \ref{thm:C1alpha}. 

\section{Moment bounds and upper regularity} \label{sec:moments}

In this section we introduce some supplementary notations and prove two preliminary statements that will prove useful to prove Theorem \ref{thm:mainCountable}. The first is a precise bound coming from the properties of the \emph{pressure}. The second is an upper-regularity estimate for $\mu_\phi$ on this setting. A key reference to the thermodynamical formalism for countable Markov shift we use in the symbolic level is given by Sarig's papers \cite{Sarig1,Sarig2}.

\subsection{Symbolic coding, Gibbs measures and transfer operators} 

Recall the notation from Section \ref{sec:mainthm}, where $
\Sigma^\N:=\{(a_1,a_2,\ldots)\in\mathcal{A}^\N:\ A(a_i,a_{i+1})=1\ \text{for all }i\}
$ is the countable Markov shift
determined by a matrix $A \in \{0,1\}^{\mathcal{A}\times\mathcal{A}}$ and $\Sigma^n$ is the set of admissible words of length $n$ and by $\Sigma^*=\bigcup_{n\ge1}\Sigma^n$. Then a word $\mathbf{a} := a_1\dots a_n \in \mathcal{A}^n$ satisfying $A(a_i,a_{i+1})=1$ for all $i$ is called \textit{admissible} and cylinder $[\a] \subset \Sigma^\N$ is the subtree of infinite words with initial $\a$. We assume there is a collection $\{I_a\}_{a\in\mathcal{A}}$ of possibly overlapping open intervals in $[0,1]$ and for each admissible pair $ab\in\Sigma^2$ we are given a $C^{1+\alpha}$ contraction
$
g_{ab}: I_b \to I_a,$
 which form an \textit{iterated function system} $\Phi = \{g_{ab} : ab \in \cA^2 \ \text{admissible}\}$.
For an admissible word $\mathbf{a}=a_1\cdots a_n\in\Sigma^n$ we define $
g_\mathbf{a}:=g_{a_1a_2}\circ\cdots\circ g_{a_{n-1}a_n}.$ Recall that the notation $\mathbf{a}\to x$ simply means that $g_\a$ is defined at $x$, i.e. $x\in I_{b\su{\a}}$ with $b\su{\a}$ denoting the last letter of $\a$.

We suppose the IFS $\Phi= \{g_{ab} : ab \in \cA^2 \ \text{admissible}\}$ satisfies the conditions (1)-(3) from Section \ref{sec:mainthm}, though in this section we only need uniform $\kappa_+$ contraction (1) and bounded $(C,\alpha)$-distortions (2). The uniform contraction of the maps $g_{ab}$ guarantees that there exists a unique attractor $F \subset I$ associated to $\Phi$, that is, for $\mathbf{a} \in \Sigma^{n+1}=a_0 \dots a_n$ the map $g_{\mathbf{a}} := g_{a_0 a_1} \circ \dots \circ g_{a_{n-1} a_n} : I_{a_n} \rightarrow I_{a_0} $, and the interval $I_{\mathbf{a}} := g_{\mathbf{a}}(I_{a_n})$, then $F$ is the compact set given by $F = \bigcap_{n \in \N}\bigcup_{\a \in \Sigma^n}\overline{I_\a}$. This also gives us a projection $\pi : \Sigma^\N \to F$.

Since the IFS $\Phi$ we work on in Theorem \ref{thm:mainCountable} has overlaps, the map $\pi$ can be very badly away from being injective, we will need to consider thermodynamical formalism and the associated measures $\mu$ as projections of measures on the symbolic space $\Sigma^\N$. Formally, if $\phi : F \to \R$ is a potential, we can consider $\phi$ also as a map $\phi \circ \pi : \Sigma^{\mathbb{N}} \to \R$. Then if we let $\nu_\phi$ be the (symbolic) Gibbs measure associated to $\phi$ on $\Sigma^\N$, we mean that if is a probability measure satisfying
\begin{align*}
 \int_{\Sigma^\N} f(\a) \, d\nu_\phi(\a) = \sum_{ab \in \Sigma^2} \int_{[b]} e^{\phi(g_{ab}(\a))} f(g_{ab}(\a)) \, d\nu_\phi(\a)
\end{align*}
for all continuous $f : \Sigma^\N \to \C$. Here $b(\mathbf{a})$ is the last letter of a word. Thus this implies the projected Gibbs measure $\mu_\phi := \pi\mu_\phi$ on the attractor $F$ satisfies the $\phi$-conformality condition
\begin{align*}
 \int_F f(x) \, d\mu_\phi(x) = \sum_{ab \in \Sigma^2} \int_{I_b} e^{\phi(g_{ab}(x))} f(g_{ab}(x)) \, d\mu_\phi(x)
\end{align*}
for all continuous $f : F \to \C$ and note by the definition of $b(\a)$ we can write $g_\mathbf{a}:I_{b(\mathbf{a})} \rightarrow I_{\mathbf{a}}$. So from now on when we write $\mu_\phi$ is a Gibbs measure, we assume $\mu_\phi$ comes as a projection of a symbolic measure $\nu_\phi$ defined on $\Sigma^\N$. We also assume $\mu_\phi$ satisfies the light $(C_{\phi,\gamma_0},\gamma_0)$-tail condition (4) from Section \ref{sec:mainthm}. 

Now the $\phi$-conformality condition can be iterated to yields the formula:
$$\forall f \in C^0(F,\mathbb{C}), \ \int_F f d\mu = \sum_{\mathbf{a} \in \Sigma^{n+1}} \int_{b(\mathbf{a})} e^{S_n \phi(g_{\mathbf{a}}(x))} f(g_{\mathbf{a}}(x)) d\mu(x),$$
where $S_n \phi(g_\a(x)) := \sum_{k=0}^{n-1} \phi( g_{\a|k}(x) )$ is a Birkhoff sum. It is then convenient to denote $w_\mathbf{a}(x) := e^{S_n \phi(g_{\mathbf{a}}(x))}$. The invariance formula can then be rewritten as
$$ \int f d\mu = \int \mathcal{L}_\varphi^n(f) d\mu, $$
where $\mathcal{L}_\phi : f \mapsto \mathcal{L}_\phi f$ is the \emph{transfer operator} associated to $\phi$, given by the formula
$$ \mathcal{L}_\phi(f)(x) := \underset{I_b \ni x}{\sum_{ab \in \Sigma^2}} w_{ab}(x) f(g_{ab}(x)). $$
Iterating yields $$ \forall x \in I_b, \ \mathcal{L}_\phi^n(f)(x) := \underset{b(\mathbf{a})=b}{\sum_{\mathbf{a} \in \Sigma^{n+1}}} w_{\mathbf{a}}(x) f(g_\mathbf{a}(x)). $$
Notice that the normalisation assumption on $\phi$ implies that $\mathcal{L}_\phi^n(1)=1$. Notice further that the $\phi$-conformality assumption (and summable variations of $\phi$) implies the following \emph{Gibbs estimate}: there exists a \textit{Gibbs constant} $C_\phi \geq 1$ such that
$$ \forall \mathbf{a} \in \Sigma^n, \forall x \in I_{b(\mathbf{a})}, \ C_\phi^{-1} w_{\mathbf{a}}(x) \leq p_\a := w_\a(x_{b(\a)}) \leq C_\phi w_{\mathbf{a}}(x).$$
Troughout the text well will use the notation $a_n \lesssim b_n$ if there exists some constant $C \geq 1$ such that $a_n \leq C b_n$. We will write $a_n \simeq b_n$ if $a_n \lesssim b_n \lesssim a_n$. Gibbs estimates are then rewritten $w_\mathbf{a}(x) \simeq p_\a$ for all $x$.

\subsection{Moment bounds from the tail condition}

By Birkhoff ergodic theorem, we know that the natural scale for the contractions $|g_{\mathbf{a}}'(x)|$ is given by $e^{-\lambda n}$ for $\mu$-almost every $x$, where $\lambda$ is the \textit{Lyapunov exponent} of $\mu$, is defined by $$\lambda := -\int \tau(x) \, d\mu(x),$$ where $\tau : I \to \R$ is the geometric potential $\tau(g_\a(x)) = \ln |g_\a'(x)|$, $x \in I$, $\a \to x$. Proposition 2.1 quantifies deviation from this scale of contraction using properties of the \textit{Gurevich pressure} $P(\phi + t(\tau - \lambda))$, which is defined for a potential $\psi$ as
\begin{align*}
 P(\psi) := \lim_{n \to \infty} \frac{1}{n} \log\sum_{\a \in \Sigma^n}\sup_{x \in I_{b(\mathbf{a})}} e^{S_n \psi(g_\a(x))}.
\end{align*}
Notice that since $\mathcal{L}_\phi^n(1)=1$, we find that our normalized potential $\phi$ satify $P(\phi)=0$. It is also known that $\frac{d}{dt} P(\phi + t \tau)_{|t=0} = \int \tau d\mu$. The differentiability properties of the pressure allows us to bound the spreading of the random variable $e^{\lambda n}|g_{\mathbf{a}}'(x)|$ in the following way:
\begin{Proposition}\label{prop:moments}
Suppose $\mu$ satisfies the $(C_{\phi,\gamma_0},\gamma_0)$-light tail condition. Then there exists $t_{\gamma_0} \in (0,1)$, $\eps_{\gamma_0} > 0$ such that:
\begin{align*}
 \forall t \in \su{-t_{\gamma_0}, t_{\gamma_0}}, \exists n_0(t), \forall n \geq n_0(t), \forall x \in I, \ \sum_{\a \in \Sigma^n\atop \a \rightarrow x} w_{\mathbf{a}}(x)
(e^{\lambda n} |g_{\mathbf{a}}'(x)|)^{t} \leq 2C_\phi e^{\eps_{\gamma_0} n t^2},
\end{align*}
where
$$\eps_{\gamma_0} := \sup_{|t| \leq t_{\gamma_0} } \frac{|P(\phi + t (\tau-\lambda))|}{t^2} < \infty.$$
\end{Proposition}

 The optimal choice of $\eps_{\gamma_0}$ comes from the second order term
$$\frac{d^2}{dt^2} P(\phi + t(\tau - \lambda))|_{t = 0} = \lim_{n \to \infty} \frac{1}{n} \int (S_n \tau)^2 \, d\mu$$
of the Taylor approximation for $t \mapsto P(\phi+ t(\tau-\lambda))$ near $0$. The formula given by the limit of Birkhoff sums of $\tau$ is proved e.g. in \cite{ParryPollicott}.

The proof of Proposition \ref{prop:moments} can be deduced from the proof leading to large deviation principle, e.g. \cite{JorSah} and the references there-in. However, as our result is stated quantitatively in terms of moments, we included a direct proof using Sarig's results from \cite{Sarig1} for the countable Markov shift for completeness.

\begin{proof}[Proof of Proposition \ref{prop:moments} ]
Let $-\gamma_0 \leq t < 0$. Then if write $p_a := w_a(x_b(a))$ and recalling that $\kappa_-(a) \leq |g_a'(x)| \leq \kappa_+$ for all $x \in I_b$, where we have supressed the prefix notation of $b$ for admissibility, then by the light $(C_{\phi,\gamma_0},\gamma_0)$-tail property and using the Gibbs constant $C_\phi$ we have
\begin{align*}
 & \sum_{\mathbf{a} \in \Sigma^n \atop \a \rightarrow x} w_{\mathbf{a}}(x)
|e^{\lambda n} g_{\mathbf{a}}'(x)|^{t} \leq C_\phi \Big( \sum_{a \in \cA} p_a (e^{\lambda}\kappa_-(a))^{t}\Big)^{n} \\
& \leq C_\phi e^{t\lambda n} \Big( \sum_{a \in \cA} p_a \kappa_-(a)^{-t}\Big)^{n} \leq C_\phi e^{t \lambda n} C_{\phi,\gamma_0}^n = C_\phi e^{(t\lambda + \log C_{\phi,\gamma_0}) n}
\end{align*}
Moreover, if $t \geq 0$, then we can simply bound
$$ \sum_{\mathbf{a} \in \Sigma^n} w_{\mathbf{a}}(x)
|e^{\lambda n} g_{\mathbf{a}}'(x)|^{t} \leq C_\phi \Big( \sum_{a \in \cA} p_a (e^{\lambda}\kappa_+)^{t}\Big)^{n} = C_\phi (e^\lambda \kappa_+)^{tn} = C_\phi e^{t(\lambda + \log \kappa_+)n},$$
where note that by uniform $\kappa_+$-contraction the Lyapunov exponent of $\mu$ satisfies $
 \lambda \geq - \int \log \kappa_+ \, d\mu(x) = -\log \kappa_+ > 1$ as $\kappa_+ < 1$, so the exponent on is non-negative. Since
\begin{align*}
 \frac{1}{n}\ln \sum_{\mathbf{a} \in \Sigma^n \atop \a \to x} w_{\mathbf{a}}(x)
|e^{\lambda n} g_{\mathbf{a}}'(x)|^{t} \underset{n \rightarrow \infty}{\longrightarrow} P(\varphi + t(\tau-\lambda))
\end{align*}
this proves that for all $t \in [-\gamma_0,\infty)$ the pressure is finite:
$$P(\varphi + t(\tau-\lambda)) \leq \begin{cases}t\lambda + \log C_{\phi,\gamma_0}, & -\gamma_0 \leq t < 0; \\
t(\lambda + \log \kappa_+), & t \geq 0. \\
\end{cases}$$
Since $\tau$ is $\alpha$-H\"older by the bounded distortion property and as the pressure $P(\varphi + t(\tau-\lambda))$ is finite for all $|t| \leq \gamma_0$, we know by \cite[Theorem 4]{Sarig1} that there exists $0 < t_{\gamma_0}' < \gamma_0$ such that the function $f:t \in (-t_{\gamma_0}',t_{\gamma_0}') \mapsto P(\varphi + t (\tau - \lambda))$ is real analytic in $(-t_{\gamma_0}',t_{\gamma_0}')$ and satisfies that $f(0)=P(\varphi)=0$ and $f'(0) = \int_{I} (\tau-\lambda) d\mu = 0$. Taylor estimate near $0$ then gives some $t_{\gamma_0}'' \in (0,1)$ such that
$|f(t)| \leq \tfrac{3}{4}f''(0) t^2$
for $|t| \leq t_{\gamma_0}''$. Thus choosing $0 < t_{\gamma_0} < t_{\gamma_0}''$ small enough, it follows that, for all $|t| \leq t_{\gamma_0}$, there exists $n_0(t) \geq 1$ such that
\begin{align*}
 \forall n \geq n_0(t), \ \sum_{\mathbf{a} \in \Sigma^n} w_{\mathbf{a}}(x)
|e^{\lambda n} g_{\mathbf{a}}'(x)|^{t} \leq e^{f''(0) n t^2}.
\end{align*}
and we get the claim for $\eps_{\gamma_0} = f''(0)$ as $\eps_{\gamma_0} t_{\gamma_0}^2 \geq |f(t_{\gamma_0})|$. We get a bound uniform in $x$ because of Gibbs estimates and bounded variations.
\end{proof}

\subsection{Upper regularity of the Gibbs measure from the tail}

We now prove the Frostman condition under the hypothesis that has the light-tail condition. I We note that if one would relax the light tail assumption to heavier tail assumptions, the upper regularity of $\mu_\phi$ is likely to fail. The result here is similar to that for overlapping self-similar measures by Feng and Lau \cite{FengLau} but generalised to the conformal setting and with infinite branches.

\begin{Proposition}[Upper regularity]\label{prop:upperreg}
Suppose that the attractor to the IFS $F$ is not a single point. Then assume one of the following:
\begin{itemize}
    \item the IFS has no overlaps, or
    \item the IFS has overlaps but the potential $\phi:\Sigma^* \rightarrow \mathbb{R}$ is locally constant in the sense that $\omega_a(x)=p_a$, with $\sum_{a} p_a = 1$
\end{itemize} 
Then exists $C_{reg} \geq 1$, $s_\mu \in (0,1)$ such that:
\begin{align*}
 \forall I \text{ interval}, \ \mu(I) \leq C_{reg} |I|^{s_\mu}.
\end{align*}
\end{Proposition}

We note that all the examples we apply Theorem \ref{thm:mainCountable} satisfy either of the assumptions (e.g. the induced overlapping infinite IFS associated to the Lyons IFS in Theorem \ref{thm:lyons} satisfies the locally constancy).

To prove Proposition \ref{prop:upperreg}, we first need a folklore result about the absence of atoms of equilibrium states under the assumptions in Proposition \ref{prop:upperreg}:

\begin{lem}\label{lma:noatoms}
Suppose that $F$ is not a single point. \begin{itemize}
\item If the IFS have no overlaps, then equilibrium states have no atoms.
\item If the IFS have overlaps, and if the potential $\phi:\Sigma^* \rightarrow \mathbb{R}$ is locally constant in the sense that $\omega_a(x)=p_a$, with $\sum_{a} p_a = 1$, then $\mu_\phi$ have no atoms.
\end{itemize}
\end{lem}

\begin{proof}
Suppose that $\mu_\phi$ have atoms. Consider a maximal atom $x_0 \in F$ of mass $M>0$. Then, for all $n$:
$$ M = \mu(x_0) = \int_F \mathbf{1}_{x_0} d\mu(x) = \int_F \mathcal{L}_\varphi^n(\mathbf{1}_{x_0}) d\mu  = \sum_{\mathbf{a} \in \Sigma^n} \int_F \omega_{\mathbf{a}}(y) \mathbf{1}_{x_0}(g_\mathbf{a}(y)) d\mu(y) $$
$$ = \underset{I_\mathbf{a} \ni x_0}{\sum_{\mathbf{a} \in \Sigma^n}} \int_F \omega_{\mathbf{a}}(y) \mathbf{1}_{g_\mathbf{a}^{-1}x_0}(y) d\mu(y) = \underset{I_\mathbf{a} \ni x_0}{\sum_{\mathbf{a} \in \Sigma^n}} \omega_{\mathbf{a}}(g_\mathbf{a}(x_0)) \mu({g_\mathbf{a}^{-1}x_0}) \leq M \underset{I_\mathbf{a} \ni x_0}{\sum_{\mathbf{a} \in \Sigma^n}} \omega_{\mathbf{a}}(g_\mathbf{a}(x_0)) \mu({g_\mathbf{a}^{-1}x_0}).$$
Hence, we get the key bound: $$ 1 \leq \underset{I_\mathbf{a} \ni x_0}{\sum_{\mathbf{a} \in \Sigma^n}} \omega_{\mathbf{a}}(g_\mathbf{a}(x_0)) .$$
Notice that $\omega_\mathbf{a}(x) < e^{-\varepsilon |a|}$ for some $\varepsilon$, hence this inequality can not hold if the IFS doesn't have overlaps, thus proving our first point. In the case where $F$ have overlaps, if we further suppose that $\omega_\mathbf{a}(y)$ are constants, we find:
$$1 \leq \underset{I_\mathbf{a} \ni x_0}{\sum_{\mathbf{a} \in \Sigma^n}} \omega_{\mathbf{a}} \leq {\sum_{\mathbf{a} \in \Sigma^n}} \omega_{\mathbf{a}} =1 $$
and hence $\forall n, \forall \mathbf{a} \in \Sigma^n, \ I_\mathbf{a} \ni x_0$. The fact that $F = \bigcap_{n} \bigcup_{\mathbf{a} \in \Sigma^n} \overline{I_\mathbf{a}}$ and that $|I_{\mathbf{a}}| \leq e^{-\varepsilon n}$ for some $\varepsilon>0$ implies that $F=\{x_0\}$.
\end{proof}

We now apply Lemma \ref{lma:noatoms} to get upper regularity via an induction-by-scales strategy inspired by Feng and Lau \cite{FengLau}. Firstly we need the following definition and a technical lemma.

\begin{definition}
Define $$\text{Reg}(F) := \{ h:K \rightarrow \mathbb{R}_+^* \ | \ \forall x,y \in F, \forall \mathbf{b} \in \Sigma^*, \omega_\b(x) h(g_\mathbf{b}(x)) \leq C_{\phi} \omega_\b(y) h(g_\mathbf{b}(y)) \},$$
where $C_\phi$ is the Gibbs constant. Notice that $1 \in \text{Reg}(F)$ by Gibbs estimates.
\end{definition}

\begin{lem}\label{lma:ursep}
There exists $\sigma_{sep}>0$ and $N_{sep} \geq 2$ such that, for all $h \in \text{Reg}(F)$, we have, for any interval $I$ of size $|I| \leq \sigma_{sep}$:
$$\forall n \geq N, \forall y \in F, \ \mathcal{L}_{\phi}^n(h \mathbf{1}_I)(y) \leq \frac{h(y)}{4 C_{\phi}} .$$
\end{lem}

\begin{proof}
Since $\mu$ has no atoms and is compactly supported, its cumulative distribution function is uniformly continuous. It follows that there exists a function $\varepsilon:[0,\infty) \rightarrow [0,\infty)$ with $\varepsilon(0)=0$, continuous at zero, such that $ \mu(I) \leq \varepsilon(|I|)$. Define $\sigma_{sep}>0$ so small $\varepsilon(2\sigma_{sep}) \leq \frac{1}{4 C_\phi^2}$. Now choose $h \in \text{Reg}(F)$, $|I| \leq \sigma_{sep}$. We can then bound:
$$ \mathcal{L}_{\phi}^n(h \mathbf{1}_{I}) \leq C_\phi h(y) \mathcal{L}_\phi^n(\mathbf{1}_I) \leq C_\phi h(y) \mathcal{L}_\phi^n(\chi).$$
Where $\chi$ is smooth, takes value in $[0,1]$ and is supported in an interval of size $\sigma_{sep}$. By Perron-Frobenius-Ruelle theorem \cite{baladi2000positive}, we find, for $n \geq N$ where $N$ is some constant that only depends on the $IFS$: $\mathcal{L}_\phi^n(\chi) \leq \frac{1}{8C_\phi^2} + \int \chi d\mu \leq \frac{1}{4 C_\phi^2}$. The desired bound follows.
\end{proof}

\begin{lem}\label{lem:uptr}
Define the $N$-th exponential moment function by $\mathcal{M}_N(t) := \sup_{x,y} \sum_{\mathbf{\a} \in \Sigma^N} \omega_\a(x) g_{\mathbf{a}}'(y)^{-t}$, which is finite for small enough $t>0$ by hypothesis, and satisfies $\mathcal{M}_N(0)=1$. Choose $s_\mu \in (0,1)$ such that $\mathcal{M}_{N_{sep}}(s_\mu) \leq 2$. Define further $C_{reg} := 2 C_\phi \sigma_{sep}^{-1}$ from Lemma \ref{lma:ursep}. Then for all $h \in \text{Reg}(F)$, for all interval $I$, we have:
$$ \forall y, \ \mathcal{L}_\phi^n(h \mathbf{1}_I)(y) \leq h(y) \cdot C_{reg}  \Big( |I|^{s_\mu} + \frac{1}{2^{n/N_{sep}}} \Big)  $$
\end{lem}

\begin{proof}
The proof is done by induction on $n \geq 0$. If $n \leq N_{sep}$, then the bound is trivial, since $\varphi$ is normalized, and since $h \in \text{Reg}(F)$:
$ \mathcal{L}_\phi^n(h \mathbf{1}_I) \leq C_\phi h(y) \mathcal{L}_\phi^n(\mathbf{1}_F) \leq 2 C_\phi \Big( |I|^{s_\mu} + \frac{1}{2^{n/N}} \Big)   $. \\
Suppose then that $n \geq N_{reg}+1$. If $\sigma > \sigma_{sep}/2$, the bound is also trivial. So suppose that $\sigma< \sigma_{sep}/2$. \\
We have:
$$ \mathcal{L}_\phi^n(h \mathbf{1}_I)(y) = \sum_{\mathbf{c} \in \Sigma^n} \omega_{\c}(y) h(g_\c y) \mathbf{1}_I(g_\c(y)) $$ $$ = \sum_{\mathbf{a} \in \Sigma^{n-N}} \sum_{\mathbf{b} \in \Sigma^N} \omega_{\mathbf{b}}(g_\a x) \omega_{\mathbf{a}}(x) h(g_{\b \a} y) \mathbf{1}_I(g_{\b \a} y). $$
Now notice that $\mathbf{1}_I(g_{\b \a} y) = \mathbf{1}_{I_\b}(g_{\a} y) \mathbf{1}_{\tilde{I}}(g_\b(y))$
where $\tilde{I}$ is an interval of size $|\tilde{I}|< \sigma_{reg}$, since $N_{reg}$ is large, and where $I_\mathbf{b}$ is an interval of size $|g_\mathbf{b}'|^{-1} |I|$. In particular:
$$ \mathcal{L}_\phi^n(h \mathbf{1}_I)(y) = \sum_{\b \in \Sigma^N} \sum_{\a \in \Sigma^{n-N}} \omega_\a(y) \omega_{\b}(g_\a(y)) h(g_\b g_\a(y)) \mathbf{1}_{I_\b}(g_\a(y)) \mathbf{1}_{\tilde{I}}(g_\b(y)) $$
$$ = \sum_{\b \in \Sigma^N} \mathbf{1}_{\tilde{I}}(g_\b(y)) \mathcal{L}^{n-N}_\phi (h_\b \mathbf{1}_{I_\b} )(y)  $$
$$ \leq \sum_{\b \in \Sigma^N} \mathbf{1}_{\tilde{I}}(g_\b(y)) \omega_\b(y) h(g_\b y) C_{reg} \Big( |I_\b|^{s_\mu} + 2^{-(n-N)/N} \Big) $$
$$ \leq  C_{reg} \Big( \sum_{\b} \omega_\b h \circ g_\b \mathbf{1}_{\tilde{I}} \circ g_\mathbf{b} \Big)^{1/2} \Big(  \sum_{\b} \omega_\b h \circ g_\b |g_\b'|^{-2s_\mu} \Big)^{1/2} |I|^{s_\mu} + C_{reg} h(y) 2^{-n/N}$$
$$ \leq C_{reg} \Big( \frac{h(y)^{1/2}}{2 C_\phi^{1/2}} (C_\phi h(y))^{1/2} M_{N_{sep}}(2 s_\mu) |I|^{s_\mu} + h(y) 2^{-n/N} \Big) $$
$$ \leq h(y) \cdot C_{reg}\Big( |I|^{s_\mu} + 2^{-n/N} \Big). $$

\end{proof}

We can now prove Proposition \ref{prop:upperreg} using Lemma \ref{lem:uptr}.

\begin{proof}[Proof of Proposition \ref{prop:upperreg}]
By Lemma \ref{lem:uptr}, we have, for all $n$:
$$ \mu(I) = \int_F \mathbf{1}_{I} d\mu =  \int_F \mathcal{L}_\varphi^n(\mathbf{1}_{I}) d\mu \leq C_{reg} \Big(|I|^{s_\mu} + 2^{-n/N_{sep}}\Big).$$
Taking the limit in $n \rightarrow \infty$ gives $\mu(I) \leq C_{reg} |I|^{s_\mu}$.
\end{proof}

\section{Discretised sum-product bounds under light tails} \label{sec:sumprod}

Our strategy is to reduce the Fourier decay bound to exponential sums of multiplicative random variables coming from products of the derivatives $g_a'(x)$ of the IFS. To bound this, as in Bourgain and Dyatlov in \cite{BourgainDyatlov}, one can reduce this to a multiscale non-concentration bound via the \emph{discretised sum-product phenomenon} \cite{Bourgain-SumProduct} coming from additive combinatorics. However, to gain the multiscale non-concentration, previous works require smoothness from the underlying dynamics such as in the form of spectral gaps of twisted transfer operators \cite{SahlstenStevens,BakerSahlsten,BakerKhalilSahlsten,Leclerc-JuliaSets,LiNaudPan}, which we can avoid.

Next we give the necessary consequences of the sum-product theorem that we need in our setting that are similar to Bourgain and Dyatlov \cite{BourgainDyatlov}. Since we also work with parabolic systems, the derivatives $g_a'$ in the infinitely branched systems may have arbitrarily small values, we need to formulate the bounds using moments of the multiplicative random variables, which we can link to the light tail condition by the previous Section \ref{sec:moments}.

Firstly, let us recall a quantitative version of this sum-product phenomenon, where this version is a consequence of \cite[Theorem 1.14]{OdSS}, see also \cite{BourgainDyatlov}:

\begin{thm}\label{thm:sumprodMain}
Let $\gamma_4 > 0$. Then there exists $\varepsilon_2 \in (0,1)$ such that for any $k \geq C \gamma_4^{-1}$ the following hold (where $C$ denotes an absolute constant). Let $\eta \in \mathbb{R}$, $|\eta| \geq 1$, let $C_0\geq 1$, and let $\mu_1, \dots, \mu_k $ be Borel measures supported on $[C_0^{-1},C_0]$ such that $\mu(\mathbb{R}) \leq C_0$ so that
\begin{align*}
 \forall a \in \mathbb{R}, \forall \sigma \in [|\eta|^{-1},|\eta|^{-\varepsilon_2}], \ \mu_j([a-\sigma,a+\sigma]) \leq C_0 \sigma^{\gamma_4}.
\end{align*}
Then, for some constant $C_1$ that only depends on $C_0,\gamma$:
\begin{align*}
 \Big| \int_{[C_0^{-1},C_0]^k} e^{i \eta x_1 \dots x_k} d\mu_1(x_1) \dots d\mu_k(x_k) \Big| \leq C_1 |\eta|^{-\exp(-C \gamma_4^{-1})}.
\end{align*}
\end{thm}

We will generalise this result for the case where $\mu$ is supported on $\mathbb{R}^*_+$, under a polynomial moment condition.

\begin{thm}\label{thm:generalSumProduct}
Let $\gamma_3 > 0$. Then there exists $\varepsilon_1 \in (0,1)$ and $k \geq 1$ such that the following holds: Let $\eta \in \mathbb{R}$, $|\eta| \geq 1$ and let $\mu_1, \dots, \mu_k $ be Borel probability measures on $\mathbb{R}^*_+$ such that:
\begin{align*}
 \forall j \in \{1,\dots, k\}, \ \forall a \in \mathbb{R}, \ \forall \sigma \in [|\eta|^{-2},|\eta|^{-\varepsilon_1}], \ (\ln_* \mu_j)([a-\sigma,a+\sigma]) \leq 2 \sigma^{\gamma_3}.
\end{align*}
Then, for some constant $c$ that only depends on $\gamma_3$, we find:
\begin{align*}
 \forall t \in (0,\varepsilon_1), \ \Big| \int_{(\mathbb{R}_+^*)^k} e^{i \eta x_1 \dots x_k} d\mu_1(x_1) \dots d\mu_k(x_k) \Big| \leq c |\eta|^{-\varepsilon_1} + |\eta|^{-\frac{t^3}{2 \sqrt{k}}} \prod_{j=1}^k M_j(t^2),
\end{align*}
where $M_j(t) := \int_{\mathbb{R}^*_+} e^{t |\ln x|} d\mu_j(x)$.
\end{thm}

\begin{proof}
Let $\gamma_3 \in (0,1)$. Define $\gamma_4 := \gamma_3/2$. Then Theorem \ref{thm:sumprodMain} gives us some $k \geq 1$ and some $\varepsilon_2 \in (0,1)$ (associated to $\gamma_4)$. Taking $\varepsilon_2$ smaller if necessary, let us assume $\varepsilon_2 \leq e^{-C/\gamma_4}$. We will define $\varepsilon_1 := \big( \frac{\varepsilon_2}{4k} \big)^{6}$. Now, let $\eta \in \mathbb{R}$ be such that $|\eta|\geq 1$. Let $\mu_1,\dots,\mu_k$ be measures that satisfies the hypothesis of Theorem \ref{thm:generalSumProduct} with respect to the quantities $(\gamma_3,k,\varepsilon_1,\eta)$. We wish to show that the bound on the integral holds.\\

The idea is to do a dyadic decomposition of the measures $\mu_j$. For each $n \in \mathbb{Z}$, we define $I_n := [2^n,2^{n+1}[$, and we define $\mu_{j,n} := (\mu_j)\mid_{I_n}$. We find
\begin{align*}
 \int_{(\mathbb{R}_+^*)^k} e^{i \eta x_1 \dots x_k} d\mu_1(x_1) \dots d\mu_k(x_k) = \sum_{\vec{n} \in \mathbb{Z}^k} \int_{I_{\vec{n}}} e^{i \eta x_1 \dots x_k} d\mu_{1,n_1}(x_1) \dots d\mu_{k,n_k}(x_k),
\end{align*}
where $\vec{n} = (n_1, \dots, n_k)$ and $I_{\vec{n}} = I_{n_1} \times \dots \times I_{n_k}$. We then cut the sum in two depending on if $|\vec{n}|_\infty := \max_j |n_j|$ is large or small. Fix some $t \in (0,\varepsilon_1)$. We have:
\begin{align*}
 &\int_{(\mathbb{R}_+^*)^k} e^{i \eta x_1 \dots x_k} d\mu_1(x_1) \dots d\mu_k(x_k) =\\
 &\underset{|\vec{n}|_\infty < t \log_2 |\eta|}{\sum_{\vec{n} \in \mathbb{Z}^k}} \int_{I_{\vec{n}}} e^{i \eta x_1 \dots x_k} d\mu_{1,n_1}(x_1) \dots d\mu_{k,n_k}(x_k) \ + \underset{|\vec{n}|_\infty \geq t \log_2 |\eta|}{\sum_{\vec{n} \in \mathbb{Z}^k}} \int_{I_{\vec{n}}} e^{i \eta x_1 \dots x_k} d\mu_{1,n_1}(x_1) \dots d\mu_{k,n_k}(x_k).
\end{align*}
The sum for large $|\vec{n}|_\infty$ gives
\begin{align*}
 \Big| \underset{|\vec{n}|_\infty \geq t \log_2 |\eta| }{\sum_{\vec{n} \in \mathbb{Z}^k}} \int_{I_{\vec{n}}} e^{i \eta x_1 \dots x_k} d\mu_{1,n_1}(x_1) \dots d\mu_{k,n_k}(x_k) \Big| \leq \underset{|\vec{n}|_\infty \geq t \log_2 |\eta|}{\sum_{\vec{n} \in \mathbb{Z}^k}} \mu_1(I_{n_1}) \dots \mu_k(I_{n_k}).
\end{align*}
The moment condition on the measures $\mu_j$ yields
\begin{align*}
 \forall j, \ \forall n \in \mathbb{Z}, \ \mu_j(I_n) \leq \int_{[2^n,2^{n+1}[} 2^{-t^2 |n|} e^{t^2 |\ln x|} d\mu_j(x) \leq 2^{-t^2 |n|} M_j(t^2),
\end{align*}
which gives, injecting into the previous sum (denoting by $d\sigma$ the Lebesgue measure on the sphere, $M(t) := M_1(t) \dots M_k(t)$ and $|\vec{n}|_p := (\sum_{j} |n_j|^p)^{1/p}$):

\begin{align*}
 &\Big| \underset{|\vec{n}|_\infty \geq t \log_2 |\eta| }{\sum_{\vec{n} \in \mathbb{Z}^k}} \int_{I_{\vec{n}}} e^{i \eta x_1 \dots x_k} d\mu_{1,n_1}(x_1) \dots d\mu_{k,n_k}(x_k) \Big| \leq M(t^2) \underset{|\vec{n}|_\infty \geq t \log_2 |\eta|}{\sum_{\vec{n} \in \mathbb{Z}^k}} 2^{- t^2 |\vec{n}|_1}\\
 &\leq M(t^2) \underset{|\vec{n}|_2 \geq t \log_2 |\eta|}{\sum_{\vec{n} \in \mathbb{Z}^k}} 2^{- \frac{t^2 |\vec{n}|_2}{\sqrt{k}}}\lesssim M(t^2) \int_{\{x \in \mathbb{R}^k, \ \|x\|_2 \geq t \log_2 |\eta|\}} 2^{-\frac{t^2 \|x\|_2}{\sqrt{k}}} dx\\
 &= M(t^2) \int_{\mathbb{S}^{k-1}} \int_{t \log_2 |\eta|}^\infty 2^{-\frac{t^2 r}{\sqrt{k}}} r^{k-1} dr d\sigma\lesssim M(t^2) \int_{t \log_2 |\eta|}^\infty 2^{-\frac{t^2 r}{2 \sqrt{k}}} dr \lesssim M(t^2) \eta^{-\frac{t^3}{2 \sqrt{k}}}.
\end{align*}

To conclude, it now suffices to bound the sum where $|\vec{n}|_\infty$ is small. We can rewrite it as
\begin{align*}
 &\underset{|\vec{n}|_\infty < t \log_2 |\eta|}{\sum_{\vec{n} \in \mathbb{Z}^k}} \int_{I_{\vec{n}}} e^{i \eta x_1 \dots x_k} d\mu_{1,n_1}(x_1) \dots d\mu_{k,n_k}(x_k) = \underset{|\vec{n}|_\infty < t \log_2 |\eta|}{\sum_{\vec{n} \in \mathbb{Z}^k}} \int_{[1,2]^k} e^{i \eta_{\vec{n}} x_1 \dots x_k} d\nu_{1,n_1}(x_1) \dots d\nu_{k,n_k}(x_k),
\end{align*}

where $\nu_{j,n}$ are the measures on $[1,2]$ defined by the relation $(m_{n})_* \nu_{j,n} = \mu_{j,n}$, for $m_n(x) := 2^n x$, and where we denoted $\eta_{\vec{n}} := \eta 2^{n_1 + \dots n_k}$. Now fix some $\vec{n} \in \mathbb{Z}^k$ with $|\vec{n}|_\infty \leq t \log_2 |\eta|$. We will check that the measures $\nu_{1,n_1}, \dots, \nu_{k,n_k}$ satisfies the hypothesis of Theorem \ref{thm:sumprodMain} for the quantities $(\gamma_4,\varepsilon_2,k,\eta_{\vec{n}})$. Notice that $|\eta|^{1/2} \leq |\eta_{\vec{n}}| \leq |\eta|^2$ since $t < \varepsilon_1 < \frac{1}{(4k)^4}$. We check the following conditions:
\begin{itemize}
\item Total mass condition:
\begin{align*}
 \forall j, \ \nu_{j,n_j}(\mathbb{R}) = \mu_{j}(I_{n_j}) \leq 1.
\end{align*}

\item Support condition:
\begin{align*}
 \forall j, \ \text{supp}(\nu_{j,n_j}) \subset [1,2].
\end{align*}

\item Non-concentration condition:
 $$\forall j, \ \forall a \in [1,2], \ \forall \sigma \in [|\eta_{\vec{n}}|^{-1}, |\eta_{\vec{n}}|^{-\varepsilon_2} ] \subset [\eta^{-2}, \eta^{-\varepsilon_2/2} ]: $$ $$\nu_{j,n_j}([a-\sigma,a+\sigma]) \leq \nu_{j,n_j}(\{ x \in [1,2], \ \ln x \in [\ln a - 2 \sigma, \ln a + 2 \sigma] \})$$
 $$ = \mu_j ( \{ x \in I_{n_j}, \ \ln x \in [n_j \ln(2) + \ln(a)- 2\sigma,n_j \ln(2) + \ln(a)+ 2\sigma ] \} ) \leq 4\sigma^{\gamma_3} $$
 by the non-concentration hypothesis on $\ln_*\mu_j$.
\end{itemize}
Hence the conclusion of Theorem \ref{thm:sumprodMain} apply in this setting. That is, for some constant $c>0$ that only depends on $\gamma_3$, we find the following bound:
\begin{align*}
 \Big| \int_{[1,2]^k} e^{i \eta_{\vec{n}} x_1 \dots x_k} d\nu_{1,n_1}(x_1) \dots d\nu_{k,n_k} \Big| \leq c |\eta_{\vec{n}}|^{-\varepsilon_2} \leq c |\eta|^{-\varepsilon_2/2}.
\end{align*}
We then inject this estimate in the sum, and find:
\begin{align*}
 \Big| \underset{|\vec{n}|_\infty < t \log_2 |\eta|}{\sum_{\vec{n} \in \mathbb{Z}^k}} \int_{I_{\vec{n}}} e^{i \eta x_1 \dots x_k} d\mu_{1,n_1}(x_1) \dots d\mu_{k,n_k}(x_k) \Big| \leq c (t \log_2 |\eta|)^k |\eta|^{-\varepsilon_2/2} \lesssim_{\gamma_3} |\eta|^{-\varepsilon_2/4} \lesssim_{\gamma_3} \eta^{-\varepsilon_1},
\end{align*}

thus concluding the proof.
\end{proof}

We will apply the previous Theorem to the special case where the measures are countable sums of Dirac mass, which correspond to our present setting.

\begin{thm}\label{thm:sumproductWords}
Let $\gamma_2 \in (0,1)$. Then there exists $\varepsilon_1 \in (0,1)$ and $k \geq 1$ such that the following holds. Let $\eta \in \mathbb{R}$, $|\eta| \geq 1$. Let $\mathcal{Z}_1,\dots,\mathcal{Z}_k$ be a finite or countable sets. Choose, for each $j$, some nonnegative weights $(w_\mathbf{b}^{(j)})_{\mathbf{b} \in \mathcal{Z}_j}$ that sums to one, and some maps $\zeta_j : \mathcal{Z}_j \rightarrow \mathbb{R}_+^*$. Suppose that for all $j \in \{1,\dots, k\}$, the following non-concentration condition holds:
 \begin{align*}
 \forall \sigma \in [\eta^{-2},\eta^{-\varepsilon_1}], \ \sum_{(\mathbf{b},\mathbf{c}) \in \mathcal{Z}_j^2} w_{\mathbf{b}}^{(j)} w_{\mathbf{c}}^{(j)} \mathbf{1}_{[-\sigma,\sigma]}( \ln \zeta_{j}(\mathbf{b}) - \ln \zeta_{j}(\mathbf{c}) ) \leq \sigma^{\gamma_2}.
 \end{align*}
Then, there exists a constant $c>0$ that only depends on $\gamma$ such that
\begin{align*}
 \forall t>0, \ \Big| {{\sum_{\mathbf{b}_1 \in \mathcal{Z}_1 }}} \dots
{{\sum_{\mathbf{b}_k \in \mathcal{Z}_k }}} w_{\mathbf{b}_1}^{(1)} \dots w_{\mathbf{b}_k}^{(k)} e^{i \eta \zeta_1(\mathbf{b}_1) \dots \zeta_k(\mathbf{b}_k)} \Big| \leq c |\eta|^{-\varepsilon_1^2} + \eta^{-\frac{t^3}{2 \sqrt{k}}} \prod_{j=1}^k M_j(t^2),
\end{align*}

where $M_j(t) := \sum_{\mathbf{b} \in \mathcal{Z}_j} w_{\mathbf{b}}^{(j)} e^{t |\ln \zeta_j(\mathbf{b})|}.$
\end{thm}

\begin{proof}
Let $\gamma_2>0$. We then fix $\gamma_3 := \gamma_2/2$ and let $\varepsilon_1$ and $k \geq 1$ be given by Theorem \ref{thm:generalSumProduct} associated to $\gamma_3$. Let $\mathcal{Z}_j$, $(w_{\mathbf{b}}^{(j)})_{\mathbf{b} \in \mathcal{Z}_j}$ and $\zeta_j : \mathcal{Z}_j \rightarrow \mathbb{R}_+^*$ be such that they satisfy the hypothesis of Theorem \ref{thm:sumproductWords}. Define some measures $\mu_j$ by
\begin{align*}
 \mu_j := \sum_{\mathbf{b} \in \mathcal{Z}_j} w_{\mathbf{b}}^{(j)} \delta_{\zeta_j(\mathbf{b})}.
\end{align*}

Let us check the non-concentration condition:
\begin{align*}
 &\ln_*\mu_j([a-\sigma,a+\sigma]) = \sum_{\mathbf{b} \in \mathcal{Z}_j} w_{\mathbf{b}}^{(j)} \mathbf{1}_{[a-\sigma,a+\sigma]}(\ln \zeta_{j}(\mathbf{b}))\\
 &= \Big( \sum_{\mathbf{b},\mathbf{c} \in \mathcal{Z}_j} w_{\mathbf{b}}^{(j)} w_{\mathbf{c}}^{(j)} \mathbf{1}_{[a-\sigma,a+\sigma]}(\ln \zeta_{j}(\mathbf{b})) \mathbf{1}_{[a-\sigma,a+\sigma]}(\ln \zeta_{j}(\mathbf{c})) \Big)^{1/2}\\
 &\leq \Big( \sum_{\mathbf{b},\mathbf{c} \in \mathcal{Z}_j} w_{\mathbf{b}}^{(j)} w_{\mathbf{c}}^{(j)} \mathbf{1}_{[2\sigma,2\sigma]}(\ln \zeta_{j}(\mathbf{b})- \ln \zeta_{j}(\mathbf{c})) \Big)^{1/2} \leq (2\sigma)^{\gamma_2/2}.
\end{align*}

The conclusion follows by applying the previous Theorem.
\end{proof}

\section{Discretisation to exponential sums of random products}\label{sec:reduction}

We finally start to prove Theorem \ref{thm:mainCountable}. The proof is separated in two parts. The first part is to reduce the Fourier decay bound to the sum-product phenomenon (Theorem \ref{thm:sumproductWords}) for some well chosen maps $(\zeta_j)$. This is the content of Section \ref{sec:reduction}. The second part is to check that the non-concentration condition is actually satisfied for these maps $\zeta_j$ under (QNL). This is the content of Section \ref{sec:nonconc}.

\subsection{Iterating into exponential sums of phase differences}

Recall from Section \ref{sec:moments} that the transfer operator associated to the potential $\phi$ is defined by
\begin{align*}
 \forall x \in I_b, \ \cL_\phi f(x) = \sum_{a \in \cA, ab \in \Sigma^2} e^{\phi(g_{ab}(x))}f(g_{ab}(x)), f : I \to \C.
\end{align*}
Let $\mu =\mu_\phi$ be the Gibbs measure associated to $\phi$, supported on the attractor $F$ of our IFS. If $\xi \in \R$ with $|\xi|$ large enough, choose the unique $n \in \N$ such that
\begin{align*}
 e^{((2k+1)\lambda+\eps_0)n}\leq |\xi| < e^{((2k+1)\lambda + \eps_0)(n+1)},
\end{align*}
where $\eps_0>0$ is a small constant that we will chose later (see Definition 4.1), the point being that $\xi e^{-(2k+1) \lambda n} \simeq e^{\varepsilon_0 n}$ is a term that we will chose to be slowly growing.

We will iterate $(2k+1)$ times the invariance relation $\int f \, d\mu = \int \cL_\phi^n f \, d\mu$. Because of the relation $g_{ab} \circ g_{bc} = g_{abc}$, it is convenient, for a given word $a_1 \dots a_{n+1}$, to define the notation $\mathbf{a}' := a_1 \dots a_n$. We can then write $g_{\mathbf{a'b}} = g_{\mathbf{a}} \circ g_{\mathbf{b}}$. Iterating the relation then gives
\begin{align*}
 \Big|\int \chi(x) e^{i\xi \psi(x)} \, d\mu(x)\Big|^2 = \Big| \sum\limits_{\A,\B} \int_{I_{b(\mathbf{A})}} w_{\A * \B}(x) \chi(g_{\A * \B}(x)) e^{i \xi \psi(g_{\A * \B}(x)) }\, d\mu(x)\Big|^2,
\end{align*}

where the sum is over blocks $\A = \a_1 \dots \a_{k+1} \in \cA^{n(k+1)},\B = \b_1 \dots \b_k \in \cA^{nk}$ of words $\a_j,\b_j\in \Sigma^n$ such that the concatenation
\begin{align*}
 \A * \B := \a_1' \b_1' \a_2' \b_2' \dots \a_k' \b_k' \a_{k+1}.
\end{align*}
is admissible (i.e. element of $\Sigma^{(2k+1)n}$). Recall that we use the weight notation
\begin{align*}
 w_{\A * \B}(x) := e^{S_{(2k+1)n} \phi(g_{\A * \B}(x)))}.
\end{align*}
We fix from now on some choices of reference points $x_a \in I_a$. We can then multiply and divide to get a form:
\begin{align*}
 \Big|\int \chi(x) e^{i\xi \psi(x)} \, d\mu(x)\Big|^2 = \Big|\sum\limits_{\A,\B} w_{\A * \B}(x_{b(\mathbf{A})}) \int_{I_{b(\mathbf{A})}} \frac{w_{\A * \B}(x)}{w_{\A * \B}(x_{b(\mathbf{A})})} \chi(g_{\A * \B}(x)) e^{i \xi \psi(g_{\A * \B}(x)) }\, d\mu(x)\Big|^2,
\end{align*}
which by Cauchy-Schwartz inequality and Gibbs property of $\mu$ is bounded by a constant multiple of
\begin{align*}
 \sum\limits_{\A,\B} w_{\A * \B}(x_{b(\mathbf{A})}) \Big|\int_{I_{b(\mathbf{A})}} \frac{w_{\A * \B}(x)}{w_{\A * \B}(x_{b(\mathbf{A})})} \chi(g_{\A * \B}(x)) e^{i \xi \psi(g_{\A * \B}(x)) }\, d\mu(x)\Big|^2.
\end{align*}
Since $\chi$ is H\"older with $\|\chi\|_{\beta_\chi} \lesssim |\xi|^{\eps_\chi}$, we have
\begin{align*}
 |\chi(g_{\A * \B}(x)) - \chi(g_{\A * \B}(x_{b(\mathbf{A})}))| \leq C |\xi|^{\eps_\chi} |g_{\A * \B}(x) - g_{\A * \B}(x_{b(\mathbf{A})})|^{\beta_\chi} \leq C |\xi|^{\eps_\chi} \kappa_+^{\beta_\chi (2k+1)n},
\end{align*}
so by the Gibbs property $\frac{w_{\A * \B}(x)}{w_{\A * \B}(x_{b(\mathbf{A})})} \lesssim 1$ we obtain a bound
\begin{align}
 \label{eq:FourierAfterCS}
 \sum\limits_{\A,\B}w_{\A * \B}(x_{b(\mathbf{A})})\chi(g_{\A * \B}(x_{b(\mathbf{A})}))^2 \Big|\int_{I_{b(\mathbf{A})}} \frac{w_{\A * \B}(x)}{w_{\A * \B}(x_0)} e^{i \xi \psi(g_{\A * \B}(x)) }\, d\mu(x)\Big|^2 + C' |\xi|^{\eps_\chi} \kappa_+^{2\beta_\chi (2k+1)n}.
\end{align}
Since $\A * \B = \a_1' \b_1' \a_2' \b_2' \dots \a_k' \b_k' \a_{k+1}$, we know that for some $t_j \in F$:
\begin{align*}
 w_{\A * \B }(x_{b(\mathbf{A})}) =w_{ \a_1}(t_{1}) w_{ \b_1}(t_{2}) w_{\a_2}(t_{3})w_{ \b_2}(t_{4}) \dots w_{\a_k}(t_{2k-1}) w_{\b_k}(t_{2k}) w_{\a_{k+1}}(t_{(2k+1)}).
\end{align*}
Thus we can apply Gibbs property applied $2k+1$ times to obtain
\begin{align*}
 w_{\A * \B }(x_{b(\mathbf{A})}) \leq C^{2k+1} \mu(I_{\a_1}) \mu(I_{\b_1}) \mu(I_{\a_2}) \mu(I_{\b_2}) \dots \mu(I_{\a_k}) \mu(I_{\b_k}) \mu(I_{\a_{k+1}}).
\end{align*}
Writing $p_\A = \mu(I_{\a_1}) \dots \mu(I_{\a_{k+1}})$ and $p_\B = \mu(I_{\b_1}) \dots \mu(I_{\b_k})$, we can bound \eqref{eq:FourierAfterCS} by

\begin{align*}
 C^{2k+1}\|\chi\|_\infty^2\sum\limits_{\A}p_\A \sum_{\B : \A \ast \B \in \Sigma^{(2k+1)n}}p_\B \Big|\int_{I_{b(\mathbf{A})}} \frac{w_{\A * \B}(x)}{w_{\A * \B}(x_{b(\mathbf{A})})} e^{i \xi \psi(g_{\A * \B}(x)) }\, d\mu(x)\Big|^2 + C' |\xi|^{\eps_\chi} \kappa_+^{2\beta_\chi (2k+1)n}.
\end{align*}

Here as $|\xi| \lesssim e^{(2k+1)\lambda n + \eps_0 n}$ and $e^{-\lambda} \leq \kappa_+$, we have $|\xi|^{\eps_\chi} \kappa_+^{2\beta_\chi (2k+1)n} \lesssim |\xi|^{-\beta'}$ for some $\beta' > 0$ as long as $\eps_\chi > 0$ is small enough. Thus, fixing $\A$, we just need to have decay for

\begin{align*}
 & \sum_{\B}p_\B \Big|\int_{b(\mathbf{A})} \frac{w_{\A * \B}(x)}{w_{\A * \B}(x_{b(\mathbf{A})})} e^{i \xi \psi(g_{\A * \B}(x)) }\, d\mu(x)\Big|^2.
\end{align*}

Writing $\A \sharp \B = \a_1' \b_1' \a_2' \b_2' \dots \a_k' \b_k$, we have $\A * \B = (\A \sharp \B)' \a_{k+1}$. We have that at any $x \in I_{b(\mathbf{A})}$:
\begin{align*}
 g_{\A * \B }'(x) = g_{ \A \sharp \B }'(g_{\a_{k+1}}(x)) g_{\a_{k+1}}'(x),
\end{align*}
and
\begin{align*}
 w_{\A * \B }(x) = w_{ \A \sharp \B }(g_{\a_{k+1}}(x)) w_{{\a_{k+1}}}(x).
\end{align*}
By Gibbs property $\frac{w_{\a_{k+1}}(x)}{w_{\a_{k+1}}(x_{b(\mathbf{A})})} \lesssim 1$, the H\"older property of $\phi$ gives us
\begin{align*}
 &\Big|\frac{w_{\A * \B}(x)}{w_{\A * \B}(x_{b(\mathbf{A})})} - \frac{w_{\a_{k+1}}(x)}{w_{\a_{k+1}}(x_{b(\mathbf{A})})}\Big| = \Big|\frac{w_{\A \sharp \B}(g_{\a_{k+1}}(x))}{w_{\A \sharp \B}(g_{\a_{k+1}}(x_{b(\mathbf{A})}))} - 1\Big|\frac{w_{\a_{k+1}}(x)}{w_{\a_{k+1}}(x_{b(\mathbf{A})})}\\
 & \lesssim |S_{2kn} \phi (g_{ \A \sharp \B }(g_{\a_{k+1}}(x))) - S_{2kn} \phi (g_{ \A \sharp \B }(g_{\a_{k+1}}(x_{b(\mathbf{A})})))|\\
 & \leq C|g_{\a_{k+1}}(x) - g_{\a_{k+1}}(x_{b(\mathbf{A})})|^{\beta} \lesssim \kappa_+^{\beta n}.
\end{align*}
Thus we obtain a bound for some universal $C' > 0$:
\begin{align*}
 &\sum_{\B} p_\B \Big|\int_{I_{b(\mathbf{A})}} \frac{w_{\a_{k+1}}(x)}{w_{\a_{k+1}}(x_{b(\mathbf{A})})} e^{i \xi \psi(g_{\A * \B}(x)) }\, d\mu(x)\Big|^2 + C' \kappa_+^{\beta n},
\end{align*}
so we just need to study
\begin{align*}
 \sum_{\B} p_\B \Big|\int_{I_{b(\mathbf{A})}} \frac{w_{\a_{k+1}}(x)}{w_{\a_{k+1}}(x_{b(\mathbf{A})})} e^{i \xi \psi(g_{\A * \B}(x)) }\, d\mu(x)\Big|^2.
\end{align*}
Expand the square so we just need to bound

\begin{align*}
 \iint_{I_{b(\mathbf{A})}^2} \sum_{\B} \frac{w_{\a_{k+1}}(x)}{w_{\a_{k+1}}(x_{b(\mathbf{A})})} \frac{w_{\a_{k+1}}(y)}{w_{\a_{k+1}}(x_{b(\mathbf{A})})} p_\B e^{i \xi (\psi(g_{\A * \B}(x))-\psi(g_{\A * \B}(y))) }\, d\mu(x) \, d\mu(y) \\
 = \iint_{I_{b(\mathbf{A})}^2} \frac{w_{\a_{k+1}}(x)}{w_{\a_{k+1}}(x_{b(\mathbf{A})})} \frac{w_{\a_{k+1}}(y)}{w_{\a_{k+1}}(x_{b(\mathbf{A})})} \sum_{\B} p_\B e^{i \xi (\psi(g_{\A * \B}(x))-\psi(g_{\A * \B}(y))) }\, d\mu(x) \, d\mu(y).
\end{align*}



\subsection{Linearisation of the phase}
Let us now linearise the phase. By the mean value theorem, we have, for some $t=t_{\mathbf{A},\mathbf{B}}(x,y)$:
\begin{align*}
\psi(g_{\A * \B}(x))-\psi(g_{\A * \B}(y))) = (\psi \circ g_{\A \sharp \B})'(t) (g_{\mathbf{a}_{k+1}}(x)-g_{\mathbf{a}_{k+1}}(y)).
\end{align*}
By the chain rule, we can then find $t_1 \in I_{\mathbf{a}_1},t_2 \in I_{\mathbf{a}_2},\dots,t_{k+1} \in I_{\mathbf{a}_{k+1}}$ (depending on $t$) such that
\begin{align*}
 (\psi \circ g_{\A \sharp \B})'(t) = \psi'(t_1) g_{\a_1' \b_1}'(t_2) \dots g_{\a_{k}' \b_k}'(t_{k+1}).
\end{align*}
We can then rewrite our sum of interest as
$$ \sum_{\B}p_\B \Big|\int_{I_{b(\mathbf{A})}} \frac{w_{\A * \B}(x)}{w_{\A * \B}(x_{b(\mathbf{A})})} e^{i \xi \psi(g_{\A * \B}(x)) }\, d\mu(x)\Big|^2 $$ $$ = \iint_{I_{b(\mathbf{A})}^2} \frac{w_{\a_{k+1}}(x)}{w_{\a_{k+1}}(x_{b(\mathbf{A})})} \frac{w_{\a_{k+1}}(y)}{w_{\a_{k+1}}(x_{b(\mathbf{A})})} \sum_{\B} p_\B e^{i \xi (\psi(g_{\A * \B}(x))-\psi(g_{\A * \B}(y))) }\, d\mu(x) \, d\mu(y) $$
$$ = \iint_{I_{b(\mathbf{A})}^2} \frac{w_{\a_{k+1}}(x)}{w_{\a_{k+1}}(x_{b(\mathbf{A})})} \frac{w_{\a_{k+1}}(y)}{w_{\a_{k+1}}(x_{b(\mathbf{A})})} \sum_{\B} p_\B e^{i \xi \psi'(t_1) g_{\a_1' \b_1}'(t_2) \dots g_{\a_{k}' \b_k}'(t_{k+1}) (g_{\mathbf{a}_{k+1}}(x) - g_{\mathbf{a}_{k+1}}(y)) }\, d\mu(x) \, d\mu(y).$$
To gain space, denote $W_{\mathbf{a}}(x,y) := \frac{w_{\a}(x)}{w_{\a}(x_{b(\mathbf{a})})} \frac{w_{\a}(y)}{w_{\a}(x_{b(\mathbf{a})})}$. It is more convenient to modify the signs in the phase so that everything becomes positive. To do so, notice now that this quantity is real. In particular, it is equal to it's real part, which gives:
$$ = \iint_{I_{b(\mathbf{A})}^2} W_{\mathbf{a}_{k+1}}(x,y) \sum_{\B} p_\B \cos\Big({ \xi \psi'(t_1) g_{\a_1' \b_1}'(t_2) \dots g_{\a_{k}' \b_k}'(t_{k+1}) (g_{\mathbf{a}_{k+1}}(x) - g_{\mathbf{a}_{k+1}}(y)) } \Big) d\mu(x) \, d\mu(y) $$
$$ = \iint_{I_{b(\mathbf{A})}^2} W_{\mathbf{a}_{k+1}}(x,y) \sum_{\B} p_\B \cos\Big({ \xi |\psi'(t_1)| |g_{\a_1' \b_1}'(t_2)| \dots |g_{\a_{k}' \b_k}'(t_{k+1})| |g_{\mathbf{a}_{k+1}}(x) - g_{\mathbf{a}_{k+1}}(y)| } \Big) d\mu(x) \, d\mu(y). $$
And then triangle inequality and the fact that $|\text{Re}(z)| \leq |z|$ yields:
$$ \leq \iint_{ I_{ b(\mathbf{A}) }^2 } W_{\mathbf{a}_{k+1}}(x,y) \Big|\sum_{\B} p_\B \cos\Big({ \xi |\psi'(t_1)| |g_{\a_1' \b_1}'(t_2)| \dots |g_{\a_{k}' \b_k}'(t_{k+1})| |g_{\mathbf{a}_{k+1}}(x) - g_{\mathbf{a}_{k+1}}(y)| } \Big)\Big| d\mu(x) \, d\mu(y) $$
$$ \leq \iint_{ I_{ b(\mathbf{A}) }^2 } W_{\mathbf{a}_{k+1}}(x,y) \Big|\sum_{\B} p_\B e^{ i \xi \psi'(t_1) |g_{\a_1' \b_1}'(t_2)| \dots |g_{\a_{k}' \b_k}'(t_{k+1})| |g_{\mathbf{a}_{k+1}}(x) - g_{\mathbf{a}_{k+1}}(y)| } \Big| d\mu(x) \, d\mu(y) $$
$$ \lesssim \iint_{ I_{ b(\mathbf{A}) }^2 } \Big|\sum_{\B} p_\B e^{ i \xi \psi'(t_1) |g_{\a_1' \b_1}'(t_2)| \dots |g_{\a_{k}' \b_k}'(t_{k+1})| |g_{\mathbf{a}_{k+1}}(x) - g_{\mathbf{a}_{k+1}}(y)| } \Big| d\mu(x) \, d\mu(y), $$
where the last inequality comes from Gibbs estimates. Let us now renormalize the phases and get rid of the dependencies of the $(t_j)_j$. Denote $x_{\mathbf{a}} := g_{\mathbf{a}}(x_{b(\mathbf{a})}) \in I_{\a}$ some reference points. For all $j \geq 1$, as $t_{j+1} \in I_{\mathbf{a}_{j+1}}$, we have by bounded $(C,\alpha)$-distortion property and mean value theorem for some $z_{\a_{j+1}} \in I$ that
\begin{align*}
 e^{-C\kappa_+^{\alpha n} } \leq \ab{\frac{g_{\a_{j}' \b_j}'(t_{j+1})}{ g_{\a_{j}' \b_j}'(x_{\mathbf{a}_{j+1}})}}\leq e^{C \kappa_+^{\alpha n} }.
\end{align*}
Similarly, by hypothesis we have that
\begin{align*}
 |\log \psi'(x) - \log \psi'(y)| \leq C_\psi |\xi|^{\tilde \eps_\psi}|x-y|^{\alpha},
\end{align*}
for some small enough $\tilde{\varepsilon}_\psi>0$. One can choose it so small that $|\xi|^{\tilde \eps_\psi} \leq \kappa_+^{-\alpha n/2}$. Then, for any $t_1 \in I_{\a_1}$, we find that
\begin{align*}
 e^{-C_\psi \kappa_+^{\alpha n/2}} \leq e^{-C_\psi |\xi|^{\tilde{\varepsilon}_\psi} \kappa_+^{\alpha n}} \leq \ab{\frac{\psi'(t_1)}{\psi'(x_{\mathbf{a}_{1}})}}\leq e^{C_\psi |\xi|^{\tilde{\varepsilon}_\psi} \kappa_+^{\alpha n}} \leq e^{C_\psi \kappa_+^{\alpha n/2}}.
\end{align*}
Thus
\begin{align*}
 e^{-\tilde{C} \kappa_+^{\alpha n/2}} \leq e^{-kC\kappa_+^{\alpha n} - C_\psi \kappa_+^{\alpha n/2}} \leq \Big|\frac{\psi'(t_1) g_{\a_1' \b_1}'(t_2) \dots g_{\a_{k}' \b_k}'(t_{k+1})}{\psi'(x_{\mathbf{a}_{1}}) g_{\a_1' \b_1}'(x_{\mathbf{a}_{2}}) \dots g_{\a_{k}' \b_k}'(x_{\mathbf{a}_{k+1}})}\Big| \leq e^{kC\kappa_+^{\alpha n} + C_\psi \kappa_+^{\alpha n/2}} \leq e^{\tilde{C} \kappa_+^{\alpha n/2}}.
\end{align*}

Since by mean value theorem $|e^x - 1| \leq |x|$ when $|x| \leq 1$, this implies for large enough $n$:
\begin{align*}
 \Bigg|\frac{|\psi'(t_1) g_{\a_1' \b_1}'(t_2) \dots g_{\a_{k}' \b_k}'(t_{k+1})|}{|\psi'(x_{\mathbf{a}_{1}}) g_{\a_1' \b_1}'(x_{\mathbf{a}_{2}}) \dots g_{\a_{k}' \b_k}'(x_{\mathbf{a}_{k+1}})|} - 1\Bigg| \leq \tilde{C} \kappa_+^{\alpha n/2}.
\end{align*}

Thus if we denote
\begin{align*}
 \eta_{\a_{k+1},\a_1}(x,y) := e^{-2\lambda n k} \xi |g_{\a_{k+1}}(x) - g_{\a_{k+1}}(y)| |\psi'(x_{\mathbf{a}_{1}})|,
\end{align*}
and
\begin{align*}
 \zeta_{j,\A}\su{\b} := e^{2\la n } |g_{\mathbf{a}_j' \mathbf{b}}'(x_{\mathbf{a}_{j+1}})|, \quad j = 1,2,\dots,k,
\end{align*}
then by the mean value theorem for some $z \in [x,y]$ we have
\begin{align*}
& |\xi \psi'(t_1) |g_{\a_1' \b_1}'(t_2)| \dots |g_{\a_{k}' \b_k}'(t_{k+1})| |g_{\mathbf{a}_{k+1}}(x) - g_{\mathbf{a}_{k+1}}(y)| - \eta_{\a_{k+1},\a_1}(x,y) \zeta_{1,\A}(\b_1) \dots \zeta_{k,\A}(\b_k)|\\
 &= |\xi| |g_{\a_{k+1}}(x) - g_{\a_{k+1}}(y)| \ab{ \psi'(t_1) |g_{\a_1' \b_1}'(t_2)| \dots |g_{\a_{k}' \b_k}'(t_{k+1})| - \psi'(x_{\mathbf{a}_{1}}) g_{\a_1' \b_1}'(x_{\mathbf{a}_{2}}) \dots g_{\a_{k}' \b_k}'(x_{\mathbf{a}_{k+1}})}\\
 & \leq \ab{\xi} \no{\psi'}_\infty |g_{\a_1' \b_1}'(x_{\mathbf{a}_{2}}) \dots g_{\a_{k}' \b_k}'(x_{\mathbf{a}_{k+1}})g_{\mathbf{a}_{k+1}}'(z)| \ab{\frac{\psi'(t_1) g_{\a_1' \b_1}'(t_2) \dots g_{\a_{k}' \b_k}'(t_{k+1})}{\psi'(x_{\mathbf{a}_{1}})g_{\a_1' \b_1}'(x_{\mathbf{a}_{2}}) \dots g_{\a_{k}' \b_k}'(x_{\mathbf{a}_{k+1}})} - 1 }\\
 & \leq \tilde{C} \ab{\xi} \no{\psi'}_\infty | g_{\a_1' \b_1}'(x_{\mathbf{a}_{2}}) \dots g_{\a_{k}' \b_k}'(x_{\mathbf{a}_{k+1}})g_{\mathbf{a}_{k+1}}'(z) | \kappa_+^{\alpha n/2} \\
 & \leq |\xi| \ | g_{\a_1' \b_1}'(x_{\mathbf{a}_{2}}) \dots g_{\a_{k}' \b_k}'(x_{\mathbf{a}_{k+1}})g_{\mathbf{a}_{k+1}}'(z) | \kappa_+^{\alpha n/3}.
\end{align*}
for $n$ large enough, if $\tilde{\varepsilon}_\psi$ is chosen so small that $ \no{\psi'}_\infty \lesssim \kappa_+^{-\alpha n/10}$. Recall now that we fixed the relation $|\xi| \simeq e^{(2k+1)\lambda n + \varepsilon_0 n}$. If we suppose that $\varepsilon_0$ is chosen so small that $\varepsilon_0 < \alpha |\ln \kappa_+|/6$, then we can further bound this difference by $$ \lesssim e^{(2k+1) \lambda n} g_{\A * \B}'(z) \kappa_+^{\alpha n/6}. $$
This is not a decaying term for all choice of words, but it is for most choices of words. More precisely, using the $1$-Lipschitz continuity of $x \mapsto e^{ix}$ and Proposition \ref{prop:moments}, we find:
\begin{align*}
 &\sum_{\A} p_\A \Big|\sum_{\B} p_\B e^{ i \xi \psi'(t_1) |g_{\a_1' \b_1}'(t_2)| \dots |g_{\a_{k}' \b_k}'(t_{k+1})| |g_{\mathbf{a}_{k+1}}(x) - g_{\mathbf{a}_{k+1}}(y)| } \Big|\\
 & \leq\sum_{\A} p_\A \Big|\sum_{\B} p_\B e^{i \eta_{\a_{k+1},\a_1}(x,y) \zeta_{1,\A}(\b_1) \dots \zeta_{k,\A}(\b_k)} \Big| \\
 & \qquad + \sum_{\A, \B} p_\A p_\B \Big|e^{ i \xi \psi'(t_1) |g_{\a_1' \b_1}'(t_2)| \dots |g_{\a_{k}' \b_k}'(t_{k+1})| |g_{\mathbf{a}_{k+1}}(x) - g_{\mathbf{a}_{k+1}}(y)| } - e^{i\eta_{\a_{k+1},\a_1}(x,y) \zeta_{1,\A}(\b_1) \dots \zeta_{k,\A}(\b_k) }\Big|\\
 & \leq\sum_{\A} p_\A \Big|\sum_{\B} p_\B e^{i \eta_{\a_{k+1},\a_1}(x,y) \zeta_{1,\A}(\b_1) \dots \zeta_{k,\A}(\b_k)} \Big| + 2 \sum_{\A, \B} p_\A p_\B \mathbf{1}( g_{\A*\B}'(x) e^{(2 k +1)\lambda n} \geq \kappa_+^{-\alpha n/20} ) \\
 & \qquad + C \sum_{\A, \B} p_\A p_\B \mathbf{1}( g_{\A*\B}'(x) e^{(2 k +1)\lambda n} \leq \kappa_+^{-\alpha n/20} ) e^{(2k+1) \lambda n} g_{\A * \B}'(z) \kappa_+^{\alpha n/6} \\
 & \leq\sum_{\A} p_\A \Big|\sum_{\B} p_\B e^{i \eta_{\a_{k+1},\a_1}(x,y) \zeta_{1,\A}(\b_1) \dots \zeta_{k,\A}(\b_k)} \Big| \\
 & \qquad + C \sum_{\mathbf{C} \in \Sigma^{(2k+1) n}} p_\mathbf{C} \Big(
e^{(2k+1) \lambda n} g_{\mathbf{C}}'(z) \kappa_+^{\alpha n/20} \Big)^t + C \kappa_+^{\alpha n/10} \\
& \leq\sum_{\A} p_\A \Big|\sum_{\B} p_\B e^{i \eta_{\a_{k+1},\a_1}(x,y) \zeta_{1,\A}(\b_1) \dots \zeta_{k,\A}(\b_k)} \Big| + C \kappa_+^{\alpha t n/20} e^{\varepsilon_{\gamma_0} n t^2} + C \kappa_+^{\alpha n/10} \\
& \leq\sum_{\A} p_\A \Big|\sum_{\B} p_\B e^{i \eta_{\a_{k+1},\a_1}(x,y) \zeta_{1,\A}(\b_1) \dots \zeta_{k,\A}(\b_k)} \Big| + C ( e^{-\frac{\alpha^2 |\ln \kappa_+|^2}{\varepsilon_{\gamma_0} 800 } n} + \kappa_+^{\alpha n/10})
\end{align*}
by choosing $t := \frac{\alpha |\ln \kappa_+|}{\varepsilon_{\gamma_0} 20}$.\\

Thus, to conclude Fourier decay of the measure $\mu$, we just need to bound
\begin{align*}
 \sum_{\A} p_\A \iint_{ I_{ b(\mathbf{A}) }^2 } X_{x,y}(\A) \, d\mu(x) \, d\mu(y),
\end{align*}
where $X_{x,y}(\mathbf{A})$ is the $\R_+$-valued random variable
\begin{align*}
 X_{x,y}(\A) := \Big| \sum_{\B} p_\B e^{i \eta_{\a_{k+1},\a_1}(x,y) \zeta_{1,\A}(\b_1) \dots \zeta_{k,\A}(\b_k)} \Big|.
\end{align*}

\subsection{Applying the sum-product bound under non-concentration of derivatives}

In the notation in Theorem \ref{thm:sumproductWords}, we can write our random variable as
\begin{align*}
 X_{x,y}(\A) = \Big|{{\sum_{\mathbf{b}_1 \in \mathcal{Z}_1 }}} \dots
{{\sum_{\mathbf{b}_k \in \mathcal{Z}_k }}} w_{\mathbf{b}_1}^{(1)} \dots w_{\mathbf{b}_k}^{(k)} e^{i \eta \zeta_1(\mathbf{b}_1) \dots \zeta_k(\mathbf{b}_k)} \Big|,
\end{align*}
where the map
\begin{align*}
 \zeta_j(\b) := \zeta_{j,\A}(\b) = e^{2\lambda n} |g_{\a_j' \b}'(x_{\a_{j+1}})|,
\end{align*}
defined on the sets $\mathcal{Z}_j := \mathcal{Z}_{j,\A} = \{\b : \a_{j}' \b' \a_{j+1} \ \text{is admissible} \}$ and
\begin{align*}
 w_{\mathbf{b}_j}^{(j)} := \mu(I_{\b_j}) \in [0,1] \quad \text{and}\quad \eta := \eta_{\a_{k+1},\a_1}(x,y).
\end{align*}

Now, by Theorem \ref{thm:sumproductWords}, if there exists $0 < \gamma < 1$ and $\eps_1 > 0$ such that if
for all $j \in \{1,\dots, k\}$, the following non-concentration condition holds for when $|\eta_{\a_{k+1},\a_1}(x,y)| \geq 1$:
\begin{align}
 \label{eq:proofnc}
\forall \sigma \in [|\eta_{\a_{k+1},\a_1}(x,y)|^{-2},|\eta_{\a_{k+1},\a_1}(x,y)|^{-\varepsilon_1}], \ \sum_{(\mathbf{b},\mathbf{c}) \in \mathcal{Z}_{j,\A}^2} \mu(I_{\b}) \mu(I_{\c}) \mathbf{1}_{[-\sigma,\sigma]}( \zeta_{j,\A}(\mathbf{b}) - \zeta_{j,\A}(\mathbf{c}) ) \leq \sigma^\gamma,
\end{align}
Then, there exists a constant $c>0$ that only depends on $\gamma$ such that
\begin{align}
 \label{eq:proofbd}\forall t>0, \ X_{x,y}(\A) \leq c |\eta_{\a_{k+1},\a_1}(x,y)|^{-\varepsilon_1^2} + |\eta_{\a_{k+1},\a_1}(x,y)|^{-\frac{t^3}{2 \sqrt{k}}}. \prod_{j=1}^k M_{j,\A}(t^2),
\end{align}

where the $t$-moment $M_{j,\mathbf{A}}(t)$ is given by
\begin{align*}
 M_{j,\A}(t) := \sum_{\mathbf{b} \in \mathcal{Z}_{j,\A}} \mu(I_\b) e^{t |\ln \zeta_{j,\A}(\mathbf{b})|}.
\end{align*}

We will now define the sets of \textit{well-distributed blocks and words} on which the sum-product phenomenon applies, which requires us to finally fix parameters. The following Definition also reveals to us how the ultimate Fourier decay exponent depend on the hypotheses of the IFS:

\begin{Definition}[Choices of parameters]
Denote by $\Theta,\rho \in (0,1)$ the two constants given by our hypothesis (QNL). Recall that $\alpha \in (0,1)$ is an Hölder exponent of $\tau$ and that $\kappa_+ := \sup_{a \in \mathcal{A}} |g_a'|_\infty \in (0,1)$. We then define
$$\eps_0 := \alpha |\ln \kappa_+| \cdot |\ln \rho| /10 > 0.$$
Furthermore, we set
$\gamma_2 := \Theta/4$. Theorem \ref{thm:sumproductWords} then fixes some $k \in \mathbb{N}$ and $\varepsilon_1 \in (0,1)$ (that are associated to $\gamma_2>0$).
\end{Definition}

With these fixed, we can now define

\begin{Definition}[Well-distributed words]\label{def:well}
We say that a block is $\mathbf{A}=\mathbf{a_0} \dots \mathbf{a}_k \in \Sigma^{(k+1)n}$ is \textit{well-distributed} if for all $j \in \llbracket 1,k\rrbracket$ and for all $x\in I$:

\begin{align*}
 \forall \sigma \in [ e^{- 4 \varepsilon_0 n}, e^{-\varepsilon_0 \varepsilon_1 n/2} ], \quad \sum_{(\mathbf{b},\mathbf{c}) \in \mathcal{Z}_{j,\mathbf{A}}^2} w_{\mathbf{b}}(x) w_{\mathbf{c}}(x) \chi_{[-\sigma,\sigma]}( \ln \zeta_{j,\mathbf{A}}(\mathbf{b}) - \ln \zeta_{j,\mathbf{A}}(\mathbf{c})) \leq \sigma^{\gamma_2}.
\end{align*}

Similarly, we call a couple of words $(\mathbf{a},\mathbf{d}) \in (\Sigma^n)^2$ \textit{well-distributed} if, for all $x \in I$:

\begin{align*}
 \forall \sigma \in [ e^{- 4 \varepsilon_0 n}, e^{-\varepsilon_0 \varepsilon_1 n/2} ], \quad \underset{\mathbf{a}'\mathbf{c}'\mathbf{d}' \in \Sigma^{3n}}{\underset{\mathbf{a}'\mathbf{b}'\mathbf{d}' \in \Sigma^{3n}}{\sum_{(\mathbf{b},\mathbf{c}) \in (\Sigma^{n+1})^2}}} w_{\mathbf{b}}(x) w_\mathbf{c}(x) \chi_{[-\sigma,\sigma]}( \ln g_{\mathbf{a}' \mathbf{b}}'(x_\mathbf{d}) - \ln g_{\mathbf{a}' \mathbf{c}}'(x_\mathbf{d})) \leq \sigma^{\gamma_2}.
\end{align*}
We will denote the set of well-distributed blocks by $\cW_n^{k+1} $, and the set of well-distributed couples by $\cW_{n}^2$.
\end{Definition}
We have the following crucial non-concentration property, whose proof is postponed to Section \ref{sec:nonconc}:

\begin{prop}\label{lma:nonconcmain0}
Define $\alpha_{reg} := \varepsilon_0 \varepsilon_1 \gamma_2 > 0$. Then most blocks are well-distributed:
\begin{align*}
 \sum_{ \A \notin \cW_n^{k+1} } w_{\mathbf{A}}(x_{b(\mathbf{A})}) \lesssim e^{-\alpha_{reg} n}.
\end{align*}

\end{prop}

Assuming Proposition \ref{lma:nonconcmain0}, we can now complete the proof of Theorem \ref{thm:mainCountable}. In the following we will use the following notation:
\begin{align*}
 \kappa_-({\a_{k}}) = \inf |g_{\a_{k}}'| \quad \text{and} \quad \kappa_+({\a_{k}}) = \sup |g_{\a_{k}}'|.
\end{align*}
Let then $\tilde \cW$ be those blocks $\A \in \Sigma^{(k+1)n}$ such that the last word $\a_{k+1}$ in the block $\A$ satisfies
\begin{align*}
 C_\psi e^{\lambda n} \kappa_+(\a_{k+1}) \leq e^{\eps_0 n},
\end{align*}
where $\eps_\psi > 0$ comes from the condition
\begin{align*}
 \inf_x |\psi'(x,\xi)| \gtrsim |\xi|^{-\eps_\psi} = \exp(-\eps_\psi c_0 n),
\end{align*}

recall that $|\xi| = e^{c_0 n}$ for $c_0 = (2k+1)\lambda + \eps_0$.

Finally, given any block $\A \in \Sigma^{(k+1)n}$, let $\cF_\A$ be the set of pairs $(x,y) \in I_{b(\mathbf{A})}^2$ such that
\begin{align*}
 |x-y| > \exp(-(\lambda -\eps_\psi c_0)n) \kappa_-({\a_{k+1}})^{-1} \exp(-\eps_0 n / 2).
\end{align*}

\begin{lem} We can bound for any small enough $t > 0$ that and large enough $n$:
\begin{align*}
 &\sum_{\A} p_\A \iint_{ I_{ b(\mathbf{A}) }^2 } X_{x,y}(\A) d\mu(x)\,d\mu(y) \\
 & \lesssim e^{-\alpha_{reg} n} + C_\psi^t 2C_\phi e^{-(\eps_0 t + \eps_{\gamma_0} t^2) n} + 2C_\phi e^{-(s_\mu(\lambda +\eps_\psi c_0- \eps_0 / 2)+t \lambda - \eps_{\gamma_0} t^2)n} \\
 & \qquad + \sum_{\A \in \cW \cap \tilde \cW} p_\A \sup_{(x,y) \in \cF_\A} X_{x,y}(\A).
\end{align*}
\end{lem}

\begin{proof}
The proof combines the bounds from Proposition \ref{lma:nonconcmain}, Proposition \ref{prop:moments} and Proposition \ref{prop:upperreg} together with the bound $X_{x,y}(\A) \leq 1$ due to $|e^{ix}| = 1$ and that $\mu$ is a probability measure.

Firstly, by the non-concentration Proposition \ref{lma:nonconcmain} and Gibbs property, we have
\begin{align*}
 \sum_{\A \notin \cW} p_\A \lesssim e^{-\alpha_{reg} n},
\end{align*}
for some $\alpha_{reg} > 0$, which gives the first term.

Secondly, by Proposition \ref{prop:moments}, we have for any small enough $t > 0$ that
\begin{align*}
 & \sum_{\A} p_\A \iint_{ I_{ b(\mathbf{A}) }^2 } X_{x,y}(\A) \1(\a_{k+1} : C_\psi e^{\lambda n} \kappa_+(\a_{k+1}) > e^{\eps_0 n}) \, d\mu(x) \, d\mu(y) \\
 & \leq C_\psi^t e^{-\eps_0 t n}\sum_{\a_1 \dots \a_k} \mu(I_{\a_1}) \dots \mu(I_{\a_k}) \sum_{\a_{k+1}} \mu(I_{\a_{k+1}}) (e^{\lambda n} \kappa_+(\a_{k+1}))^t \\
 & \leq C_\psi^t e^{-\eps_0 t n}2C_\phi e^{\eps_{\gamma_0} t^2 n} = C_\psi^t 2C_\phi e^{-(\eps_0 t - \eps_{\gamma_0} t^2) n},
\end{align*}

where $\eps_{\gamma_0} t^2 < \eps_0 t$ as long as $t > 0$ is small enough.

Finally, we can use the $s_\mu$-upper regularity of $\mu$ (Proposition \ref{prop:upperreg}) and that $\mu$ is a probability measure, and that $|e^{ix}| = 1$ to get a bound:
\begin{align*}
 &\sum_{\A} p_\A \iint_{\cF_\A^c} X_{x,y}(\A) \, d\mu(x) \, d\mu(y) \\
 & \leq\sum_{\A} p_\A \mu^{2 \otimes}\big((x,y) : |x-y| \leq \exp(-(\lambda -\eps_\psi c_0)n) \kappa_-({\a_{k+1}})^{-1} \exp(-\eps_0 n / 2)\big) \\
 & \lesssim\sum_{\A} p_\A \exp(-s_\mu(\lambda -\eps_\psi c_0)n) \kappa_-({\a_{k+1}})^{-s_\mu} \exp(-s_\mu \eps_0 n / 2)\\
 & \leq \exp(-n s_\mu(\lambda -\eps_\psi c_0+\eps_0 / 2))\exp(-t\lambda n) \sum_{\a_1 \dots \a_k} \mu(I_{\a_1})\dots \mu(I_{\a_k}) \sum_{\a_{k+1}} \mu(I_{\a_{k+1}}) e^{t\lambda n} \kappa_-({\a_{k+1}})^{-t} \\
 & \leq \exp(-n s_\mu(\lambda -\eps_\psi c_0+ \eps_0 / 2))\exp(-t\lambda n) 2C_\phi \exp(\eps_{\gamma_0} t^2 n),
\end{align*}
as $\mu$ is a probability measure by Proposition \ref{prop:moments} as long as $0 < t < \min\{t_{\gamma_0},s_\mu\}$. Thus choosing $t > 0$ and $\eps_\psi > 0$ small enough that
\begin{align*}
 s_\mu(\lambda -\eps_\psi c_0+ \eps_0 / 2)+t \lambda - \eps_{\gamma_0} t^2 > 0,
\end{align*}
we have exponential decay for this term. \end{proof}

Let $\A \in \cW \cap \tilde \cW$ and $(x,y) \in \cF_\A$. Then, since $c_0 = (2k+1)\lambda + \eps_0$, we can bound from below:
\begin{align*}|\eta_{\a_{k+1},\a_1}(x,y)|
&= |e^{-2\lambda n k} \xi (g_{\a_{k+1}}(x) - g_{\a_{k+1}}(y)) \psi'(g_{\mathbf{a}_{1}}(x_0)) | \\
& \geq \exp((\lambda +\eps_0)n) (\inf |\psi'|) \kappa_-({\a_{k+1}}) |x-y|
\\
& \geq \exp((\lambda +\eps_0)n) \exp(-\eps_\psi c_0 n) \kappa_-({\a_{k+1}}) |x-y|
\\
& \geq \exp(\eps_0 n) \exp((\lambda -\eps_\psi c_0)n) \kappa_-({\a_{k+1}}) |x-y|,
\\
& \geq e^{\eps_0 n / 2},
\end{align*}
and from above as $\A \in \tilde \cW$ and $|x-y| \leq |I|$:
\begin{align*}|\eta_{\a_{k+1},\a_1}(x,y)| & \leq e^{-2\lambda n k} |\xi| (\sup |\psi'|) (\sup |g_{\a_{k+1}}'|) |x-y|\\
& \lesssim e^{-2\lambda n k} e^{(2k+1)\lambda n + \eps_0 n} (\sup |\psi'|) (\sup |g_{\a_{k+1}}'|) |x-y|\\
&\leq e^{\eps_0 n} C_\psi e^{\lambda n} \kappa_+({\a_{k+1}}) \\
& \leq e^{2\eps_0 n}.
\end{align*}
Thus by \eqref{eq:proofbd} we know that for all $t > 0$:
\begin{align}
 \label{eq:crucialbound} \sum_{\A \in \cW \cap \tilde \cW} p_\A \sup_{(x,y) \in \cF_\A} X_{x,y}(\A) \leq c e^{-\eps_0 \varepsilon_1^2 n / 2} + \sum_{\A} p_\A (\exp(\eps_0 n / 2))^{-\frac{t^3}{2 \sqrt{k}}} \prod_{j=1}^k M_{j,\A}(t^2)
\end{align}
so we have reduced Theorem \ref{thm:mainCountable} to the moment bound for $M_{j,\A}(t^2)$. If $t \in (0,\eps_1)$, since $|\ln \zeta_{j,\A}\su{\b}| = |2\la n- \ln |g_{\a_j\b}'\su{g_{\a_{j+1}}\su{x_0}}|$, we have by the chain rule and Gibbs property
\begin{align*}\mu(I_{\a_j})\mu(I_\b) e^{t^2 | \ln \zeta_{j,\A}(\b)|}& \lesssim w_{\a_j' \b}(x_{\a_{j+1}}) \Big (e^{2\la n} g_{\a_j\b}'\su{x_{\a_{j+1}}}\Big)^{t^2}
\end{align*}
Thus by Proposition \ref{prop:moments} we have
\begin{align*}
 \sum_{\a_1\dots\a_k} \mu(I_{\a_1}) \dots \mu(I_{\a_k}) \prod_{j=1}^k M_{j,\A}(t^2) &= \prod_{j=1}^k \sum_{\a_j} \mu(I_{\a_j}) M_{j,\A}(t^2) \\
&= \prod_{j=1}^k \sum_{\a_j} \mu(I_{\a_j}) \sum_{\mathbf{b} \in \mathcal{Z}_{j,\A}} \mu(I_\b) e^{t |\ln \zeta_{j,\A}(\mathbf{b})|}\\ &= \prod_{j=1}^k \sum_{\a_j} \sum_{\b} \mu(I_{\a_j}) \mu(I_\b) e^{t^2 | \ln \zeta_{j,\A}(\b)|}\\
& \lesssim \Big( 2C_\phi e^{2\eps_{\gamma_0} nt^4}\Big)^k.
\end{align*}
Thus Proposition \ref{prop:moments} gives
\begin{align*}
& \sum_{\A} p_\A (\exp(\eps_0 n / 2))^{-\frac{t^3}{2 \sqrt{k}}} \prod_{j=1}^k M_{j,\A}(t^2) \\
& = \sum_{\a_{k+1}} \mu(I_{\a_{k+1}}) (\exp(\eps_0 n / 2))^{-\frac{t^3}{2 \sqrt{k}}} \sum_{\a_1\dots \a_k} \mu\su{I_{\a_{1}}}\cdots \mu\su{I_{\a_{k}}}\prod_{j=1}^k M_{j,\A}(t^2) \\
& \lesssim \sum_{\a_{k+1}} \mu(I_{\a_{k+1}}) (\exp(\eps_0 n / 2))^{-\frac{t^3}{2 \sqrt{k}}} \Big( 2C_\phi e^{2\eps_{\gamma_0} nt^4}\Big)^k\\
& = (2C_\phi)^k \exp\Big(-n(\frac{t^3}{ 4\sqrt{k}} \eps_0 - 2k\eps_{\gamma_0} t^4 )\Big).
\end{align*}
Now we see that
\begin{align*}
 \frac{t^3}{ 4\sqrt{k}} \eps_0 - 2k\eps_{\gamma_0} t^4 > 0 \quad \text{if and only if} \quad 0 < t < \frac{1}{ 8k^{3/2}\eps_{\gamma_0}} \eps_0,
\end{align*}
which is possible if $t > 0$ is chosen small enough. This completes the proof of Theorem \ref{thm:mainCountable}.

\section{Reducing multiscale non-concentration to larger probability space}\label{sec:nonconc}

We are now left with proving Proposition \ref{lma:nonconcmain0}. Recall its formulation:
\begin{prop}\label{lma:nonconcmain}
Define $\alpha_{reg} := \varepsilon_0 \varepsilon_1 \gamma_2 > 0$. Then most blocks are well-distributed:
\begin{align*}
 \sum_{ \A \notin \cW_n^{k+1} } w_{\mathbf{A}}(x_{b(\mathbf{A})}) \lesssim e^{-\alpha_{reg} n},
\end{align*}
recall Definition \ref{def:well} for the Definition of well-distributed blocks of words.
\end{prop}

This estimate will be proved using by a sequence of reductions, adapting arguments found e.g. in \cite{BourgainDyatlov,SahlstenStevens,Leclerc-JuliaSets,BakerSahlsten}. For clarity, we will drop the conditions on admissiblity of words in the sums, they are implicit. We denote $w_\mathbf{a} := w_{\mathbf{a}}(x_{b(\mathbf{a})})$.

The first step is to reduce this estimate to an easier bound in a larger probability, using Markov's inequality.

\begin{lem}\label{lma:firstreduction}
Suppose that, for $n$ large enough and for every $\sigma \in [e^{-4 \varepsilon_0 n}, e^{-\varepsilon_0 \varepsilon_1 n /2}]$:
\begin{align*}
\sum_{(\mathbf{a},\mathbf{b},\mathbf{c},\mathbf{d}) \in (\Sigma^n)^4} w_{\mathbf{a}} w_\mathbf{b} w_\mathbf{c} w_{\mathbf{d}} \chi_{[-\sigma,\sigma]}( \ln g_{\mathbf{a}' \mathbf{b}}'(x_\mathbf{d}) - \ln g_{\mathbf{a}' \mathbf{c}}'(x_\mathbf{d}) ) \leq \sigma^{2 \gamma_2}.
\end{align*}
Then the conclusion of Proposition \ref{lma:nonconcmain} holds.
\end{lem}

\begin{proof}
First of all, notice that well-distributed blocks are concatenations of well-distributed couples. In other words:
\begin{align*}
 \cW_n^{k+1} = \bigcap_{j=1}^k \se{ \mathbf{A} \in \Sigma^{(k+1)n}\mid \ (\mathbf{a}_{j-1},\mathbf{a}_{j}) \in \cW_n^2 }.
\end{align*}

We then find (using Gibbs estimates $w_\mathbf{a} \simeq \mu(I_\mathbf{a})$ and $w_\mathbf{A} \simeq w_{\mathbf{a}_1} \dots w_{\mathbf{a}_k}$):
\begin{align*}
 &\sum_{\mathbf{A} \in \Sigma^{n(k+1)} \setminus \cW_n^{k+1}} w_{\mathbf{A}} \leq \sum_{j=1}^k \underset{(\mathbf{a}_{j-1},\mathbf{a}_j) \notin \cW_n^2}{\sum_{\mathbf{A} \in \Sigma^{n(k+1)}}} w_\mathbf{A} \lesssim {\sum_{(\mathbf{a},\mathbf{c}) \in (\Sigma^{n})^2 \setminus \cW_n^2}} w_{\mathbf{a}} w_\mathbf{c}.
\end{align*}
If we define, for all $l \in \mathbb{N}$, the set
\begin{align*}
 \cW_{n,l}^2 := \Big\{ (\mathbf{a},\mathbf{d}) \in (\Sigma^{n})^2, \ \sum_{\mathbf{b},\mathbf{c} \in \Sigma^n} w_{\mathbf{b}} w_{\mathbf{c}} \chi_{[-e^{-l},e^{-l}]}( \ln g_{\mathbf{a}' \mathbf{b}}'(x_\mathbf{d}) - \ln g_{\mathbf{a}' \mathbf{c}}'(x_\mathbf{d}) ) \leq e^{-(l+1) \gamma_2} \Big\},
\end{align*}
then we have the inclusion
\begin{align*}
 \cW_{n}^2 \supset \bigcap_{\varepsilon_0 \varepsilon_1 n/2 \leq l \leq 4 \varepsilon_0 n} \cW_{n,l}^2,
\end{align*}
which gives, by Markov's inequality:
\begin{align*}
 &{\sum_{(\mathbf{a},\mathbf{c}) \in (\Sigma^{n})^2 \setminus \cW_n^2}} w_{\mathbf{a}} w_\mathbf{c} \leq \underset{\varepsilon_0 \varepsilon_1 n/2 \leq l \leq 4 \varepsilon_0 n}{\sum_{l \in \mathbb{N}}}
 {\sum_{(\mathbf{a},\mathbf{d}) \in (\Sigma^{n})^2 \setminus \cW_{n,l}^2}} w_{\mathbf{a}} w_\mathbf{d}\\
 &\leq \underset{\varepsilon_0 \varepsilon_1 n/2 \leq l \leq 4 \varepsilon_0 n}{\sum_{l \in \mathbb{N}}} e^{(l+1)\gamma_2}
 {\sum_{(\mathbf{a},\mathbf{b},\mathbf{c},\mathbf{d}) \in (\Sigma^{n})^4}} w_{\mathbf{a}} w_{\mathbf{b}} w_{\mathbf{c}} w_\mathbf{d} \ \chi_{[-e^{-l},e^{-l}]}(e^{2 \lambda n} g_{\mathbf{a}' \mathbf{b}}'(x_\mathbf{d}) - e^{2 \lambda n}g_{\mathbf{a}' \mathbf{c}}'(x_\mathbf{d}) )\\
 &\leq \underset{\varepsilon_0 \varepsilon_1 n/2 \leq l \leq 4 \varepsilon_0 n}{\sum_{l \in \mathbb{N}}} e^{(l+1)\gamma_2} e^{-2 l \gamma_2} \lesssim e^{- \varepsilon_0 \varepsilon_1 \gamma_2 n }.
\end{align*}\end{proof}

Recall that we can write $\ln g_\mathbf{a}'(x) = - S_n \tau \circ g_\mathbf{a}$. The bound in the hypothesis of the previous Lemma can then be rewritten in terms of Birkhoff sums. We check in the next Lemma that the desired bound holds under (QNL).

\begin{lem}
Suppose (QNL). Then for every $\sigma \in [e^{-4 \varepsilon_0 n}, e^{-\varepsilon_0 \varepsilon_1 n /2}]$:
\begin{align*} \sum_{\mathbf{a},\mathbf{b},\mathbf{c},\mathbf{d} \in \Sigma^n} w_\mathbf{a} w_{\mathbf{b}} w_{\mathbf{c}} w_\mathbf{d} \ \chi_{[-\sigma,\sigma]}(S_{2n} \tau \circ g_{\mathbf{a}' \mathbf{b}}(x_\mathbf{d}) -S_{2n} \tau \circ g_{\mathbf{a}' \mathbf{c}}(x_\mathbf{d}) ) \lesssim \sigma^{\Theta/2}.
\end{align*}
\end{lem}

\begin{proof}
We have, by Cauchy-Schwarz applied to the sum on $\mathbf{a},\mathbf{b},\mathbf{c}$:
\begin{align*}
&\Big(\sum_{\mathbf{a},\mathbf{b},\mathbf{c},\mathbf{d} \in \Sigma^n} w_{\a} w_{\b} w_{\c} w_{\mathbf{d}} \ \chi_{[-\sigma,\sigma]}(S_{2n} \tau \circ g_{\mathbf{a}' \mathbf{b}}(x_\mathbf{d}) -S_{2n} \tau \circ g_{\mathbf{a}' \mathbf{c}}(x_\mathbf{d}) ) \Big)^2\\
 &\lesssim \sum_{\mathbf{a},\mathbf{b},\mathbf{c} \in \Sigma^n} w_{\a} w_{\b} w_{\c} \Big( \sum_{\mathbf{d} \in \Sigma^n} w_{\mathbf{d}} \chi_{[-\sigma,\sigma]}(S_{2n} \tau \circ g_{\mathbf{a}' \mathbf{b}}(x_\mathbf{d}) -S_{2n} \tau \circ g_{\mathbf{a}' \mathbf{c}}(x_\mathbf{d}) ) \Big)^2 \\
 &\lesssim \sum_{\mathbf{a,b,c,d,e}} w_{\a} w_{\b} w_{\c} w_{\mathbf{d}} w_{\mathbf{e}} \chi_{[-\sigma,\sigma]}(S_{2n} \tau( g_{\mathbf{a} '\mathbf{b}}(x_\mathbf{d})) -S_{2n} \tau( g_{\mathbf{a}' \mathbf{c}}(x_\mathbf{d})) ) \chi_{[-\sigma,\sigma]}(S_{2n} \tau( g_{\mathbf{a}' \mathbf{b}}(x_\mathbf{e})) -S_{2n} \tau( g_{\mathbf{a}' \mathbf{c}}(x_\mathbf{e}) ))\\
 &\leq \sum_{\mathbf{a,b,c,d,e}} w_{\a} w_{\b} w_{\c} w_{\mathbf{d}} w_{\mathbf{e}} \chi_{[-2\sigma,2\sigma]}(S_{2n} \tau \circ g_{\mathbf{a}' \mathbf{b}}(x_\mathbf{d}) -S_{2n} \tau \circ g_{\mathbf{a}' \mathbf{c}}(x_\mathbf{d}) - S_{2n} \tau \circ g_{\mathbf{a}' \mathbf{b}}(x_\mathbf{e}) + S_{2n} \tau \circ g_{\mathbf{a}' \mathbf{c}}(x_\mathbf{e}) ).
\end{align*}

Notice then that
\begin{align*}
 S_{2n} \tau \circ g_{\mathbf{a}' \mathbf{b}}(x_\mathbf{d}) = S_{n} \tau \circ g_{\mathbf{a}}(g_{\mathbf{b}} x_\mathbf{d}) + S_n \tau \circ g_\mathbf{b}(x_\mathbf{d}).
\end{align*}
By bounded $(C,\alpha)$-distortions
\begin{align*}
 S_{2n} \tau \circ g_{\mathbf{a}' \mathbf{b}}(x_\mathbf{d}) - S_{2n} \tau \circ g_{\mathbf{a}' \mathbf{b}}(x_\mathbf{e}) = S_n \tau \circ g_{\mathbf{b}}(x_\mathbf{d}) - S_n \tau \circ g_{\mathbf{b}}(x_\mathbf{e}) + \mathcal{O}(\text{diam}(I_{\mathbf{b}})^\alpha),
\end{align*}
which gives, since $\text{diam}(I_\mathbf{b})^\alpha \lesssim \kappa_+^{\alpha n} \leq e^{-8 \varepsilon_{0} n} \leq \sigma^2$ as $\varepsilon_0 \leq \alpha |\ln \kappa_+|/8$:
\begin{align*}
 &S_{2n} \tau \circ g_{\mathbf{a}' \mathbf{b}}(x_\mathbf{d}) -S_{2n} \tau \circ g_{\mathbf{a}' \mathbf{c}}(x_\mathbf{d}) - S_{2n} \tau \circ g_{\mathbf{a}' \mathbf{b}}(x_\mathbf{e}) + S_{2n} \tau \circ g_{\mathbf{a}' \mathbf{c}}(x_\mathbf{e})\\
 &= S_{n} \tau \circ g_{\mathbf{b}}(x_\mathbf{d}) -S_{n} \tau \circ g_{\mathbf{c}}(x_\mathbf{d}) - S_{n} \tau \circ g_{\mathbf{b}}(x_\mathbf{e}) + S_{n} \tau \circ g_{\mathbf{c}}(x_\mathbf{e}) + \mathcal{O}(\sigma^2).
\end{align*}
In particular, for $n$ large enough, we find:
\begin{align*}
 &\sum_{\mathbf{a,b,c,d,e}} w_{\a} w_{\b} w_{\c} w_{\mathbf{d}} w_{\mathbf{e}} \chi_{[-2\sigma,2\sigma]}(S_{2n} \tau \circ g_{\mathbf{a} \mathbf{b}}(x_\mathbf{d}) -S_{2n} \tau \circ g_{\mathbf{a} \mathbf{c}}(x_\mathbf{d}) - S_{2n} \tau \circ g_{\mathbf{a} \mathbf{b}}(x_\mathbf{e}) + S_{2n} \tau \circ g_{\mathbf{a} \mathbf{c}}(x_\mathbf{e}) )\\
 &\leq \sum_{\mathbf{a,b,c,d,e}} w_{\a} w_{\b} w_{\c} w_{\mathbf{d}} w_{\mathbf{e}} \chi_{[-3\sigma,3\sigma]}( S_{n} \tau \circ g_{\mathbf{b}}(x_\mathbf{d}) -S_{n} \tau \circ g_{\mathbf{c}}(x_\mathbf{d}) - S_{n} \tau \circ g_{\mathbf{b}}(x_\mathbf{e}) + S_{n} \tau \circ g_{\mathbf{c}}(x_\mathbf{e}) ) \\
 &\leq \sum_{\mathbf{b,c,d,e}} w_{\b} w_{\c} w_{\mathbf{d}} w_{\mathbf{e}} \chi_{[-3\sigma,3\sigma]}( S_{n} \tau \circ g_{\mathbf{b}}(x_\mathbf{d}) -S_{n} \tau \circ g_{\mathbf{c}}(x_\mathbf{d}) - S_{n} \tau \circ g_{\mathbf{b}}(x_\mathbf{e}) + S_{n} \tau \circ g_{\mathbf{c}}(x_\mathbf{e}) )
\end{align*}
Under (QNL), this last sum is bounded by a multiple of $\sigma^\Theta + \rho^n \leq 2 \sigma^\Theta$
since $\sigma$ is slowly decaying comparing to $\rho$: $\sigma^\Theta \geq e^{-4 \varepsilon_0 \Theta n} \geq \rho^n$ since $\varepsilon_0 \leq |\ln(\rho)|/4$.
This gives the desired bound since $\sigma^{\Theta/2} = \sigma^{2 \gamma_2}$. \end{proof}

This concludes the proof of Theorem \ref{thm:mainCountable}. In the next section, we show how to check (QNL), first under (MNL) in a truly $C^{1+\alpha}$ IFS setting with finite branches, and then under (UNI) in a $C^2$-setting with infinite branches and the exponential moment condition.

\section{Proofs of the non-concentration estimates}\label{sec:nonconcentration}

\subsection{$C^{1+\alpha}$ IFSs: Using (MNL) for non-concentration}
This subsection is devoted to the proof of Theorem \ref{thm:QNL}, that is, the proof that (MNL) implies (QNL) in the setting where $|\mathcal{A}|<\infty$. To simplify, suppose that the dynamics is also fully branched: the subshift of finite type setting is similar. \\

Let us denote $0< \kappa_+, \kappa_-<1$ such that for all $a \in \mathcal{A}$, $ \kappa_+ < g_a' < \kappa_- $. Recall that we suppose that $\mu$ is the equilibrium state associated with $\delta \tau$ for $\delta := \dim_H(F_\Phi)$, so recall that in particular we have bounds:
$$ \forall \mathbf{c} \in \mathcal{A}^*, \ \mu(I_\mathbf{c}) \simeq |I_{\mathbf{c}}|^\delta $$
and, by our one-dimensional setting:

$$ \forall \mathbf{c} \in \mathcal{A}^*, \ |I_\mathbf{c}| \simeq |g_\mathbf{c}'(0)| \quad ; \quad \forall \mathbf{a},\mathbf{b} \in \mathcal{A}^*, \ |I_\mathbf{ab}| \simeq |I_\mathbf{a}|\cdot|I_\mathbf{b}|. $$

Recall we defined for $\mathbf{a} \in \mathcal{A}^n$ and $\a\to x,y$, $X_\mathbf{a}(x,y) := S_n \tau \circ g_{\mathbf{a}}(x)-S_n \tau \circ g_{\mathbf{a}}(y)$.
Notice the relation:
$$X_{\tilde{\mathbf{a}}\widehat{\mathbf{a}}}(x,y) = S_{n_1+n_2} \tau \circ g_{\tilde{\mathbf{a}}\widehat{\mathbf{a}}}(x)-S_{n_1+n_2} \tau \circ g_{\tilde{\mathbf{a}}\widehat{\mathbf{a}}}(y)$$
$$ = (S_{n_1} \tau \circ g_{\tilde{\mathbf{a}}}) (g_{\widehat{\textbf{a}}} x) - (S_{n_1} \tau \circ g_{\tilde{\mathbf{a}}}) (g_{\widehat{\textbf{a}}} y) + S_{n_2} \tau \circ g_{\widehat{\mathbf{a}}}(x)-S_{n_2} \tau \circ g_{\widehat{\mathbf{a}}}(y) $$
$$ = X_{\tilde{\mathbf{a}}}( g_{\widehat{\mathbf{a}}} x,g_{\widehat{\mathbf{a}}} y ) + X_{\widehat{\mathbf{a}}}(x,y). $$
Notice further that since $\tau \in C^\delta$, we have, uniformly in $\mathbf{a} \in \mathcal{A}^*$ :
$$ |X_\mathbf{a}(x,y)| \lesssim |x-y|^\delta.$$

For the duration of this subsection 6.1, we will suppose that (MNL) holds. This is intuitively a very strong nonconcentration/spreading statement in $x$ of the function $X_\mathbf{a}(\cdot,x_0)-X_\mathbf{b}(\cdot,x_0)$, since the pushforward by this function of the measure $\mu$ flattens it to an absolutely continuous measure. (MNL) implies in particular a more familliar \say{UNI} statement \emph{everywhere and at every scales}.

\begin{lem}
For all $N \geq 1$, there exists $\widehat{\mathbf{a}},\widehat{\mathbf{b}} \in \mathcal{A}^N$ such that for all $\textbf{c} \in \mathcal{A}^*$ and for all $y \in I_\mathbf{c} \cap F$, there exists $x \in I_\mathbf{c} \cap F$ such that
$$ |X_{\widehat{\mathbf{a}}}(x,y)-X_{\widehat{\mathbf{b}}}(x,y)| \geq \frac{1}{3 C_{\text{(MNL)}}} \mu(I_{\mathbf{c}}). $$
\end{lem}

\begin{proof}
Let $N \geq 1$ and let $\mathbf{a},\mathbf{b}$ be such that (MNL) is satisfied. Fix any $y \in F$. Now consider $I := [-\frac{1}{3 C_{\text{(MNL)}}}\mu(I_\mathbf{c}),\frac{1}{3 C_{\text{(MNL)}}}\mu(I_\mathbf{c})]$. Then
$$ \mu(x \in F, \ X_\mathbf{a}(x,y)-X_\mathbf{b}(x,y) \in I ) \leq C_{\text{(MNL)}} |I| = \frac{2}{3} \mu(I_\mathbf{c}). $$
This ensure that the set $\{x \in F, \ X_\mathbf{a}(x,y)-X_\mathbf{b}(x,y) \in I \}$ can not contain $I_\mathbf{c} \cap F$.
\end{proof}

Thanks to this seperation condition, we can show that we can always find a word $\widehat{\mathbf{a}} \in \mathcal{A}^N$ such that $X_{\widehat{\mathbf{a}}}$ can avoid some given small enough intervals. \\

\begin{lem}
Let $y \in F$. Consider some function $t:F \rightarrow \mathbb{R}$. \\ Then for all $N \geq 1$, for all $\mathbf{c} \in \mathcal{A}^*$, there exists $\widehat{\mathbf{a}} \in \mathcal{A}^N$ such that:
$$ \exists x \in I_{\mathbf{c}} \cap F, \ |X_{\widehat{\mathbf{a}}}(x,y) -t(x)| \geq \frac{1}{10 C_{\text{(MNL)}}} \mu(I_\mathbf{c}).$$
\end{lem}

\begin{proof}
Let $y \in F$, and $t:F \rightarrow \mathbb{R}$. Let $N \geq 1$ and let $\widehat{\mathbf{a}},\widehat{\mathbf{b}} \in \mathcal{A}^N$ be given by the previous Lemma. Let $\mathbf{c} \in \mathcal{A}^*$. We then know that there exists $x \in I_{\mathbf{c}}$ such that $|X_{\widehat{\mathbf{a}}}(x,y)-X_{\widehat{\mathbf{b}}}(x,y)| \geq \frac{1}{3C_{\text{(MNL)}}} \mu(I_{\mathbf{c}}).$
It follows that $|X_{\widehat{\mathbf{a}}}(x,y)-t(x)| \geq \frac{1}{10C_{\text{(MNL)}}} \mu(I_\mathbf{c})$ or $|X_{\widehat{\mathbf{b}}}(x,y)-t(x)| \geq \frac{1}{10C_{\text{(MNL)}}} \mu(I_\mathbf{c})$.
\end{proof}

\begin{lem}
Let $A \geq 1$. There exists $N := N(A) \geq 1$ large enough such that the following holds. \\
Let $t \in C^\delta(F,\mathbb{R})$ such that $\|t\|_{C^\delta(F,\mathbb{R})} \leq A$. For all $\mathbf{c} \in \mathcal{A}^*$ and $y \in I_\mathbf{c}$:
$$ \exists (\widehat{\mathbf{a}}, \widehat{\mathbf{c}}) \in (\mathcal{A}^{N})^2, \ \forall x \in I_{\mathbf{c} \widehat{\mathbf{c}}}, \ \forall \tilde{\mathbf{a}} \in \mathcal{A}^{*}, \ |X_{\tilde{\mathbf{a}} \widehat{\mathbf{a}}}(x,y)-t(x)| \geq \frac{1}{100 C_{\text{(MNL)}}} \mu(I_\mathbf{c}). $$
\end{lem}

\begin{proof}
Let $A \geq 1$. Let $N \geq 1$, we will fix it large enough during the proof depending on $A$. Let $t \in C^\delta$ be such that $\|t\|_{C^\delta} \leq A$. Let $\mathbf{c} \in \mathcal{A}^*$, and let $y \in I_\mathbf{c}$. \\

We know already that there exists $\widehat{\mathbf{a}} \in \mathcal{A}^N$ and $x_0 \in I_{\mathbf{c}}$
such that $$|X_{\widehat{\mathbf{a}}}(x_0,y)-t(x_0)| \geq \frac{1}{10 C_{\text{(MNL)}}} \mu(I_\mathbf{c}).$$
To conclude, we just need to show that this bound still holds under perturbations. Let us define $\widehat{c} \in \mathcal{A}^N$ as the word for wich $x_0 \in I_{\mathbf{c} \widehat{\mathbf{c}}}$. We have three terms to study.
\begin{enumerate}
\item First of all, for all $x \in I_{\mathbf{c} \widehat{\mathbf{c}}}$, we have a bound
$$ |t(x)-t(x_0)| \leq A |I_{\mathbf{c} \widehat{\mathbf{c}}}|^\delta \lesssim A \kappa_-^{\delta N} \mu(I_\mathbf{c}) \leq \frac{1}{100 C_{\text{(MNL)}}} \mu(I_\mathbf{c}) $$
if $N$ is large enough depending on $A$.
\item Then, for any $\tilde{\mathbf{a}} \in \mathcal{A}^*$, since $g_{\widehat{\mathbf{a}}}(x_0),g_{\widehat{\mathbf{a}}}(y) \in I_{\widehat{\mathbf{a}}\mathbf{c}}$:
$$ |X_{\tilde{\mathbf{a}}\widehat{\mathbf{a}}}(x_0,y) - X_{\widehat{\mathbf{a}}}(x_0,y)| = |X_{\tilde{\mathbf{a}}}(g_{\widehat{\mathbf{a}}}x_0{},g_{\widehat{\mathbf{a}}}y)| \lesssim |I_{\mathbf{a} \mathbf{c}}|^\delta \lesssim \kappa_-^{\delta N} \mu(I_\mathbf{c}) \leq \frac{1}{100 C_{\text{(MNL)}}} \mu(I_\mathbf{c}) $$
if $N$ is large enough depending on $A$.
\item Finally, for any $\tilde{\mathbf{a}} \in \mathcal{A}^*$ and for all $x \in I_{\mathbf{c} \widehat{\mathbf{c}}}$:
$$ |X_{\tilde{\mathbf{a}}\widehat{\mathbf{a}}}(x_0,y) - X_{\tilde{\mathbf{a}}\widehat{\mathbf{a}}}(x,y)| = |X_{\tilde{\mathbf{a}}\widehat{\mathbf{a}}}(x_0,x)| \lesssim |I_{\mathbf{c} \widehat{\mathbf{c}}}|^\delta \lesssim \kappa^{\delta N} \mu(I_\mathbf{c}) \leq \frac{1}{100 C_{\text{(MNL)}}} \mu(I_\mathbf{c}) $$
if $N$ is chosen large enough depending on $A$.
\end{enumerate}
Let us now fix $N=N(A)$ so that all three bounds holds. We can then write, for any $\tilde{\mathbf{a}} \in \mathcal{A}^*$ and $x \in I_{\mathbf{c} \widehat{\mathbf{c}}}$:
$$ |X_{\tilde{\mathbf{a}} \widehat{\mathbf{a}}}(x,y)-t(x)| \geq |X_{\widehat{\mathbf{a}}}(x_0,y)-t(x_0)| $$ $$ - \ |t(x)-t(x_0)| - |X_{\tilde{\mathbf{a}}\widehat{\mathbf{a}}}(x_0,y) - X_{\widehat{\mathbf{a}}}(x_0,y)| - |X_{\tilde{\mathbf{a}}\widehat{\mathbf{a}}}(x_0,y) - X_{\tilde{\mathbf{a}}\widehat{\mathbf{a}}}(x,y)| $$
$$ \geq \frac{1}{100 C_{\text{(MNL)}}} \mu(I_\mathbf{c}). $$
which concludes the proof. \end{proof}

Before concluding, let us recall that all the maps $X_\mathbf{a}(\cdot,y)$ are $C^\delta$, and this uniformly in the choice of $\mathbf{a}$. It follows that the previous Lemma holds for $t := X_\mathbf{a}(\cdot,y)$ with a uniform $N$. \\
For the rest of the section let us define, for any $\sigma>0$:
$$ n(\sigma) := \max \{ n \geq 1, \ \forall \mathbf{c} \in \mathcal{A}^n \ | \ \mu(I_{\mathbf{c}}) \geq 100 C_{\text{(MNL)}} \sigma \}. $$
Notice that we have $\mu(I_\mathbf{c}) \simeq |I_\mathbf{c}|^\delta \geq \kappa_-^{\delta n}$, so that $n(\sigma) \geq c_0 |\ln(\sigma)|$ for some $c_0$. \\

As an immediate Corollary, we get the following Lemma.

\begin{lem} Let $A \geq 1$. There exists $N := N(A) \geq 1$ large enough such that the following holds. \\
Let $t \in C^\delta(F,\mathbb{R})$ such that $\|t\|_{C^\delta(F,\mathbb{R})} \leq A$.
Let $\sigma>0$, and denote $I(x) := [t(x)-\sigma,t(x)+\sigma]$. For all $\mathbf{c} \in \mathcal{A}^*$ with $|\mathbf{c}| \leq n(\sigma)$, and for all $y \in I_\mathbf{c}$:
$$ \exists (\widehat{\mathbf{a}}, \widehat{\mathbf{c}}) \in (\mathcal{A}^{N})^2, \ \forall x \in I_{\mathbf{c} \widehat{\mathbf{c}}}, \ \forall \tilde{\mathbf{a}} \in \mathcal{A}^{*}, \ X_{\tilde{\mathbf{a}} \widehat{\mathbf{a}}}(x,y) \notin I(x). $$
\end{lem}

The idea now, to prove our main estimate, is to repetitively apply our previous Lemma.
Recall that $x_\mathbf{c}$ is some point in $I_\mathbf{c}$ and $w_\mathbf{c}(x) \simeq \mu(I_\mathbf{c})$.

\begin{Proposition}
Let $\sigma >0$ be small enough.
Let $t \in C^\delta(F,\mathbb{R})$ such that $\|t\|_{C^\delta}$. Fix $A := \|t\|_{C^\delta} + \sup_{\mathbf{a} \in \mathcal{A}^\delta} \sup_{y \in F} \|X_\mathbf{a}(\cdot,y) \|_{C^\delta}$. \\

Define $I(x) := [t(x)-\sigma,t(x)+\sigma]$.
We have, for all $x_0,y_0,z_0 \in F$ and for all $n \geq 1$:
$$ \sum_{\mathbf{a} \in \mathcal{A}^n} \sum_{\mathbf{c} \in \mathcal{A}^n} w_{\mathbf{a}}(x_0) w_{\mathbf{c}}(y_0) \mathbf{1}_{I(x_\mathbf{c})} \Big( X_\mathbf{a}(x_\mathbf{c},z_0) \Big) \leq C_0( \sigma^\gamma + \rho^{n}) $$
for some $\gamma,\rho \in (0,1), C_0 \geq 1$, depending on $A$.
\end{Proposition}

\begin{proof}
Fix $A := \|t\|_{C^\delta} + \sup_{\mathbf{a} \in \mathcal{A}^\delta} \sup_{y \in F} \|X_\mathbf{a}(\cdot,y) \|_{C^\delta}$, and consider $N:=N(A)$ given by the previous Lemma.
Define $m := \min(n,n(\sigma))/(2N)$. We wish to show $$ \sum_{\mathbf{a} \in \mathcal{A}^n} \sum_{\mathbf{c} \in \mathcal{A}^n} w_{\mathbf{a}}(x_0) w_{\mathbf{c}}(y_0) \mathbf{1}_{I(x_\mathbf{c})} \Big( X_\mathbf{a}(x_\mathbf{c},z_0) \Big) \leq \rho^m,$$
for some $\rho \in (0,1)$.
This will imply our desired bound, since then $$ \rho^m = \max (\rho^{n/(2N)}, \rho^{n(\sigma)/(2N)}) \leq \rho^{n/(2N)} + \rho^{n(\sigma)/N} \lesssim \rho^{n/(2N)} +\sigma^\gamma $$
for some $\gamma \in (0,1)$.

To prove this bound, we are going to use our previous Lemma $m$ times. Suppose that $n \geq N$ and that $\sigma$ is small enough so that $n(\sigma) \geq N$. First of all, we know that there exists $(\widehat{\mathbf{a}}^{(1)}, \widehat{\mathbf{c}}^{(1)}) \in (\mathcal{A}^{N})^2$ such that:
$$ \forall \tilde{\mathbf{a}} \in \mathcal{A}^{n-N}, \ \forall x \in I_{\widehat{\mathbf{c}}}, \ X_{\tilde{\mathbf{a}}\widehat{\mathbf{a}}}(x,z_0) \notin I(x). $$
We can then write:
$$\sum_{\mathbf{a}, \mathbf{c} \in \mathcal{A}^n} w_{\mathbf{a}}(x_0) w_{\mathbf{c}}(y_0) \mathbf{1}_{I(x_\mathbf{c})} \Big( X_\mathbf{a}(x_\mathbf{c},z_0) \Big) $$
$$ = \sum_{\widehat{\mathbf{a}}, \widehat{\mathbf{c}} \in \mathcal{A}^N} w_{\widehat{\mathbf{a}}}(x_0) w_{\widehat{\mathbf{c}}}(y_0) \sum_{\tilde{\mathbf{a}}, \tilde{\mathbf{c}} \in \mathcal{A}^{n-N}} w_{\tilde{\mathbf{a}}}( g_{\widehat{\mathbf{a}}}x_0) w_{\tilde{\mathbf{c}}}(g_{\widehat{\mathbf{c}}}y_0) \mathbf{1}_{I(x_{\widehat{\mathbf{c}} \tilde{\mathbf{c}}})} \Big( X_{\tilde{\mathbf{a}} \widehat{\mathbf{a}} } (x_{\widehat{\mathbf{c}} \tilde{\mathbf{c}}},z_0) \Big) $$
$$ = \underset{ (\widehat{\mathbf{a}}, \widehat{\mathbf{c}}) \neq (\widehat{\mathbf{a}}^{(1)}, \widehat{\mathbf{c}}^{(1)}) }{\sum_{\widehat{\mathbf{a}}, \widehat{\mathbf{c}} \in \mathcal{A}^{N}}} w_{\widehat{\mathbf{a}}}(x_0) w_{\widehat{\mathbf{c}}}(y_0) \Bigg( \sum_{\tilde{\mathbf{a}}, \tilde{\mathbf{c}} \in \mathcal{A}^{n-N}} w_{\tilde{\mathbf{a}}}( g_{\widehat{\mathbf{a}}}x_0) w_{\tilde{\mathbf{c}}}(g_{\widehat{\mathbf{c}}}y_0) \mathbf{1}_{I(x_{\widehat{\mathbf{c}} \tilde{\mathbf{c}}})} \Big( X_{\tilde{\mathbf{a}} \widehat{\mathbf{a}} } (x_{\widehat{\mathbf{c}} \tilde{\mathbf{c}}},z_0) \Big) \Bigg). $$
To iterate, we rewrite the event $ X_{\tilde{\mathbf{a}} \widehat{\mathbf{a}} } (x_{\tilde{\mathbf{c}}\widehat{\mathbf{c}}},z_0) \in I(x_{\tilde{\mathbf{c}}\widehat{\mathbf{c}}})$ in another manner. We have
$$ X_{\tilde{\mathbf{a}} \widehat{\mathbf{a}} } (x_{\widehat{\mathbf{c}} \tilde{\mathbf{c}}},z_0) - t(x_{\widehat{\mathbf{c}} \tilde{\mathbf{c}}}) = X_{\tilde{\mathbf{a}} } (g_{\widehat{\mathbf{a}}} x_{\widehat{\mathbf{c}} \tilde{\mathbf{c}}},g_{\widehat{\mathbf{a}}} z_0) + t_{\widehat{\mathbf{a}}}(x_{\widehat{\mathbf{c}} \tilde{\mathbf{c}}}). $$
for $t_{\widehat{\mathbf{a}}}(x) := t(x)
- X_{\widehat{\mathbf{a}}}(x,z_0)$. Notice that $\|t_{\widehat{\mathbf{a}}}\|_{C^\delta} \leq A$. We can then write
$$ \sum_{\tilde{\mathbf{a}}, \tilde{\mathbf{c}} \in \mathcal{A}^{n-N}} w_{\tilde{\mathbf{a}}}( g_{\widehat{\mathbf{a}}}x_0) w_{\tilde{\mathbf{c}}}(g_{\widehat{\mathbf{c}}}y_0) \mathbf{1}_{I(x_{\widehat{\mathbf{c}} \tilde{\mathbf{c}}})} \Big( X_{\tilde{\mathbf{a}} \widehat{\mathbf{a}} } (x_{\widehat{\mathbf{c}} \tilde{\mathbf{c}}},z_0) \Big) $$
$$ = \sum_{\tilde{\mathbf{a}}, \tilde{\mathbf{c}} \in \mathcal{A}^{n-N}} w_{\tilde{\mathbf{a}}}( g_{\widehat{\mathbf{a}}}x_0) w_{\tilde{\mathbf{c}}}(g_{\widehat{\mathbf{c}}}y_0) \mathbf{1}_{I_{\widehat{\mathbf{a}}}(x_{\widehat{\mathbf{c}} \tilde{\mathbf{c}}})} \Big( X_{\tilde{\mathbf{a}} } (g_{\widehat{\mathbf{a}}} x_{\widehat{\mathbf{c}} \tilde{\mathbf{c}}},g_{\widehat{\mathbf{a}}} z_0) \Big), $$
where $I_{\widehat{\mathbf{a}}}(x) := [t_{\widehat{\mathbf{a}}}(x) - \sigma,t_{\widehat{\mathbf{a}}}(x)+\sigma]$. We can then iterate the procedure. Applying the previous Lemma, we find that there exists $( \widehat{\mathbf{a}}^{(2)}, \widehat{\mathbf{c}}^{(2)} ) \in (\mathcal{A}^{N})^2$ (depending implicitely on $(\widehat{\mathbf{a}},\widehat{\mathbf{b}})$) such that:
$$ \forall \tilde{\mathbf{a}} \in \mathcal{A}^{n-2N}, \ \forall x \in I_{\widehat{\mathbf{a}}\widehat{\mathbf{c}} \widehat{\mathbf{c}}^{(2)}}, \ X_{\tilde{\mathbf{a}}\widehat{\mathbf{a}}^{(2)}}(x,g_{\widehat{\mathbf{a}}}z_0) \notin I_{\widehat{\mathbf{a}}}(x). $$
It follows that:
$$ \underset{ (\widehat{\mathbf{a}}, \widehat{\mathbf{c}}) \neq (\widehat{\mathbf{a}}^{(1)}, \widehat{\mathbf{c}}^{(1)}) }{\sum_{\widehat{\mathbf{a}}, \widehat{\mathbf{c}} \in \mathcal{A}^{N}}} w_{\widehat{\mathbf{a}}}(x_0) w_{\widehat{\mathbf{c}}}(y_0) \Bigg( \sum_{\tilde{\mathbf{a}}, \tilde{\mathbf{c}} \in \mathcal{A}^{n-N}} w_{\tilde{\mathbf{a}}}( g_{\widehat{\mathbf{a}}}x_0) w_{\tilde{\mathbf{c}}}(g_{\widehat{\mathbf{c}}}y_0) \mathbf{1}_{I_{\widehat{\mathbf{a}}}(x_{\widehat{\mathbf{c}} \tilde{\mathbf{c}}})} \Big( X_{\tilde{\mathbf{a}} } (g_{\widehat{\mathbf{a}}} x_{\widehat{\mathbf{c}} \tilde{\mathbf{c}}},g_{\widehat{\mathbf{a}}} z_0) \Bigg) $$
$$ = \underset{ (\widehat{\mathbf{a}}_1, \widehat{\mathbf{c}}_1) \neq (\widehat{\mathbf{a}}^{(1)}, \widehat{\mathbf{c}}^{(1)}) }{\sum_{\widehat{\mathbf{a}}_1, \widehat{\mathbf{c}}_1 \in \mathcal{A}^{N}}} w_{\widehat{\mathbf{a}}_1}(x_0) w_{\widehat{\mathbf{c}}_1}(y_0) \sum_{\tilde{\mathbf{a}}_2, \tilde{\mathbf{c}}_2 \in \mathcal{A}^{N}} w_{\widehat{\mathbf{a}}_2}(g_{\widehat{\mathbf{a}}_1} x_0) w_{\widehat{\mathbf{c}}_2}(g_{\widehat{\mathbf{c}}_1} y_0) \Bigg( $$ $$
\sum_{\tilde{\mathbf{a}}, \tilde{\mathbf{c}} \in \mathcal{A}^{n-2N}} w_{\tilde{\mathbf{a}}}(g_{\widehat{\mathbf{a}}_2 \widehat{\mathbf{a}}_1} x_0) w_{\tilde{\mathbf{c}}}(g_{\widehat{\mathbf{c}}_2 \widehat{\mathbf{c}}_1} y_0)
\mathbf{1}_{I_{\widehat{\mathbf{a}}}(x_{\widehat{\mathbf{c}} \tilde{\mathbf{c}}})} \Big( X_{\tilde{\mathbf{a}} } (g_{\widehat{\mathbf{a}}} x_{\widehat{\mathbf{c}} \tilde{\mathbf{c}}},g_{\widehat{\mathbf{a}}} z_0) \Big) \Bigg) $$
$$ = \underset{ (\widehat{\mathbf{a}}_1, \widehat{\mathbf{c}}_1) \neq (\widehat{\mathbf{a}}^{(1)}, \widehat{\mathbf{c}}^{(1)}) }{\sum_{\widehat{\mathbf{a}}_1, \widehat{\mathbf{c}}_1 \in \mathcal{A}^{N}}} w_{\widehat{\mathbf{a}}_1}(x_0) w_{\widehat{\mathbf{c}}_1}(y_0) \underset{(\widehat{\mathbf{a}}_2,\widehat{\mathbf{c}}_2) \neq (\widehat{\mathbf{a}}^{(2)},\widehat{\mathbf{c}}^{(2)}) }{\sum_{\tilde{\mathbf{a}}_2, \tilde{\mathbf{c}}_2 \in \mathcal{A}^{N}}} w_{\widehat{\mathbf{a}}_2}(g_{\widehat{\mathbf{a}}_1} x) w_{\widehat{\mathbf{c}}_2}(g_{\widehat{\mathbf{c}}_1} x) \Bigg( $$ $$
\sum_{\tilde{\mathbf{a}}, \tilde{\mathbf{c}} \in \mathcal{A}^{n-2N}} w_{\tilde{\mathbf{a}}}(g_{\widehat{\mathbf{a}}_2 \widehat{\mathbf{a}}_1} x_0) w_{\tilde{\mathbf{c}}}(g_{\widehat{\mathbf{c}}_2 \widehat{\mathbf{c}}_1} y_0)
\mathbf{1}_{I_{\widehat{\mathbf{a}}}(x_{\widehat{\mathbf{c}} \tilde{\mathbf{c}}})} \Big( X_{\tilde{\mathbf{a}} } (g_{\widehat{\mathbf{a}}} x_{\widehat{\mathbf{c}} \tilde{\mathbf{c}}},g_{\widehat{\mathbf{a}}} z_0) \Big) \Bigg).$$
The previous argument is then iterated $m$ times. After $m$ iterations, we get a nested sum that we can bound by:
$$ \underset{ (\widehat{\mathbf{a}}_1, \widehat{\mathbf{c}}_1) \neq (\widehat{\mathbf{a}}^{(1)}, \widehat{\mathbf{c}}^{(1)}) }{\sum_{\widehat{\mathbf{a}}_1, \widehat{\mathbf{c}}_1 \in \mathcal{A}^{N}}} w_{\widehat{\mathbf{a}}_1}(x_0) w_{\widehat{\mathbf{c}}_1}(y_0) \underset{ (\widehat{\mathbf{a}}_2, \widehat{\mathbf{c}}_2) \neq (\widehat{\mathbf{a}}^{(2)}, \widehat{\mathbf{c}}^{(2)}) }{\sum_{\tilde{\mathbf{a}}_2, \tilde{\mathbf{c}}_2 \in \mathcal{A}^{N}}} w_{\widehat{\mathbf{a}}_2}(g_{\widehat{\mathbf{a}}_1} x_0) w_{\widehat{\mathbf{c}}_2}(g_{\widehat{\mathbf{c}}_1} y_0) \dots $$ $$ \dots \underset{ (\widehat{\mathbf{a}}_m, \widehat{\mathbf{c}}_m) \neq (\widehat{\mathbf{a}}^{(m)}, \widehat{\mathbf{c}}^{(m)}) }{\sum_{\widehat{\mathbf{a}}_m, \widehat{\mathbf{c}}_m \in \mathcal{A}^{N}}} w_{\widehat{\mathbf{a}}_m}(g_{\widehat{\mathbf{a}}_1 \dots \widehat{\mathbf{a}}_{m-1}}x_0) w_{\widehat{\mathbf{c}}_m}(g_{\widehat{\mathbf{c}}_1 \dots \widehat{\mathbf{c}}_{m-1}}y_0). $$
Now, denoting $\rho := 1 - \inf_{\mathbf{a} \in \mathcal{A}^N} \inf_x w_\mathbf{a}(x) w_\mathbf{c}(x)$, we have $$ \underset{ (\widehat{\mathbf{a}}, \widehat{\mathbf{c}}) \neq (\widehat{\mathbf{a}}^{(1)}, \widehat{\mathbf{c}}^{(1)}) }{\sum_{\widehat{\mathbf{a}}, \widehat{\mathbf{c}} \in \mathcal{A}^{N}}} w_{\widehat{\mathbf{a}}}(x_0) w_{\widehat{\mathbf{c}}}(y_0) = {\sum_{\widehat{\mathbf{a}}, \widehat{\mathbf{c}} \in \mathcal{A}^{N}}} w_{\widehat{\mathbf{a}}}(x_0) w_{\widehat{\mathbf{c}}}(y_0) - w_{\widehat{\mathbf{a}}^{(1)}}(x_0) w_{\widehat{\mathbf{c}}^{(1)}}(y_0) $$ $$= (\mathcal{L}_\varphi^N)(1)^2 - w_{\widehat{\mathbf{a}}^{(1)}}(x_0) w_{\widehat{\mathbf{c}}^{(1)}}(y_0) = 1 - w_{\widehat{\mathbf{a}}^{(1)}}(x_0) w_{\widehat{\mathbf{c}}^{(1)}}(y_0) \leq \rho. $$
Using this bound inductively from the more nested sum to the first one then yields our desired bound:
$$ \sum_{\mathbf{a}, \mathbf{c} \in \mathcal{A}^n} w_{\mathbf{a}}(x_0) w_{\mathbf{c}}(y_0) \mathbf{1}_{I(x_\mathbf{c})} \Big( X_\mathbf{a}(x_\mathbf{c},z_0) \Big) \leq \rho^m. $$
\end{proof}

We are now ready to prove Theorem \ref{thm:QNL}.

\begin{proof}
Suppose (MNL). Then Proposition 6.1 applies. Thus, for any interval $I \subset \mathbb{R}$ of small enough diameter and for all $n$ large enough:
$$ \sum_{\mathbf{a,b,c,d}} w_{\a} w_{\b} w_{\c} w_{\mathbf{d}} \chi_{[-\sigma,\sigma]}( S_{n} \tau \circ g_{\mathbf{a}}(x_\mathbf{c}) -S_{n} \tau \circ g_{\mathbf{a}}(x_\mathbf{d}) - S_{n} \tau \circ g_{\mathbf{b}}(x_\mathbf{c}) + S_{n} \tau \circ g_{\mathbf{b}}(x_\mathbf{d}) ) $$
$$= \sum_{\mathbf{b,d}} w_{\b} w_{\mathbf{d}} \sum_{\a,\c} w_\a w_\c \mathbf{1}_{[ X_{\b}(x_\mathbf{c},x_\mathbf{d}) -\sigma,X_{\b}(x_\mathbf{c},x_\mathbf{d})+\sigma]}( X_{\a}(x_\mathbf{c},x_\mathbf{d}) ) $$
$$ \lesssim \sum_{\mathbf{b,d}} w_{\b} w_{\mathbf{d}} \Big( \sigma^\gamma + \rho^n \Big) \lesssim \sigma^\gamma + \rho^n, $$
where we used the uniform $\delta$-Hölder regularity of the maps $X_{\mathbf{b}}(\cdot,x_\mathbf{d})$ to get a uniform bound from Proposition 6.1. This is the (QNL) bound.
\end{proof}

\subsection{$C^2$ IFSs: Using UNI to prove the non-concentration}\label{sec:HolderDer}
The goal of this subsection is to prove UNI part of \ref{thm:QNL}, that is, we will show that (QNL) holds under (UNI). \\

To this aim, we will state and prove a countable $C^{1+\alpha}$ branch version of the Tree Lemma presented in \cite[Lemma 2.5.7]{leclerc2024nonlinearityfractalsfourierdecay} in the $C^2$ finite branch setting. This type of \say{tree Lemma} is useful as it allows us to avoid the use of spectral gap argument while still adapting some parts of a Dolgopyat style approach. The main difficulty of the countable compared to \cite[Lemma 2.5.7]{leclerc2024nonlinearityfractalsfourierdecay} is the possible blow-ups of derivatives caused by the parabolicity or infinite number of branches, which prevents us to use same parameters as in the finite branch case. While we cannot work with the derivatives $\partial_t \log \ab{g_\a'(t)}_{|t = x}$ directly, it turns out that the simplicity of the tree Lemma allows us to conclude the necessary conclusion with the following \textit{$\alpha$-H\"older difference quotients} of $\log \ab{g_\a'}$:
 \begin{align*}
 \la_\a^{\al}\su{x,y} &\equiv \frac{X_\a(x,y)}{|x-y|^\alpha}= \frac{\log\su{\ab{g_\a'\su{x}}} - \log\su{\ab{g_\a'\su{y}}}}{\ab{x - y}^\al},
 \end{align*}
 where $\a\in\Sigma^N$ for $N\in\N$ and $x \neq y$ with $\a \to x$ and $\a \to y$.

Recall that in our setting $\Phi$ satisfies the
\begin{itemize}
 \item[(1)]\textit{uniform $\kappa_+$-contraction}:
 \begin{align*}
 \fa ab\in \Sigma^2, x \in I_b: 0 < \kappa_-(ab) \leq |g_{ab}'(x)| \leq \kappa_+ < 1.
 \end{align*}

 \item[(2)] \textit{bounded $(C,\alpha)$-distortions}:
 \begin{align*}
 \mar{HA}\su{\Phi} &:= \underset{x,y \in I_{b
 }}{\sup_{ab \in \Sigma^2 }}
 \ab{\la_{ab}^\al\su{x,y}} < \infty.
 \end{align*}
\item[(3)] \textit{Uniformly $(c_0,\alpha)$-Non-Integrability} (UNI): there exists infinitely many $N\in\N$ such that for all $c \in \mathcal{A}$, for any two $\a, \b\in\Sigma^N$ with $b(\mathbf{a})=b(\mathbf{b})=c$, and for all $x,y \in I_{c}$ such that $x \neq y$, the following inequality holds:
 \begin{align*}
 \ab{\la_\a^\al\su{x,y} - \la_\b^\al\su{x,y}}\geq c_0.
 \end{align*}
\end{itemize}
Moreover, we chose the potential $\phi$ such that we have the following

\begin{itemize}
 \item[(4)] \textit{light $(C_{\phi,\gamma_0},\gamma_0)$-tail} some $\ga > 0$ we have $C_{\phi,\ga_0} := \sum_{ab\in \Sigma^2}\sup_{x\in I_{b}}w_a\su{x}\ab{g_a'\su{x}}^{-\gamma_0} < \infty$.
\end{itemize}

Our aim is to prove the following non-concentration property on a Cantor set defined by the $\lambda_\mathbf{a}^\alpha(x,y)$ when $\mathbf{a}$ varies:

\begin{prop}[Non-concentration of $\lambda_\mathbf{a}(x,y)$]
 \label{lms:HUNI_HAdler_tree_lemma_countable_IFS_fractal_domain}
 Let $c_0 > 0$ be the constant in the lower bound of the UNI condition, and $N\geq 0$ be sufficiently large such that it satisfies the UNI condition for $c_0$ and also that $A\kappa_+^{\al N}\leq \frac{c_0}{4}$, where $A\equiv \mar{HA}\su{\Phi} / \su{1 - \kappa_+^\al}$. Let $b \in \mathcal{A}$.

 For $n\geq 1,x\in I_b, y,z\in F\cap I_b:y\neq z,t:F^2\to \R,\si > 0$ and the parameter $\al > 0$, define the loss function $L_{x,y,z,t,\si,\al}\su{n}$
 \begin{align*}
 L_{x,y,z,t,\sigma,\al}\su{n} &\equiv \underset{b(\a)=b}{ \sum_{\a\in\Sigma^n} }w_\a\su{x}1_{\kk{t\su{x, y} - \si, t\su{x, y} + \si}}\su{\la_{\a}^\al\su{y,z}}.
 \end{align*}

 Then: $\ex \ga, \alpha_N, c_0 > 0, \kappa\in\su{0,1}:\fa t\in \R,\si > 0,n\geq 1, x\in I_b,y \neq z, y,z \in I_{b} \cap F$,
 \begin{align*}
 L_{x,y,z,t,\sigma,\al}\su{n}\leq \su{\frac{4\si}{c_0}}^{\ga} + \exp\su{-\alpha_N \su{n-N}}.
 \end{align*}
 Explicitly,

\begin{align*}
    \gamma := \min\Big\{\gamma_0,\frac{\gamma_0}{2\alpha},\frac{1}{\alpha \lambda N},\frac{\alpha \lambda N}{\eps_{\gamma_0} N\alpha^2},\frac{1}{2(\lambda N +C)}, \frac{\eps_{\mathrm{UNI}}}{4 \alpha C_{\ga',\alpha} e^{2\al\ga_0 C}C_\phi}\Big\},
\end{align*}
 where $\gamma' = \min\{\frac{\gamma_0}{2\alpha},\frac{t_{\gamma_0}}{\alpha}\}$ and $C_{\alpha,\gamma'} = -W(-\alpha \gamma')/(\alpha \gamma') > 0$ for the $W$ is the Lambert W-function and where $(C,\alpha)$ are the constants of bounded distortions, $\lambda$ the Lyapunov exponent, $C' >0 $ is a universal constant, $C_\phi > 0, 0 < t_{\gamma_0} < \gamma_0$ from Proposition \ref{prop:moments} under light $(C_{\phi,\gamma_0},\gamma_0)$ tail, and
 \begin{align*}
 \eps_{\mathrm{UNI}} := \min\se{\inf_{x \in I_{b}}w_{\a_1}\su{x}, \inf_{x\in I_{b}}w_{\b_1}\su{x}} > 0,
 \end{align*}
 for words $\a_1,\b_1 \in \Sigma^N$ satisfying the $(c_0,\alpha)$-UNI assumption of the IFS $\Phi$. Lastly, define the exponent $\al_N$ with these words $\a_1,\b_1$:
 \begin{align*}
 \al_N := \min\Big\{ \inf_{x\in [0,1]}-\frac{1}{N}\ln\Big(\sum_{\a\in\Sigma^{N}:\a\neq \c \atop \a\to x}w_{\a}\su{x}\Big) \mid \c \in \{\a_1,\b_1\} \Big\}.
 \end{align*}
\end{prop}

To prove Proposition \ref{lms:HUNI_HAdler_tree_lemma_countable_IFS_fractal_domain}, we need first the following important non-concentration on the fractal $F$ where UNI is used. Before we can state that result, let us first review the exact $\su{c_0, \alpha}$-UNI condition that we are interested in, and then introduce two useful definitions: The aforementioned $\su{c_0,\alpha}$-UNI condition was:
\textit{
\begin{itemize}
 \item[(UNI)] \textit{$(c_0,\alpha)$-Uniform Non-Integrability}: $\tau\in C^{1+\alpha}$ and there exists $c_0>0$ such that for infinitely many $N\in\N$ and every $c\in\mathcal{A}$ there are $\a,\b\in\Sigma^N$ with $b(\a)=b(\b)=c$ and, for all $x\ne y\in F\cap I_c$,
 \[
 |X_\a(x,y)-X_\b(x,y)|\ \ge\ c_0\, |x-y|^\alpha,
 \]
\end{itemize}
where
\[
X_\mathbf{a}(x,y):=S_n\tau\circ g_\mathbf{a}(x)-S_n\tau\circ g_\mathbf{a}(y)
=\log g_\a'(x)-\log g_\a'(y), \qquad \mathbf{a}\to x,y.
\]}

Next, let $c\in \cA, N\in \N, \a,\b\in \Sigma^N$ and $c_0 > 0, \alpha > 0$ be such that $b\su{\a} = b\su{\b} = c$ and $\forall x\neq y\in F\cap I_c: \ab{X_{\a}\su{x, y} - X_\b\su{x, y}}\geq c_0\ab{x - y}^{\alpha}$.

\textbf{Define} such a pair $\su{\a, \b}$ as \emph{$\su{c_0,\alpha}$-UNI pair of length $N$ at leaf $c$}. Next, given $N\in \N, c\in \cA$, define the set
\begin{align*}
    \mar{UP}\su{N, c} &\equiv \se{\su{\a, \b}\mid \text{$(\a, \b)$ is $\su{c_0,\alpha}$-UNI pair of length $N$ at leaf $c$}}.
\end{align*}



\begin{lem}[Application of UNI]
 \label{lms:HUNI_concentration_fractal_domain}
 Assume the hypotheses of Proposition \ref{lms:HUNI_HAdler_tree_lemma_countable_IFS_fractal_domain}, suppose $0 < p < \frac{c_0}{4}$ is a fixed parameter, let $t:F^2\to \R$ and assume that $\su{N_k}_{k\in\N}\subset \N$ is a sequence of integers satisfying the aforementioned $\su{c_0,\alpha}$-Uniform Non-Integrability condition. Let $k\in\N$ be fixed, but arbitrary. Then,

\begin{align*}
&\forall k\in \N:\forall c\in \cA:\exists \mathfrak{p}\in \mar{UP}\su{N_k, c}:\forall x,y\in F\cap I_c:x\neq y \im\\
&\bigg(\exists \c\su{x,y}\in \mathfrak{p}:\forall n\geq N_k:\forall \b\in \Sigma^{n-N_k}:\kk{\su{b\su{\b} = c}\land \su{\b'\c\su{x,y}\in \Sigma^{n}}}\implies\\
&\lambda_{\b'\c\su{x,y}}^\alpha\su{x,y}\not\in \kk{t\su{x, y} - p, t\su{x, y} + p}\bigg)
 \end{align*}
 
\end{lem}

To prove Lemma \ref{lms:HUNI_concentration_fractal_domain}, we first need the following Lemma that follows from the bounded distortions and uniform contraction:

\begin{lem}
 \label{lms:Hölder_Tree_Map_Word_Bound_fractal_domain}
 We have
 \begin{align*}
 \fa N\in\N: \fa \a\in \Sigma^N:\sup_{x\neq y\in I_{b(\a)}}\ab{\la_\a^\al\su{x,y}}\leq \frac{\mar{HA}\su{\Phi}\su{1 - \kappa_+^{\al N}}}{1 - \kappa_+^\al}.
 \end{align*}
\end{lem}

\begin{proof}
 The proof is done by induction on the variable $N$. The case $N = 1$ is covered by hypothesis. Assume that for $1\leq N\leq k$ we have
 \begin{align*}
 \fa N\in\N: \fa \a\in \Sigma^N:\sup_{x \neq y \in I_{b(\mathbf{a})}}\ab{\la_\a^\al\su{x,y}}\leq \mar{HA}\su{\Phi}\sum_{j=0}^{N-1}\kappa_+^{\al j}.
 \end{align*}

 Next, for any such $\c = a\b\in \Sigma^{k+1}$ with $a\in\cA$ and $\b\in\Sigma^{k}$ and $x \neq y \in I_{b(\mathbf{a})}$, the chain rule implies that
 \begin{align*}
 \la_{\c}^\al\su{x,y} &= \la_{\b}^\al\su{x,y} + \frac{\log\su{\ab{g_a'\su{g_\b\su{x}}}} - \log\su{\ab{g_a'\su{g_\b\su{y}}}}}{\ab{x - y}^\al},
 \end{align*}
 where
 \begin{align}
 \label{local_induction_step_proved_fractal_domain}
 \ab{\frac{\log\su{\ab{g_a'\su{g_\b\su{x}}}} - \log\su{\ab{g_a'\su{g_\b\su{y}}}}}{\ab{x - y}^\al}}&\leq \frac{\mar{HA}\su{\Phi}\ab{g_\b\su{x} - g_\b\su{y}}^\al}{\ab{x - y}^\al}\leq \mar{HA}\su{\Phi}\kappa_+^{\al k},
 \end{align}
 by the mean value theorem. Hence, combining the induction hypothesis with \eqref{local_induction_step_proved_fractal_domain} and taking supremum w.r.t. $x,y$ and $\c\in\Sigma^{k+1}$ shows that
 \begin{align*}
 \fa \c\in\Sigma^{k+1}:\sup_{ x \neq y \in I_{b(\mathbf{c})} }\ab{ \la_{\c}^\al\su{x,y}}&\leq \mar{HA}\su{\Phi}\sum_{j=0}^{k}\kappa_+^{\al j} = \frac{\mar{HA}\su{\Phi}\su{1 - \kappa_+^{\al \su{k+1}}}}{1 - \kappa_+^\al}.
 \end{align*}
\end{proof}

We can now prove Lemma \ref{lms:HUNI_concentration_fractal_domain}:

\begin{proof}[Proof of Lemma \ref{lms:HUNI_concentration_fractal_domain}]
 Let $c\in \mathcal{A}$. By hypothesis we know that there exists two words $\a,\b\in\Sigma^N$ such that $b(\mathbf{a}) = b(\mathbf{b}) = c$ and
 \begin{align}
 \label{initial_ineq_fractal_domain}
 \fa x \neq y \in F\cap I_c :\ab{\la_{\a}^\al\su{x,y} - \la_{\b}^\al\su{x,y}}\geq c_0.
 \end{align}
Let $x, y\in F\cap I_c$ be fixed and $t:F^2\to \R$ a map. One of the words, call it $\c\su{x, y}$, satisfies
 \begin{align*}
 \ab{\la_{\c\su{x, y}}^\al\su{x,y} - t\su{x, y}}\geq \frac{c_0}{2},
 \end{align*}
 at $\su{x, y}$. Otherwise we would arrive at a contraction in \eqref{initial_ineq_fractal_domain} after adding and subtracting $t\su{x, y}$ and using the triangle inequality. For the rest of the proof this pair $\su{x, y}$ is fixed and so we write $\c := \c\su{x,y}$.
 
 We are ultimately looking to show that the inequality
 \begin{align*}
    \ab{\la_{\b'\c}^\al\su{x,y} - t\su{x,y}}\geq \frac{c_0}{4}.
 \end{align*}
 holds for any prefix $\b\in\Sigma^{n-N}$ at $\su{x, y}$. Let $\b\in\Sigma^{n-N}$ be fixed, but arbitrary. By the chain rule we obtain that
 \begin{align*}
 \la_{\b'\c}^\al\su{x,y} &= \la_{\c}^\al\su{x,y} + \frac{\log\su{\ab{g_\b'\su{g_\c\su{x}}}} - \log\su{\ab{g_\b'\su{g_\c\su{y}}}}}{\ab{x - y}^\al},
 \end{align*}
 and by applying exactly the same kind of argument as in the proof of Lemma \ref{lms:Hölder_Tree_Map_Word_Bound_fractal_domain}, we get
 \begin{align*}
 \ab{\frac{\log\su{\ab{g_{\b}'\su{g_{\c}\su{x}}}} - \log\su{\ab{g_{\b}'\su{g_{\c}\su{y}}}}}{\ab{x - y}^\al}}&\leq \frac{A\ab{g_\c\su{x} - g_\c\su{y}}^\al}{\ab{x - y}^\al} = A\ab{g_\c'\su{z}}^\al\leq A\kappa_+^{N\al}\leq \frac{c_0}{4},
 \end{align*}
 for some $z\in\su{x,y}\subset I_c$. Consequently, we obtain:
 \begin{align*}
 \ab{\la_{\b'\c}^\al\su{x,y} - t\su{x,y}} &\geq \ab{\la_{\c}^\al\su{x,y} - t\su{x,y}} - \ab{\frac{\log\su{g_{\b}'\su{g_{\c}\su{x}}} - \log\su{g_{\b}'\su{g_{\c}\su{y}}}}{\su{x - y}^\al}}\\
 \ab{\la_{\b'\c}^\al\su{x,y} - t\su{x,y}} &\geq \frac{c_0}{2} - A\kappa_+^{N\al}\geq \frac{c_0}{4} \im \la_{\b'\c}^\al\su{x,y} \not\in \kk{t\su{x,y} - p, t\su{x,y} + p},
 \end{align*}
 for the chosen $x, y\in F\cap I_{c}$, with $\b\in \Sigma^{n-N}$ s.t. the word $\b'\c$ is admissible, $ n > N$, and $ p\in \su{0,\frac{c_0}{4}}$.

\end{proof}

With Lemma \ref{lms:HUNI_concentration_fractal_domain}, we can finish the proof of Proposition \ref{lms:HUNI_HAdler_tree_lemma_countable_IFS_fractal_domain} in the following:

\begin{proof}[Proof of Proposition \ref{lms:HUNI_HAdler_tree_lemma_countable_IFS_fractal_domain}]

		The proof is done by induction on the variable $n$. Write $A\equiv \mar{HA}\su{\Phi}$. Let us first note that the loss function $L_{x,y,z,t,\sigma,\al}\su{n}$ has the following autosimilarity property:


  \begin{align*}
    L_{x,y,z,t,\sigma,\al}\su{n} &= \sum_{\a\in\Sigma^{k} \atop \a \to x,y,z}w_\a\su{x} L_{x\su{\a},y\su{\a},z\su{\a},t\su{\a,\al,\Phi}\su{.,.},\sigma\su{\a,y,z,\al},\al}\su{n - k},
 \end{align*}

 for any $1\leq k\leq n-1$. This can be seen as follows: Fix $1\leq k\leq n-1$ and partition a word $\c\in\Sigma^{n}$ into a suffix $\a\in\Sigma^{k}$ and a prefix $\b\in\Sigma^{n-k}$, so that $\c = \b'\a$.

 Next the inclusion condition
		\begin{align*}
			\la_{\b'\a}^\al\su{y,z}\in\kk{t\su{y,z} - \si, t\su{y,z} + \si},
		\end{align*}

 for $y,z$ such that $\b'\a\to y,z$ can be rewritten \say{in terms of a fixed $\a\in\Sigma^{k}$} with the use of the mean value theorem as follows:
 \begin{align*}
 \la_{\b'\a}^\al\su{y,z} &= \la_\a^\al\su{y,z} + \frac{\log\su{\ab{g_{\b}'\su{g_{\a}\su{y}}}} - \log\su{\ab{g_{\b}'\su{g_{\a}\su{z}}}}}{\ab{y - z}^\al}\\
 &= \la_\a^\al\su{y,z} +\la_\b^\al\su{g_\a\su{y},g_\a\su{z}}\frac{\ab{g_\a\su{y} - g_\a\su{z}}^\al}{\ab{y - z}^\al}\\
 &= \la_\a^\al\su{y,z} +\la_\b^\al\su{g_\a\su{y},g_\a\su{z}}\ab{g_\a'\su{\eta\su{\a, y,z}}}^\al,
 \end{align*}
 for some $\eta\su{\a, y,z}\in\su{z,y} \subset I_{b(\mathbf{a})}$. Hence,
		\begin{align*}
			t\su{y,z} - \si &\leq \la_\a^\al\su{y,z} +\la_\b^\al\su{g_\a\su{y},g_\a\su{z}}\ab{g_\a'\su{\eta\su{\a, y,z}}}^\al \leq t\su{y,z} + \si\Lo\\
			\frac{t\su{y,z} - \la_{\a}^\al\su{y,z} - \si}{\ab{g_\a'\su{\eta\su{\a, y,z}}}^\al} &\leq \la_{\b}^\al\su{g_{\a}\su{y},g_{\a}\su{z}} \leq \frac{t\su{y,z} - \la_{\a}^\al\su{y,z} + \si}{\ab{g_\a'\su{\eta\su{\a, y,z}}}^\al}.
		\end{align*}
 Hence, we can set
		\begin{align*}
			t\su{\a,\al,\Phi}\su{y,z} &\equiv \frac{t\su{y, z} - \la_{\a}^\al\su{y,z}}{\ab{g_\a'\su{\eta\su{\a, y,z}}}^\al},\ \si\su{\a,y,z,\al} \equiv \frac{\si}{\ab{g_\a'\su{\eta\su{\a, y,z}}}^\al},\\
 x\su{\a} &\equiv g_\a\su{x},\ y\su{\a} \equiv g_\a\su{y}, \ z\su{\a} \equiv g_\a\su{z}.
		\end{align*}
 Thus,
 \begin{align*}
 \la_{\b}^\al\su{g_{\a}\su{y},g_{\a}\su{z}} \in \kk{t\su{\a,\al,\Phi}\su{.,.} - \si\su{\a,y,z,\al}, t\su{\a,\al,\Phi}\su{.,.} + \si\su{\a,y,z,\al}},
 \end{align*}
 and consequently
 \begin{align*}
 & L_{x,y,z,t,\sigma,\al}\su{n} = \sum_{\c\in\Sigma^{n}\atop \c \to x,y,z}w_\c\su{x}1_{\kk{t\su{y, z} - \si, t\su{y, z} + \si}}\su{\la_{\c}^\al\su{y,z}}\\
 &= \sum_{\a\in\Sigma^{k}\atop \a \to x,y,z}w_\a\su{x}\sum_{\b\in\Sigma^{n-k} \atop \b \to g_{\a}(x)}w_{\b}\su{g_\a\su{x}}1_{\kk{t\su{\a,\al,\Phi}\su{y, z} - \si\su{\a,y,z,\al}, t\su{\a,\al,\Phi}\su{y, z} + \si\su{\a,y,z,\al}}}\su{\la_{\b}\su{g_{\a}\su{y},g_{\a}\su{z}}}\\
 &= \sum_{\a\in\Sigma^{k}\atop \a \to x,y,z}w_\a\su{x} L_{x\su{\a},y\su{\a},z\su{\a},t\su{\a,\al,\Phi}\su{.,.},\sigma\su{\a,y,z,\al},\al}\su{n - k}.
 \end{align*}
 Next, let us consider the different cases of the induction proof:

 \textbf{Case 1.} If $n\leq N$, then it follows that
		\begin{align*}
			L_{x,y,z,t,\sigma,\al}\su{n} \leq \sum_{\a\in\Sigma^{n}\atop \a \to x,y,z}w_\a\su{x} = 1,
		\end{align*}
		and as $n - N\leq 0$ and $e^{-\alpha_N} < 1$, we have that $\exp\su{-\alpha_N \su{n-N}} \geq 1$ and so
 \begin{align*}
 L_{x,y,z,t,\sigma,\al}\su{n} &\leq \su{\frac{4\si}{c_0}}^{\ga} + \exp\su{-\alpha_N \su{n-N}}.
 \end{align*}

 We will then use Lemma \ref{lms:HUNI_concentration_fractal_domain} to cut unnecessary branches of the tree.

	 \textbf{Case 2.} If $N + 1\leq n$, there are two subcases to consider.

 \textbf{Case 2(a).} First, if $\si \geq \frac{c_0}{4}$, then $\frac{4}{c_0}\si\geq \frac{1}{\si}\si = 1$. As the weights $w_\a\su{x}$ sum to one w.r.t. $\a\in\Sigma^{N}$, we get the trivial bound
 \begin{align*}
 L_{x,y,z,t,\sigma,\al}\su{n}\leq 1 \leq \su{\frac{4\si}{c_0}}^{\ga} + \exp\su{-\alpha_N \su{n-N}}.
 \end{align*}

		\textbf{Case 2(b).} Assume that $\si < \frac{c_0}{4}$. Due to proof technical reasons we will proceed iteratively by showing that the induction hypothesis applies for all $n$ such that $pN + 1\leq n \leq \su{p+1}N$ for $p = 1,2,\dots$. This in turn will conclude the claim for all $n\in \N$.

 So suppose first that $N + 1 \leq n \leq 2N$. The autosimilar property of the loss function yields for $k = N$ that
 \begin{align*}
 L_{x,y,z,t,\si,\al}\su{n} &= \sum_{\a\in \Si^N\atop \a\to x,y,z}w_\a\su{x}L_{x\su{\a},y\su{\a},z\su{\a},t\su{\a,\al,\Phi}\su{.,.},\si\su{\a,y,z,\al},\al}\su{n-N}.
 \end{align*}
 We will next use Lemma \ref{lms:HUNI_concentration_fractal_domain} to cut unnecessary branches of the tree. Let $\c$ be such a suffix guaranteed by the Lemma \ref{lms:HUNI_concentration_fractal_domain} so that
 \begin{align}
 \label{tree_lemma_autosimilarity_reduction_eq}
 L_{x,y,z,t,\si,\al}\su{n} &= \sum_{\a\in \Si^N:\a\neq \c\atop \a\to x,y,z}w_\a\su{x}L_{x\su{\a},y\su{\a},z\su{\a},t\su{\a,\al,\Phi}\su{.,.},\si\su{\a,y,z,\al},\al}\su{n-N}.
 \end{align}




 Since $n - N\leq N$, the previously written steps allow us to apply the induction hypothesis
 \begin{align*}
 L_{x,y,z,t,\sigma,\al}\su{s} &\leq \su{\frac{4\si}{c_0}}^{\ga} + \exp\su{-\alpha_N \su{s-N}},
 \end{align*}
 to the value $s = n - N$, so that
		\begin{align}
 \label{tree_lemma_induction_hypothesis_applied}
			 L_{x,y,z,t,\sigma,\al}\su{n} &\leq \sum_{\a\in\Sigma^{N}:\a\neq\c \atop \a \to x,y,z }w_{\a}\su{x}\su{\su{\frac{4\si\su{\a,y,z,\al}}{c_0}}^{\ga} + \exp(-\alpha_N(n-2N))}.
		\end{align}
 Since $\si\su{\a,y,z,\al} = \si / \ab{g_\a'\su{\eta\su{\a, y,z}}}^\al$, we arrive to
		\begin{align}
			\label{ineqs:almost_last_ineq2_fractal_domain}
			 L_{x,y,z,t,\sigma,\al}\su{n} &\leq \sum_{\a\in\Sigma^{N}:\a\neq\c\atop \a \to x,y,z }w_{\a}\su{x}\su{\su{\frac{4\si}{c_0}}^{\ga} \ab{g_\a'\su{\eta\su{\a, y,z}}}^{-\al\ga} + \exp(-\alpha_N(n-2N))}.
		\end{align}
 where the sum can in the RHS of \eqref{ineqs:almost_last_ineq2_fractal_domain} can be bounded by the bounded $(C,\alpha)$-distortions by
 \begin{align*}
 \sum_{\a\in\Sigma^{N}:\a\neq\c\atop \a \to x,y,z }w_{\a}\su{x}\su{\su{\frac{4\si}{c_0}}^{\ga} e^{\al\ga C}\ab{g_\a'\su{x}}^{-\al\ga} + \exp\su{-\alpha_N\su{n-2N}}}.
		\end{align*}
 For any $0 < \gamma \leq \gamma_0/\alpha$, denote
 \begin{align*}
 \mal{Q}_{N,x}\su{\ga}\equiv e^{\al\ga_0 C} \sum_{\a\in\Sigma^N\atop \a \to x}w_{\a}\su{x}\ab{g_\a'\su{x}}^{-\al\ga} = e^{\al\ga \lambda N + \al\ga_0 C} \sum_{\a\in\Sigma^N\atop \a \to x} w_{\a}\su{x} (e^{\lambda N}\ab{g_\a'\su{x}})^{-\al\ga}.
 \end{align*}
 Next, an application of Taylor's theorem shows that
 \begin{align*}
 \fa \ga'\in\su{0,1}:\exists C_{\ga'} > 0:\fa x\in\ka{1,\infty}:\ln\su{x}\leq C_{\alpha,\ga'}x^{\alpha \ga'},
 \end{align*}
 where $C_{\alpha,\gamma'} = -W(-\alpha \gamma')/(\alpha \gamma') > 0$ and $W$ is the Lambert W-function. In turn, this shows that $-\ln\su{\ab{g_\a'\su{x}}}\leq C_{\alpha,\ga'}\ab{g_\a'\su{x}}^{-\alpha \ga'}$ so that
 \begin{align*}
 w_\a\su{x}\su{-\al\ln\su{\ab{g_\a'\su{x}}}}\ab{g_\a'\su{x}}^{-\al\ga}\leq \al C_{\ga',\alpha}w_\a\su{x}\ab{g_\a'\su{x}}^{-\al(\ga + \ga')}.
 \end{align*}
 Hence, as long as $\max\{\gamma,\gamma'\} \leq \frac{\gamma_0}{2\alpha}$, then $\al(\ga + \ga')\leq \ga_0$, which shows that $\mal{Q}_{N,x}\su{\ga}$ is differentiable w.r.t. $\ga$, and we can bound the derivative
 \begin{align*}
 \mathcal{Q}_{N,x}'(\gamma)=e^{\al\ga_0 C} \sum_{\mathbf{a} \rightarrow x} w_{\mathbf{a}}(x) (-\alpha \ln |g_{\mathbf{a}}'(x)|) |g_{\mathbf{a}}'(x)|^{-\alpha \gamma} \leq \alpha C_{\ga',\alpha} e^{\al\ga_0 C}\mal{Q}_{N,x}(\gamma + \gamma').
 \end{align*}
 Now, by Proposition \ref{prop:moments}, as long as $\max\{\gamma,\gamma'\} \leq t_{\gamma_0}/\alpha$ we have
 \begin{align*}\mal{Q}_{N,x}(\gamma+\gamma') &= e^{\al(\ga+\ga') \lambda N + \al\ga_0 C} \sum_{\a\in\Sigma^N\atop \a \to x} w_{\a}\su{x} (e^{\lambda N}\ab{g_\a'\su{x}})^{-\al(\ga+\ga')}\\
 & \leq e^{\al(\ga+\ga') \lambda N} 2C_\phi e^{\eps_{\gamma_0} N \alpha^2 (\ga+\ga')^2} \\
 & \leq 2C_\phi e^{\al\ga_0 C}
 \end{align*}
 as long as $\max\{\gamma,\gamma'\} \leq \min\{\frac{1}{\alpha \lambda N},\frac{\alpha \lambda N}{\eps_{\gamma_0} N\alpha^2}\}$ since then $\al(\ga+\ga') \lambda N + \eps_{\gamma_0} N \alpha^2 (\ga+\ga')^2 \leq 1$. Thus by the mean value theorem
 $$\mal{Q}_{N,x}(\gamma) \leq 1+\gamma \mal{Q}_{N,x}'(\gamma) \leq 1+2\gamma \alpha C_{\ga',\alpha} e^{2\al\ga_0 C}C_\phi.$$
 Now recall that $\c$ in Lemma \ref{lms:HUNI_concentration_fractal_domain} is one of the words $\a_1,\b_1 \in \Sigma^N$ in the $(c_0,\alpha)$-UNI assumption of the IFS $\Phi$. Since
 $$\eps_{\mathrm{UNI}} \leq w_\c(x)$$
 and with the choices $\gamma':= \min\{\frac{\gamma_0}{2\alpha},\frac{t_{\gamma_0}}{\alpha}\}$ and
 $$\gamma := \min\Big\{\gamma_0,\frac{\gamma_0}{2\alpha},\frac{1}{\alpha \lambda N},\frac{\alpha \lambda N}{\eps_{\gamma_0} N\alpha^2},\frac{1}{2(\lambda N +C)}, \frac{\eps_{\mathrm{UNI}}}{4 \alpha C_{\ga',\alpha} e^{2\al\ga_0 C}C_\phi}\Big\}.$$
we have that
$$\mal{Q}_{N,x}\su{\ga} \leq 1+w_\c(x)$$
so by defining
 \begin{align*}
 \mal{R}\su{\ga,N,\c}\equiv \sum_{\a\in\Sigma^N:\a\neq\c\atop \a \to x,y,z}e^{\al\ga C}w_{\a}\su{x}\ab{g_\a'\su{x}}^{-\al\ga}
 \end{align*}
 we have
 \begin{align}
 \mal{R}\su{\ga,N,\c} < 1+ w_\c(x) - e^{\al\ga C}w_\c\su{x}\ab{g_\c'\su{x}}^{-\al\ga}\leq 1 + w_\c\su{x}\su{1 - e^{\al\ga C} \ab{g_\c'\su{x}}^{-\al\ga}}.
 \end{align}
 Since $\ab{g_\c'\su{x}} < 1$, it follows that we have $1 - e^{\al\ga C} \ab{g_\c'\su{x}}^{-\al\ga} < 0,$
 so that $ w_\c\su{x}\su{1 - \ab{g_\c'\su{x}}^{-\al\ga}} < 0$ so for our choice of $\gamma$ we indeed have have:
 \begin{align}
 \label{ineqs:sigma_ineq_fractal_domain}
			 \su{\frac{4\si}{c_0}}^{\ga}\sum_{\a\in\Sigma^{N}:\a\neq\c\atop \a \to x,y,z }e^{\al\ga C}w_{\a}\su{x}\ab{g_\a'\su{x}}^{-\al\ga} &\leq \su{\frac{4\si}{c_0}}^{\ga}.
		\end{align}
 Lastly recall that
 \begin{align*}
 \al_N &\leq \al_N\su{\c} \equiv \inf_{x\in [0,1]}-\frac{1}{N}\ln\su{\sum_{\a\in\Sigma^{N}:\a\neq \c \atop \a\to x}w_{\a}\su{x}},
 \end{align*}
 so that,
 $$\sum_{\a\in\Sigma^{N}:\a\neq \c\atop \a \to x }w_{\a}\su{x} \leq \exp\su{-\al_N N},$$
 and so
 \begin{align*}
 \sum_{\a\in\Sigma^{N}:\a\neq\c\atop \a \to x,y,z }w_{\a}\su{x}\su{\exp(-\alpha_N(n-2N))}&\leq \exp\su{-\al_N\su{n-N}},
	\end{align*}
 which together with \eqref{ineqs:sigma_ineq_fractal_domain} shows that
 \begin{align*}
 \sum_{\a\in\Sigma^{N}:\a\neq\c\atop \a \to x,y,z }w_{\a}\su{x}\su{\su{\frac{4\si}{c_0}}^{\ga} e^{\al\ga C}\ab{g_\a'\su{x}}^{-\al\ga} + \exp\su{-\alpha_N\su{n-2N}}} &\leq \su{\frac{4\si}{c_0}}^{\ga} + \exp\su{-\al_N\su{n-N}}.
 \end{align*}
 Thus
 \begin{align*}
 L_{x,y,z,t,\sigma,\al}\su{n} &\leq \su{\frac{4\si}{c_0}}^{\ga} + \exp\su{-\al_N\su{n-N}},
 \end{align*}
 for all $n\in \N:N+1\leq n\leq 2N$.

 Next, let us note that the precise part where we had to use the assumption $N + 1\leq n \leq 2N$ was when we used the induction hypothesis to move from the equation \eqref{tree_lemma_autosimilarity_reduction_eq} to the inequality \eqref{tree_lemma_induction_hypothesis_applied}. Crucial to this application was that $n - N\leq B$ for some $B$ such that the induction hypothesis applies for all $k = 1,2,\dots,B$. Beyond this point, the assumption $N + 1\leq n\leq 2N$ played no part in rest of the proof. Therefore, now that the induction hypothesis applies for all $k = 1,\dots,2N$, we can redo the proof of the part b of the case 2 for $n$ s.t. $pN + 1\leq n\leq \su{p+1}N$ for $p = 2,3,\dots$. Hence the induction hypothesis applies for all $n\in \N$, which proves the claim.
\end{proof}

We now conclude this subsection by proving Theorem \ref{thm:QNL}, using the non-concentration estimate given by Proposition \ref{lms:HUNI_HAdler_tree_lemma_countable_IFS_fractal_domain}.

\begin{proof}
Suppose (UNI). Then Proposition \ref{lms:HUNI_HAdler_tree_lemma_countable_IFS_fractal_domain} holds. Let us now reduce (QNL) to this nonconcentration bound. Let $I \subset \mathbb{R}$ be some interval. Consider the sum:

$$ \sum_{\mathbf{b,c,d,e}} w_{\b} w_{\c} w_{\mathbf{d}} w_{\mathbf{e}} \chi_{[-\sigma,\sigma]}( S_{n} \tau \circ g_{\mathbf{b}}(x_\mathbf{d}) -S_{n} \tau \circ g_{\mathbf{b}}(x_\mathbf{e}) - S_{n} \tau \circ g_{\mathbf{c}}(x_\mathbf{d}) + S_{n} \tau \circ g_{\mathbf{c}}(x_\mathbf{e}) ) $$
$$ = \sum_{\mathbf{b,c,d,e}} w_{\b} w_{\c} w_{\mathbf{d}} w_{\mathbf{e}} \chi_{[-\sigma,\sigma]}\Big( \big(\lambda_{\b}^\alpha(x_{\mathbf{d}},x_{\mathbf{e}}) - \lambda_{\c}^\alpha(x_{\mathbf{d}},x_{\mathbf{e}}) \big) \big| x_\mathbf{d} - x_{\mathbf{e}} \big|^\alpha \Big) $$

Define a neighbourhood of the diagonal in the following way:
\begin{align*}
 \mathcal{D}_n(\sigma^{1/2}) := \{ (\mathbf{d},\mathbf{e}) \in \Sigma^{2n}, \ |x_\mathbf{d}-x_{\mathbf{e}}| \leq \sigma^{1/2} \}.
\end{align*}
Notice that, for every $(x,y) \in I_{\mathbf{d}} \times I_{\mathbf{e}}$ with $(\mathbf{d},\mathbf{e}) \in \mathcal{D}_n(\sigma^{1/2})$, we have
\begin{align*}
 |x-y| \leq |x-x_\mathbf{d}| + |y-x_\mathbf{e}| + |x_\mathbf{d}-x_{\mathbf{e}}| \lesssim 2 \kappa_+^n + \sigma^{1/2}.
\end{align*}
Summing characteristic functions then gives:
\begin{align*}
 \sum_{(\mathbf{d},\mathbf{e}) \in \mathcal{D}_n(\sigma^{1/2})} \chi_{I_\mathbf{d} \times I_{\mathbf{e}}} \leq \chi_{\{ (x,y), \ |x-y| \leq 2 \kappa_+^n + \sigma^{1/2} \}}.
\end{align*}
Upper regularity of the measure $\mu$ and Gibbs estimates then yields:
\begin{align*}
 &\sum_{(\mathbf{d},\mathbf{e}) \in \mathcal{D}_n(\sigma^{1/2})} w_{\mathbf{d}} w_{\mathbf{e}} \leq \iint_{I \times I} \sum_{(\mathbf{d},\mathbf{e}) \in \mathcal{D}_n(\sigma^{1/2})} \chi_{I_\mathbf{d} \times I_{\mathbf{e}}} d\mu^2(x,y)\\
 &\leq \mu^{\otimes 2}\big( (x,y) \in I^2, \ |x-y| \leq 2 \kappa_+^n + \sigma^{1/2} \big) = \int_{I} \mu(B(x,2 \kappa_+^n + \sigma^{1/2}) d\mu(x) \lesssim \kappa_+^{s_\mu n} + \sigma^{s_\mu/2}.
\end{align*}

We can then conclude by throwing away the diagonal part, and then bounding the off-diagonal part. For the diagonal, we find:
\begin{align*}
 \sum_{\su{\bd,\be}\in\mal{D}_n\su{\si^{1/2}}}w_{\mathbf{d}}w_{\mathbf{e}}\sum_{\b,\c}w_{\b} w_{\c} \chi_{[-\sigma,\sigma]}\Big( \big(\lambda_{\b}^\alpha(x_{\mathbf{d}},x_{\mathbf{e}}) - \lambda_{\c}^\alpha(x_{\mathbf{d}},x_{\mathbf{e}}) \big) \big| x_\mathbf{d} - x_{\mathbf{e}} \big|^\alpha \Big) \lesssim \kappa_+^{s_\mu n} + \sigma^{s_\mu/2}.
\end{align*}
It follows that we get a bound
$$ \sum_{\mathbf{b,c,d,e}} w_{\b} w_{\c} w_{\mathbf{d}} w_{\mathbf{e}} \chi_{[-\sigma,\sigma]}\Big( \big(\lambda_{\b}^\alpha(x_{\mathbf{d}},x_{\mathbf{e}}) - \lambda_{\c}^\alpha(x_{\mathbf{d}},x_{\mathbf{e}}) \big) \big| x_\mathbf{d} - x_{\mathbf{e}} \big|^\alpha \Big) $$
$$ \lesssim \kappa_+^{s_\mu n} + \sigma^{s_\mu/2} + \sum_{\su{\bd,\be}\notin\mal{D}_n\su{\si^{1/2}}}w_{\mathbf{d}}w_{\mathbf{e}}\sum_{\b,\c}w_{\b} w_{\c} \chi_{[-\sigma,\sigma]}\Big( \big(\lambda_{\b}^\alpha(x_{\mathbf{d}},x_{\mathbf{e}}) - \lambda_{\c}^\alpha(x_{\mathbf{d}},x_{\mathbf{e}}) \big) \big| x_\mathbf{d} - x_{\mathbf{e}} \big|^\alpha \Big) $$
$$ \lesssim \kappa_+^{s_\mu n} + \sigma^{s_\mu/2} + \sum_{\mathbf{b,c,d,e}} w_{\b} w_{\c} w_{\mathbf{d}} w_{\mathbf{e}} \chi_{[-\sigma^{1/2},\sigma^{1/2}]} \big(\lambda_{\b}^\alpha(x_{\mathbf{d}},x_{\mathbf{e}}) - \lambda_{\c}^\alpha(x_{\mathbf{d}},x_{\mathbf{e}}) \big) $$
To bound this second sum, we write
\begin{align*}
 &\sum_{\b,\c}w_{\b} w_{\c} \chi_{\kk{-\si^{1/2},\si^{1/2}}}\su{\la_{\b}\su{x_\bd,x_{\be}} - \la_{\c}\su{x_\bd,x_{\be}}} =\\ &\sum_{\c}w_{\c}\su{x}\sum_{\b}w_{\b} \chi_{\kk{\la_{\c}\su{x_\bd,x_{\be}}-\si^{1/2},\la_{\c}\su{x_\bd,x_{\be}} + \si^{1/2}}}\su{\la_{\b}\su{x_\bd,x_{\be}}}.
\end{align*}
Notice then that by Proposition \ref{lms:HUNI_HAdler_tree_lemma_countable_IFS_fractal_domain} we obtain for some $\ga > 0, c_0 > 0$:
\begin{align*}
 &\sum_{\b}w_{\b} \chi_{\kk{\la_{\c}\su{x_\bd,x_{\be}}-\si^{1/2},\la_{\c}\su{x_\bd,x_{\be}} + \si^{1/2}}}\su{\la_{\b}\su{x_\bd,x_{\be}}}\\
 &\leq \su{\frac{4 \si^{1/2}}{c_0}}^{\ga} + \exp\su{-\al_N\su{n - N}},
\end{align*}
where $N\in\N$ is a large enough constant such that it satisfies the UNI condition for the IFS $\Phi$.
Thus,
\begin{align*}
 &\sum_{\c}w_{\c}\sum_{\b}w_{\b} \chi_{\kk{\la_{\c}\su{x_\bd,x_{\be}}-\si^{1/2},\la_{\c}\su{x_\bd,x_{\be}} + \si^{1/2}}}\su{\la_{\b}\su{x_\bd,x_{\be}}}\\
 &\leq \su{\frac{12\si^{1/2}}{c_0}}^{\ga} + e^{{-\al_N \su{n - N}}}.
\end{align*}

Thus we can bound
\begin{align*}
 &\sum_{\bd,\be}w_{\bd}w_{\be}\sum_{\b,\c}w_{\b}w_{\c} \chi_{\kk{-3\si^{1/2},3\si^{1/2}}}\su{\la_{\b}\su{x_\bd,x_{\be}} - \la_{\c}\su{x_\bd,x_{\be}}}\\
 &\lesssim \kappa_+^{s_\mu n} + \sigma^{s_\mu/2} + \sum_{\su{\bd,\be}\not\in\mal{D}_n\su{\si^{1/2}}} w_{\bd} w_{\be} \sum_{\b,\c}w_{\b}w_{\c}\chi_{\kk{-\si^{1/2},\si^{1/2}}}\su{\la_{\b}\su{x_\bd,x_{\be}} - \la_{\c}\su{x_\bd,x_{\be}}}\\
 &\lesssim \kappa_+^{s_\mu n} + \sigma^{s_\mu/2} + \sum_{\su{\bd,\be}\not\in\mal{D}_n\su{\si^{1/2}}}w_{\bd}w_{\be} \ (\sigma^{\gamma/2} + e^{-\alpha_N (n-N)} )\lesssim\si^{\min(s_\mu,\gamma)/2} + \max(e^{-\alpha_N},\kappa_+^{s_\mu})^n.
\end{align*}
we thus proved (QLN) with constants $\Theta := \min(s_\mu,\gamma)/2 >0$ and $\rho := \max(e^{-\alpha_N},\kappa_+^{s_\mu}) \in (0,1)$.
\end{proof}

\section{Proof of Theorem \ref{thm:lyons} on the Lyons measure} \label{sec:lyons}

In this section we will prove Theorem \ref{thm:lyons} by inducing approach and using Theorem \ref{thm:mainCountable}. Fix $t\in(0,1)$. The Lyons IFS is the parabolic Möbius system
\[
\Phi=\Big\{f_t(x)=\frac{x+t}{x+t+1},\qquad
f_0(x)=\frac{x}{x+1}\Big\},\qquad x\in[0,1].
\]
With equal weights $(1/2,1/2)$, its invariant measure $\nu_t$ is Lyons conductance measure with the law
of a random continued fraction with i.i.d. digits in $\{0,t\}$, see
\cite{Lyons,simon2001invariant,mihailescu2016random}.
As in \cite{mihailescu2016random}, we accelerate along $\phi_2$ and obtain the
induced uniformly expanding countable IFS
\[
\tilde\Phi=\{\psi_n^t=f_0^n \circ f_t :\ n\in\N_0\},
\qquad
\psi_n^t(x)=\frac{x+t}{(n+1)(x+t)+1}.
\]
Its branches satisfy
\[
(\psi_n^t)'(x)=\frac{1}{\big((n+1)(x+t)+1\big)^2}<1,
\qquad x\in[0,1],
\]
and the invariance relation expands as
\[
\nu_t=\frac12\,(\phi_1^t)_*\nu_t+\frac12\,(\phi_2)_*\nu_t
\quad\Longrightarrow\quad
\nu_t=\frac12\sum_{n=0}^\infty 2^{-n}(\psi_n^t)_*\nu_t.
\]
Thus $\nu_t$ is a projection of a Gibbs measure on $\N^\N$ associated to a locally constant potential $\phi(\a) := \log 2^{-a_1}$ on $\N^\N$. Thus by Lemma \ref{lma:noatoms} $\nu_t$ has no atoms so we are in realm where Theorem \ref{thm:mainCountable} can be applied if we verify the UNI condition (so QNL holds by Theorem \ref{thm:QNL}) and light tail.

Let us now verify the UNI condition for the induced IFS.

\begin{lem}\label{lem:lyons-uni}
For every $t\in(0,1)$ the accelerated Lyons IFS
$\tilde\Phi=\{\psi_n^t\}_{n\ge0}$ satisfies the
$(c_0,1)$-UNI condition.
\end{lem}

\begin{proof}
Consider the branches
\[
\psi_0^t(x)=\frac{x+t}{x+t+1},
\qquad
\psi_1^t(x)=\frac{x+t}{2x+2t+1},
\]
giving us a finite uniformly hyperbolic analytic sub-IFS $\Phi_0 := \{\psi_0^t,\psi_1^1\}$. Now, we will just verify the $(c_0,1)$-UNI for $\Phi_0$, which implies it for $\tilde \Phi$ since if for $N \in \N$ we have the UNI condition for words $\a,\b \in \{0,1\}^N$, these will also appear coding in the infinite IFS.

To verify $(c_0,1)$-UNI, since $\Phi_0$ is analytic and hyperbolic, we can rely on \cite[Claim 2.1]{AlgomRHWang-Polynomial} as an avenue to prove this. We say that $\Phi_0$ is not $C^2$-conjugate to a linear IFS if
there does not exist a $C^2$ diffeomorphism
$h:[0,1]\to h([0,1])\subset\R$ such that for all $i\in\mathcal \{0,1\}$,
\[
h\circ \varphi_i \circ h^{-1}(x)=a_i x+b_i
\]
for some constants $a_i\neq 0$, $b_i\in\R$. Then Claim~2.1 of
\cite{AlgomRHWang-Polynomial} says if $\Phi_0$ is not $C^2$-conjugate to a linear IFS, then there exists
$N\in\N$ and admissible words
$\mathbf a,\mathbf b\in\{0,1\}^N$
such that the corresponding branches
$\psi_{\mathbf a}^t,\psi_{\mathbf b}^t$
satisfy
\[
m\le
\left|
\partial_x\Big(
\log(\psi_{\mathbf a}^t)'(x)
-
\log(\psi_{\mathbf b}^t)'(x)
\Big)
\right|
\le m',
\qquad x\in[0,1],
\]
for some constants $m,m'>0$ and $m-2\tilde C\,\tilde\rho^{\,N} > 0$. This implies the UNI since for any $\mathbf c \in \{0,1\}^n$ the same estimate holds for the concatenations
$\mathbf c\mathbf a$ and $\mathbf c\mathbf b$
up to a bounded distortion error:
\[
\left|
\partial_x\Big(
\log(\psi_{\mathbf c\mathbf a}^t)'-
\log(\psi_{\mathbf c\mathbf b}^t)'
\Big)(x)
\right|
\ge
m-2\tilde C\,\tilde\rho^{\,N}.
\]
Indeed, using the chain rule we have
\[
\log(\psi_{\mathbf c\mathbf a}^t)'(x)
=
\log(\psi_{\mathbf a}^t)'(\psi_{\mathbf c}^t(x))
+
\log(\psi_{\mathbf c}^t)'(x),
\]
and similarly for $\mathbf c\mathbf b$.
Subtracting gives
\[
\log(\psi_{\mathbf c\mathbf a}^t)'(x)
-
\log(\psi_{\mathbf c\mathbf b}^t)'(x)
=
\Big(
\log(\psi_{\mathbf a}^t)'
-
\log(\psi_{\mathbf b}^t)'
\Big)
(\psi_{\mathbf c}^t(x)).
\]
Differentiating and using bounded distortion of
$\tilde\Phi$ yields
\[
\Big|
\partial_x\Big(
\log(\psi_{\mathbf c\mathbf a}^t)'
-
\log(\psi_{\mathbf c\mathbf b}^t)'
\Big)(x)
\Big|
\ge
m-2\tilde C\tilde\rho^{\,N}
=:c_0>0.
\]
Thus by mean value theorem then gives for all
$x\neq y\in[0,1]$,
\[
\big|
X_{\mathbf c\mathbf a}(x,y)
-
X_{\mathbf c\mathbf b}(x,y)
\big|
\ge
c_0|x-y|,
\]
which is precisely the $(c_0,\alpha)$-UNI estimate
(with $\alpha=1$) required in
Theorem~\ref{thm:mainCountable}.

Thus we just need to verify $\Phi_0$ is not $C^2$-conjugate to a linear IFS. This will follow if we prove the following non-$C^1$-cohomology to constant property for the IFS $\Phi_0$: there does not exist $u\in C^1$ and constants
$c_0,c_1$ with
\begin{equation}\label{eq:cohom-lyons}
\log|(\psi_j^t)'(x)|
=
u(\psi_j^t(x))-u(x)+c_j, \quad j = 0,1.
\end{equation}
Indeed, arguing by contraposition, suppose that $\Phi_0$ were
$C^2$-conjugate to a linear IFS. Then there exists a
$C^2$ diffeomorphism
$h:[0,1]\to h([0,1])\subset\R$ and constants
$a_j\neq0$, $b_j\in\R$ such that
\[
h\circ\psi_j^t\circ h^{-1}(y)=a_jy+b_j,
\qquad j=0,1.
\]
Equivalently,
$\psi_j^t=h^{-1}\circ L_j\circ h$ where
$L_j(y)=a_jy+b_j$.
Differentiating and using the chain rule gives for all $x\in[0,1]$
\[
(\psi_j^t)'(x)
=
a_j\,
\frac{h'(x)}{h'(\psi_j^t(x))}.
\]
Taking logarithms and setting
$u(x):=-\log h'(x)\in C^1([0,1])$ and
$c_j:=\log|a_j|$,
we obtain
\[
\log|(\psi_j^t)'(x)|
=
u(\psi_j^t(x))-u(x)+c_j,
\qquad j=0,1,
\]
which is precisely \eqref{eq:cohom-lyons}.
Therefore the failure of \eqref{eq:cohom-lyons}
implies that $\Phi_0$ cannot be $C^2$-conjugate to a linear IFS.

So we are done if we can prove the non-$C^1$-cohomology to constant property \eqref{eq:cohom-lyons} for $\Phi_0$. Assume on the contrary that \eqref{eq:cohom-lyons} holds. Consider the composed branch
\[
\psi_{(1,0)}^t(x)
=\frac{(1+t)x+t^2+2t}{(3+2t)x+2t^2+5t+1}.
\]
Let $x_0,x_1$ be the positive fixed points of $\psi_0^t,\psi_1^t$ and
let $x_{10}$ be the fixed point of $\psi_{(1,0)}^t$:
\[
x_0=\frac{-t+\sqrt{t^2+4t}}{2},\qquad
x_1=\frac{-t+\sqrt{t^2+2t}}{2},\qquad
x_{10}=\frac{-t(1+t)+\sqrt{t(1+t)(t+2)(t+3)}}{3+2t}.
\]
Then by \eqref{eq:cohom-lyons} we know that
$$\log (\psi_{(1,0)}^t)'(x_{10}) =
\log (\psi_0^t)'(x_0)+
\log (\psi_1^t)'(x_1).$$
Indeed, summing \eqref{eq:cohom-lyons} along the periodic orbit of the word
$(1,0)$ (at $x_{10}$) yields
\[
\log(\psi_{(1,0)}^t)'(x_{10})=c_1+c_0.
\]
Evaluating \eqref{eq:cohom-lyons} at $x_0,x_1$ gives
$\log(\psi_0^t)'(x_0)=c_0$ and
$\log(\psi_1^t)'(x_1)=c_1$ as claimed.

However, if we define
\[
Q(t):=
\log (\psi_{(1,0)}^t)'(x_{10})
-\log (\psi_0^t)'(x_0)
-\log (\psi_1^t)'(x_1)
\]
we can compute
\[
(\psi_0^t)'(x)=\frac{1}{(x+t+1)^2},\qquad
(\psi_1^t)'(x)=\frac{1}{(2x+2t+1)^2},
\]
\[
(\psi_{(1,0)}^t)'(x)
=\frac{1}{\big((3+2t)x+2t^2+5t+1\big)^2}.
\]
so evaluating the derivatives at these periodic points yields
\begin{align*}
(\psi_0^t)'(x_0)
&=\frac{4}{\big(t+2+\sqrt{t^2+4t}\big)^2},\\[4pt]
(\psi_1^t)'(x_1)
&=\frac{1}{\big(t+1+\sqrt{t^2+2t}\big)^2},\\[4pt]
(\psi_{(1,0)}^t)'(x_{10})
&=\frac{1}{\big(t^2+3t+1+\sqrt{t(1+t)(t+2)(t+3)}\big)^2}.
\end{align*}

Substituting gives
\begin{equation}\label{eq:Qclosed-lyons}
Q(t)
=
2\log\!\Bigg(
\frac{\big(t+1+\sqrt{t^2+2t}\big)\big(t+2+\sqrt{t^2+4t}\big)}
{2\big(t^2+3t+1+\sqrt{t(1+t)(t+2)(t+3)}\big)}
\Bigg),
\end{equation}
but $Q(t)\neq 0$ for all $t\in(0,1)$. This is a contradiction. Thus \eqref{eq:cohom-lyons} holds. 
\end{proof}

To finish the proof of Theorem \ref{thm:lyons}, we will just need to check that satisfies the light tail property of $\nu_t$.

\begin{lem}
The conductance measure $\nu_t$ is light-tailed for the accelerated IFS for all $\gamma\in\R$, that is, for where $p_n=2^{-(n+1)}$, $n\ge0$:
$
\forall \gamma\in\R:\quad
\sup_{x\in[0,1]}
\sum_{n=0}^\infty
p_n\,\big|(\psi_n^t)'(x)\big|^{-\gamma}
<\infty.
$
\end{lem}

\begin{proof}
For the light tail, note that
\[
(\psi_n^t)'(x)=\frac{1}{\big((n+1)(x+t)+1\big)^2},
\]
so for any $x\in[0,1]$
\begin{align*}
\sum_{n=0}^\infty
p_n\,\big|(\psi_n^t)'(x)\big|^{-\gamma}
=
\sum_{n=0}^\infty
\frac{1}{2^{n+1}}
\big((n+1)(x+t)+1\big)^{2\gamma}
\le
\sum_{n=0}^\infty
\frac{1}{2^{n+1}}
\big((n+1)t+1\big)^{2\gamma}
=:S.
\end{align*}
If $\gamma\le0$, then $S\le1$.
If $\gamma>0$, the ratio test yields $S < 1$ as
\[
\frac{(n+2)^{2\gamma}/2^{n+2}}
     {(n+1)^{2\gamma}/2^{n+1}}
=
\frac12\Big(\frac{n+2}{n+1}\Big)^{2\gamma}
\longrightarrow
\frac12<1.
\]
\end{proof}

Thus Theorem \ref{thm:lyons} is proved by Theorem \ref{thm:mainCountable}.

\section{Proof of Theorem \ref{thm:ps} on Patterson-Sullivan measures}\label{sec:ps}

Theorem \ref{thm:ps} on Patterson-Sullivan measures for hyperbolic surfaces with cusps follows from Theorem \ref{thm:mainCountable} when coding the limit set with the coding by Li and Pan \cite[Prop.~4.1]{LiPan}. Indeed Li-Pan construct a symbolic coding in which $\mu_\Gamma$ is a Gibbs measure with a light-tail property for a subshift of finite type of an IFS $\Phi$ satisfying (1)–(3) in our setting, and $\Lambda_\Gamma$ coincides with the attractor $F_\Phi$. More precisely, there exists a fundamental interval $I\subset\dot{\mathbb{R}}$ for the parabolic fixed point at $\infty$ (denoted $\Delta_\infty$ in \cite[§2.3]{LiPan}), constants $C_1>0$, $\kappa_+\in(0,1)$, $\gamma_0>0$, a countable family of disjoint open sets $\{I_a\}_{a\in\mathcal{A}}\subset I$, and an expanding map $T:\bigcup_a I_a\to I$ such that:
\begin{itemize}
\item[(1)] $\sum_{a\in\mathcal{A}}\mu_\Gamma(I_a)=\mu_\Gamma(I)$;
\item[(2)] for each $a$ there exists $\gamma_a\in\Gamma$ with $I_a=\gamma_a(I)$ and $T|_{I_a}=\gamma_a^{-1}$;
\item[(3)] (\emph{uniform contraction}) $|g_a'(x)|\le \kappa_+$ for $x\in I$;
\item[(4)] (\emph{bounded distortion}) $|\partial_x\log|\gamma_a'(x)||<C_1$ on $I$;
\item[(5)] (\emph{finite exponential moments}) with $\tau(x)=\log|DT(x)|$, one has $\int e^{\gamma_0\tau}\, d\mu_\Gamma<\infty$ for some $\gamma_0>0$.
\end{itemize}
Thus $\Phi=\{\gamma_a: a\in\mathcal{A}\}$ is an IFS with uniform $\kappa_+$–contraction and bounded $(C_1,1)$–distortion, and \cite{LiPan} shows
\[
C_{\phi,\gamma_0}=\sum_{a\in\mathcal{A}}\sup_{x\in I} |\gamma_a'(x)|^{\delta-\gamma_0}<\infty.
\]
With $\phi(x):=-\delta\log|\gamma_a'(x)|_{\mathbb{S}^1}$ for $x\in I_a$, equivariance yields that $\mu_\Gamma$ is a Gibbs measure for $\phi$, and the $(C_{\phi,\gamma_0},\gamma_0)$–tail condition holds. Finally, Li–Pan prove a UNI statement \cite[Lemma~4.5]{LiPan}: writing $\Phi^n=\{\gamma_{\mathbf{a}}:\mathbf{a}\in\mathcal{A}^n\}$ and $\tau_n(x)=\sum_{k=0}^{n-1}\tau(T^k x)$,
there exist $r,c_0>0$ such that for all large $N$ there are words $\{\gamma_{mj}\}_{m=1,2;\,1\le j\le j_0}$ with
\[
\big|\partial_z\big(\tau_{N}(\gamma_{1j}(z))-\tau_N(\gamma_{2j}(z))\big)\big|_{z=y}\big|\ \ge c_0,
\qquad \forall\,y\in B(x,r),\ \forall\,x\in \Lambda_\Gamma\cap \overline I,
\]
which implies $(c_0,1)$–UNI by the mean value theorem as we argued in the proof of Theorem \ref{thm:lyons}. Therefore Theorem~\ref{thm:mainCountable} applies, proving Theorem~\ref{thm:ps}. The decay exponent in Theorem~\ref{thm:ps} can be traced explicitly in terms of the contraction/distortion data, the UNI constants, and the tail exponent (all available from \cite{LiPan}).

\section{Proof of Theorem \ref{thm:FourierMP} on Manneville-Pommeau and Lorenz systems}\label{sec:para}

This section is devoted to the proof of Theorem \ref{thm:FourierMP}, which gives Fourier decay for some invariant measures for the non-accelerated dynamics of the Manneville-Pommeau map $T_\alpha$. The proof strategy is similar to Theorem \ref{thm:lyons} but as the measures we consider are not just Bernoulli measures like for the Lyons conductance measures $\nu_t$, we need more careful use of the unducing approach. The strategy here may be of independent interest to get power Fourier decay for Gibbs measures using inducing in general. 

To be able to use Theorem \ref{thm:mainCountable} to study Theorem \ref{thm:FourierMP}, we need to check that the inverse branches $g_a$ of $T_A$ indeed satisfy the axioms (1)-(3). The conditions (1) and (2) are immediate, e.g. they were deduced in \cite{LorenzYoungCoding} for Lorenz-maps but for Manneville-Pommeau maps we can for example see that $T_A$ satisfies the following $(C,1)$-distortion bound: there exists a constant $C > 0$ such that, for all $n \geq 1$ and any word $\mathbf{a} \in \mathbb{N}^n$:
\begin{align}
 \label{Adler_ineq}
 \ab{(S_n \tau \circ g_\mathbf{a})'} \leq C,
\end{align}
where $\tau = \log T_A'$ and $g_\a = g_{a_1} \circ \dots \circ g_{a_n}$. Indeed, and induction argument (with $n=1$ covered e.g. in \cite{Leslie}) shows that there exists $C \geq 1$ such that, for all $n \geq 0$, we have $\mar{Adl}(T_A^n) \leq C$, where $\mar{Adl}\su{f} \equiv \no{\frac{f''}{\su{f'}^2}}_\infty$, which implies the bounded distortions by the identity
\begin{align*}
 (S_n \tau \circ g_\mathbf{a})' = \frac{(T_A^n)'' \circ g_\mathbf{a}
 \cdot g_\mathbf{a}'}{(T_A^n)' \circ g_\mathbf{a}} = \text{Adl}(T_A^n) \circ g_\mathbf{a}.
\end{align*}

Now in both cases of Manneville-Pommeau or Lorenz maps, the induced map $T_A$ also satisfies the following Uniform Non-Integrability condition:

\begin{lem}[UNI]\label{lma:UNImpLorenz}
The $T_A$ associated to the MP map satisfies $(c_0,1)$-UNI on $I$.
\end{lem}

\begin{proof}[Proof of Lemma \ref{lma:UNImpLorenz}]
Let us first show for MP maps that there exists $c_0\in\su{0,1}$ and $N \geq 1$ such that, for all $n \geq N$, there exists $ \a,\b\in \Sigma^n$ such that for all $ x \in I$:
		\begin{align*}
			\ab{\su{S_n\tau\circ g_\a - S_n\tau\circ g_\b}'\su{x}} > c_0.
		\end{align*}
Note that explicitly $g_0 = f_0|_{I}$, so $g_0$ is linear map. Consider the words
\begin{align*}
 \a = 0^{N-1} 1 \quad \text{and} \quad \b = 0^N,
\end{align*}
where the power indicates that the letter is repeated that many times to make a longer word. This implies for all $N \in \N$ and $x \in I$ that
\begin{align*}
 \ab{\su{S_N\tau\circ g_\a - S_N\tau\circ g_\b}'\su{x}} = \Big|\frac{d}{dx}\Big(-\log g_1'(x) + \log 2^{N-1} - (- \log 2^N)\Big)\Big| = \Big|\frac{d}{dx} \log g_1'(x)\Big|.
\end{align*}
Now, if $x \in A_n$ we have $T_A(x) = (f_1^{-1})^n(f_0^{-1}(x))$ so in particular we know that $g_1(x) = (f_0 \circ f_1)(x)$ for all $x \in I$. Thus at $x \in I$ we have
\begin{align*}
 g_1'(x) &= \frac{1}{2} f_1'(x) = \frac{1}{2} \cdot \frac{1}{1+(\alpha +1)2^{\alpha} f_1(x)^{\alpha}},
\end{align*}
and thus
\begin{align*}
 \Big|\frac{d}{dx} \log g_1'(x)\Big| = \frac{\alpha(\alpha+1) 2^{\alpha} f_1(x)^{\alpha -1}}{(1+(\alpha +1)2^{\alpha} f_1(x)^{\alpha})^2}.
\end{align*}
Notice finally that for all $x \in I$, we have $f_1(x) \in I_1 \subset (0,1/2]$. In particular, we get the bound:
\begin{align*}
 \forall x \in I, \ \frac{\alpha(\alpha+1) 2^{\alpha} f_1(x)^{\alpha -1}}{(1+(\alpha +1)2^{\alpha} f_1(x)^{\alpha})^2} > c_0,
\end{align*}
for some $c_0 > 0$ that only depends on $\alpha$. Now, using the mean value theorem (since $S_n \tau \circ g_\mathbf{a}$ is $C^1$ on $I$), one can rewrite the above into a condition that does not involve any derivatives:
\begin{align*}
 \forall (x,y) \in I^2, \ |S_n \tau \circ g_\mathbf{a}(x) - S_n \tau \circ g_\mathbf{b}(x) - S_n \tau \circ g_\mathbf{a}(y) + S_n \tau \circ g_\mathbf{b}(y)| \geq c_0 |x-y|,
\end{align*}
which is exactly the bound of $(c_0,1)$-UNI.
\end{proof}

\begin{remark}
As for UNI condition for Lorenz maps, we can argue as in the proof of Theorem \ref{thm:lyons}. As with Theorem \ref{thm:lyons}, the sub-IFS $\{g_0,g_1\}$ for the induced system satisfies UNI if it is not $C^2$ conjugated to a linear IFS. Now $C^2$ conjugacy to an affine IFS is not possible in general, and we do this for specific numbers here to get examples of Lorenz maps. Indeed, denoting by $x_0,x_1,x_{10}$ the fixed points of $g_0,g_1,g_{10}$ respectively, we have that
$ \log (g_{1}\circ g_{0})'(x_{10}) = \log g_{0}'(x_0) + \log g_1'(x_1)$. We omit the details here, once can proceed as in Theorem \ref{thm:lyons} by finding the fixed points and computing the derivatives. One can see this already in simulations e.g. for $b_0 = -1/2$, $b_1 = 1/2$, $a = 1.1$ and $\alpha = 0.25$ reveals $g_0(x) = -0.68301 (0.5 - x)^4$ with fixed point $x_0 = -0.07429$, $g_0’(x) = 2.73205 (0.5 - x)^3$ with $g_0’(x_0) = 0.51748$, $g_1(x) = 0.68301 (0.5 + x)^4$ with fixed point $x_1 = 0.07429$, $g_1’(x) = 2.73205 (0.5 + x)^3$ with $g_1’(x_1) = 0.51748$ so the product $g_0’(x_0) g_1’(x_1) \approx 0.26779$. Moreover, the composition $\su{g_1 \circ g_0}(x) = 0.683013 (0.5 -0.683013 (0.5 - x)^4)^4$ has fixed point $x_{10} = 0.03259$ and the derivative $(g_1 \circ g_0)’(x) = 7.46411 (0.5 - 0.683013 (-0.5 + x)^4)^3 (0.5 - x)^3$ with $(g_1 \circ g_0)’(x_{10}) \approx 0.077825$. Thus: $| \log (g_1 \circ g_0)’(x_{10}) - \log g_0’(x_0)-\log g_1’(x_1)| \approx 1.23572$.
\end{remark}

Thus we are in the setting of Theorem \ref{thm:mainCountable} to have that all equilibrium states for $T_A$ of light-tail have powefr Fourier decay. Then to go back to the parabolic system, we will just have to use the inducing structure in the following:

\begin{proof}[Proof of Theorem \ref{thm:FourierMP} assuming Theorem \ref{thm:mainCountable}] Let us first write the Manneville-Pommeau map case. Firstly, since
\begin{align*}
 \nu_\phi(E) = \sum_{a = 0}^\infty \sum_{k = 0}^a \mu_\phi(T_\alpha^{-k}(E) \cap A_a), \quad E \subset I,
\end{align*}
we can write
\begin{align*}
 \int \chi(x) e^{i\xi \psi(x)} \, d\nu_\phi(x) = \sum_{a = 0}^\infty \sum_{k = 0}^a \int_{I} \chi(T_\alpha^k(x)) e^{i\xi \psi(T_\alpha^k(x))} \1_{A_a}(x) \, d\mu_\phi(x).
\end{align*}
Applying the transfer operator associated to $\phi$ once, we see that this equals to
\begin{align*}
 \sum_{a = 0}^\infty \sum_{k = 0}^a \sum_{b = 0}^\infty \int_{I} \chi(T_\alpha^k(g_b(x))) e^{i\xi \psi(T_\alpha^k(g_b(x)))} \1_{A_a}(g_b(x)) e^{\phi(g_b(x))}\, d\mu_\phi(x).
\end{align*}
Notice that as $g_b : I \to A_b$, we have that $\1_{A_a}(g_b(x)) = 1$ if and only if $a = b$, so we arrive to
\begin{align*}
 \sum_{a = 0}^\infty \sum_{k = 0}^a \int_{I} \chi(T_\alpha^k(g_a(x))) e^{i\xi \psi(T_\alpha^k(g_a(x)))} e^{\phi(g_a(x))}\, d\mu_\phi(x).
\end{align*}
Define now maps $\psi_{a,k} : I \to \R$ and $\chi_{a,k} : I \to \R$ by the formulae
\begin{align*}
 \psi_{a,k}(x) := \psi(T_\alpha^k(g_a(x))), \quad \chi_{a,k}(x) := \chi(T_\alpha^k(g_a(x))) \frac{e^{\phi(g_a(x))}}{e^{\phi(g_a(3/4))}}, \quad x \in I.
\end{align*}
Then we arrive to the formula
\begin{align*}
 \sum_{a = 0}^\infty \sum_{k = 0}^a e^{\phi(g_a(3/4))} \int_{I} \chi_{a,k}(x) e^{i\xi \psi_{a,k}(x)} d\mu_\phi(x).
\end{align*}
Notice that, for any $a$, we have $T_\alpha^a(g_a(x))=x$. Remember also that $|T_\alpha'|\geq 1$. It follows that, for any $k \leq a$, we have a bound
$ |(T_{\alpha}^k \circ g_\mathbf{a})'(x)| \leq 1 $. In particular, for every $k\geq a$, we have a uniform bound $\| T_\alpha^k \circ g_a \|_{C^\alpha} \lesssim 1$. It follows that we can find uniform bounds
\begin{align*}
 \|\chi_{a,k}\|_{C^{\alpha}} &\lesssim 1,\\
 \|\log \psi_{a,k}\|_{C^{1+\alpha}} &\lesssim 1.
\end{align*}
As $\inf |\psi'| > c$ and $|T_\alpha'| \geq 1$, we have
\begin{align*}
 |\psi_{a,k}'(x)| \geq c|g_a'(x)| \geq c'|g_a'(3/4)|,
\end{align*}
 for some constant $c' > 0$ by the bounded distortion property of $T_A$.
We will then cut the sum over $a$ depending on if $|g_a'(3/4)| \geq |\xi|^{-\varepsilon_\psi}$, where $\varepsilon_\psi$ is small enough so that Theorem \ref{thm:mainCountable} gives us a uniform decay bound. Notice also that, because of the support condition on $\chi$, there exists some integer $K(\chi) \geq 1$ such that $\chi_{a,k}=0$ for any $k \geq K(\chi)$. We then find:
$$ \Big|\int \chi(x) e^{i\xi \psi(x)} \, d\nu_\phi(x)\Big| \leq \sum_{a=0}^\infty \sum_{k=0}^a e^{\phi(g_a(3/4))} \Big| \int_I \chi_{a,k}(x) e^{i \xi \psi_{a,k}(x)} d\mu_\phi(x) \Big| $$
$$ \leq \underset{|g_a'(3/4)| \geq |\xi|^{-\varepsilon_\psi}}{\sum_{a \geq 0}} \sum_{k=0}^a e^{\phi(g_a(3/4))} \Big| \int_I \chi_{a,k}(x) e^{i \xi \psi_{a,k}(x)} d\mu_\phi(x) \Big| + \underset{|g_a'(3/4)| \leq |\xi|^{-\varepsilon_\psi}}{\sum_{a \geq 0}} \sum_{k=0}^a e^{\phi(g_a(3/4))} \|\chi_{a,k}\|_{\infty} $$
$$ \lesssim {\sum_{a \geq 0}} K(\chi) e^{\phi(g_a(3/4))} |\xi|^{-\beta} + {\sum_{a \geq 0}} K(\chi) e^{\phi(g_a(3/4))} \mathbf{1}_{|g_a'(3/4)| \leq |\xi|^{-\varepsilon_\psi}}. $$
$$ \lesssim |\xi|^{-\beta} + {\sum_{a \geq 0}} e^{\phi(g_a(3/4))} \mathbf{1}_{|g_a'(3/4)| \leq |\xi|^{-\varepsilon_\psi}} $$
$$ \lesssim |\xi|^{-\beta} + {\sum_{a \geq 0}} e^{\phi(g_a(3/4))} |g_a'(3/4)|^{-\gamma} |\xi|^{-\gamma \varepsilon_\psi} \lesssim |\xi|^{- \min( \beta, \gamma \varepsilon_\psi )},$$
which concludes the proof.

For Lorenz maps $f$ the proof is same due to the formula for the pushed measure given by the same expression as the Manneville-Pommeau maps. The main difference is the choice of $3/4$ which we now choose inside the inducing region constructed in \cite{LorenzYoungCoding} in a neighbourhood of $0$.
\end{proof}

\section{Proof of Theorem \ref{thm:C1alpha} on examples of $C^{1+\alpha}$ IFSs}\label{sec:IFSconstruct}

We now prove Theorem \ref{thm:C1alpha} by constructing $C^{1+\alpha}$ IFSs satisfying MNL, yielding fractal equilibrium states with power Fourier decay. The construction proceeds via fixed point theory in the moduli space of IFSs parametrized by equilibrium states, ensuring the required non-concentration to implement the proof of Theorem~\ref{thm:mainCountable}. The key input is the existence of a branch whose log-derivative has $\alpha$–Hölder \emph{lower} regularity, as furnished by a Cantor-type function with jumps precisely on the attractor.

Let us first define the space of Frostman measures $\mu$ in $[-1,1]$ for given constants $C\ge1$ and $\varepsilon>0$ by
\[
\mathcal{M}_{C,\varepsilon}:=\Big\{\mu\in\mathcal{P}([-1,1])\ :\ \mu(I)\le C |I|^\varepsilon\ \text{for all intervals } I \text{ with } |I|\le \tfrac{1}{100}\Big\}.
\]
Given $\mu\in\mathcal{M}_{C,\varepsilon}$, we can define an IFS $\Phi_\mu$ depending on $\mu$ with branches
\[
g_{\mu,1}(x):=\kappa \int_0^x \exp(\mu(-\infty,t])\, dt + b_1,\qquad g_{\mu,0}(x):=\kappa_0 x + b_0,\qquad x\in[0,1],
\]
where $0<\kappa_0\le\kappa<1/10$ and $b_0,b_1\in\R$ are chosen so that $\Phi_\mu=\{g_{\mu,0},g_{\mu,1}\}$ satisfies the strong separation condition. By the Frostman property, $\Phi_\mu$ is $C^{1+\alpha}$ for any $\alpha\le\varepsilon$. Let $T_\mu$ be the associated expanding map on $F_{\Phi_\mu}$ and write $\tau_\mu(x):=-\log |T_\mu'(x)|$ as the $C^\alpha$ roof function.

\begin{thm}\label{thm:C1+alpha}
For any $C>0$ and $\varepsilon>0$, there exist $\mu\in\mathcal{M}_{C,\varepsilon}$ and $0<\delta_\mu<1$ such that the Gibbs measure associated to the potential $-\delta_\mu \tau_\mu$ for the (truly) $C^{1+\varepsilon}$ IFS $\Phi_\mu$ has power Fourier decay.
\end{thm}

The proof performs a fixed-point analysis for ``$\mu$–controlled'' transfer operators
\[
\mathcal{L}_\mu f(x):=\sum_i |g'_{\mu,i}(x)|^{\delta_\mu}\, f(g_{\mu,i}(x)),
\]
which depend on the equilibrium state $\mu$, where $\delta_\mu \in (0,1)$ is chosen as the unique number such that
$$\cL_{\mu}^* (\mu)(\R) = 1,$$
which exists by the intermediate value theorem as $\mu$ is a probability measure. Then this gives us an operator $\Psi : \cM_{C,\eps} \to \cM_{C,\eps}$ defined by
$$\Psi(\mu) := \cL_{\mu}^* (\mu), \quad \mu \in \cM_{C,\eps}.$$
Now the strategy to prove Theorem \ref{thm:C1+alpha} is to construct a fixed point $\mu$ for the operator $\Psi$, that is, a measure satisfying
$$\cL_{\mu}^* (\mu) = \mu.$$
Therefore, the pressure $P(\phi^\mu) = 0$ for $\phi^\mu := \delta_\mu \tau_\mu$, which implies, as $\Phi_\mu$ is $C^{1+\alpha}$, that $\delta^\mu = \Hd (\mathrm{supp} \, \mu)$ and that $\mu$ is $\delta^\mu$ Ahlfors-David regular. This then allows us to have a convenient form for the geometric potential $\tau_\mu$: for some $c_1 \in \R$, we have
$$\log g_1'(x) = c_1 + \mu(-\infty,x) \in C^{\delta_\mu}.$$
We can find this fixed point $\mu$ by adapting Schauder's fixed point theorem to the operator $\Psi : \cM_{C,\eps} \to \cM_{C,\eps}$, which is possible by the following Lemma:

\begin{lem}\label{lma:propertiesC1alpha}
 \begin{itemize}
 \item[(0)] As a map from $\cM_{C,\eps}$, the maps $\mu \mapsto g_{\mu,i}'$ and $\mu \mapsto g_{\mu,i}$ are continuous
 \item[(1)] The choice of $\delta^\mu$ is well-defined and continuous as a map of $\mu$, and that $\delta^\mu \geq 3\eps$ for all $\mu \in \cM_{C,\eps}$
 \item[(2)] $\Psi$ is well-defined: $\Psi(\cM_{C,\eps}) \subset \cM_{C,\eps}$
 \item[(3)] $\Psi$ is continuous on $\cM_{C,\eps}$ in the weak-$\ast$ topology
 \item[(4)] $\cM_{C,\eps}$ is closed and convex
\end{itemize}
\end{lem}

We postpone the proof of these properties the end of the section. Now, the fixed point property then allows us to prove Theorem \ref{thm:C1+alpha} by following Lemma:

\begin{lem}[Measure Non-Linearity]\label{lma:MNL}
The IFS $\Phi_\mu$ and the fixed point $\mu$ satisfies MNL condition: There exists $C_{\text{(MNL)}}$ such that for all $N \geq 1$, there exists $\widehat{\mathbf{a}},\widehat{\mathbf{b}} \in \mathcal{A}^N$ such that
$$\forall x_0 \in F, \ \forall I \subset \mathbb{R}, \ (X_\mathbf{a}(\cdot,x_0)-X_\mathbf{b}(\cdot,x_0))_*\mu(I) \leq C_{\text{(MNL)}} |I|. $$
\end{lem}

\begin{proof}
Since $\log g_1'(x) = c_1 + \mu(-\infty,x)$ for some constant $c_1 \in \R$, and since $\log g_0' = c_0$, we know that for every $N$ there exists $\mathbf{a},\mathbf{b} \in \mathcal{A}^N$ such that
$$S_N \tau_\mu \circ g_\a - S_N \tau_\mu \circ g_\b = F_\mu(x) + c_0.$$
Indeed, for example words $\mathbf{a} := 0^{N} 1$ and $\mathbf{b} := 0^{N-1}1$ gives us such property. The general fact that $(F_\mu)_* \mu$ is equal to the Lebesgue measure on $[0,1]$ since $\mu$ doesn't have atoms concludes our proof.
\end{proof}

So we are left with proving Lemma \ref{lma:propertiesC1alpha}.

\begin{proof}[Proof of Lemma \ref{lma:propertiesC1alpha}]

\textit{Property (0)}. Since any $\mu \in \cM_{C,\eps}$ has no atoms, we know that $\mu_n(-\infty,t] \to \mu(-\infty,t]$ if $\mu_n \to \mu$ in $\cM_{C,\eps}$. Thus by the dominated convergence theorem $g_{\mu_n,i}(x) \to g_{\mu,i}(x)$. Similarly $g_{\mu_n,i}'$ is continuous as the formula involves also just applying $\mu_n(-\infty,t] \to \mu(-\infty,t]$ and dominated convergence theorem. In fact, convergence for the cumulative distribution function is uniform since $\mu$ has no atoms (see for example Exercise 5.2 in \cite{Par}). \\

\textit{Property (1)}. Write $\cA = \{0,1\}$ and define
$$\zeta_\mu(\alpha) := \int_\R \sum_{a \in \cA} |g_{\mu,a}'(x)|^\alpha \, d\mu(x) = \sum_{a \in \cA} \int_\R |g_{\mu,a}'(x)|^\alpha \, d\mu(x).$$
Then $\zeta_\mu$ is continuous, decreasing and satisfies $\zeta_\mu(0) = 2$ and
$$\zeta_\mu(1) = \sum_{a \in \cA} \int_\R |g_{\mu,a}'(x)| \, d\mu(x) \leq \frac{6}{10} < 1$$
by the choice of $\kappa < 1/10$. Thus there indeed exists unique $\delta^\mu$ such that $\zeta_\mu(\delta^\mu) = 1$.

Let us now show that $\mu \mapsto \delta^\mu$ is continuous map from $\cM_{C,\eps}$. Suppose $\mu_n\to \mu$. We have for any $n \in \N$ by the choices of $\delta^\mu$ we have $\zeta_{\mu_n}(\delta^{\mu_n}) = 1 = \zeta_{\mu}(\delta^{\mu})$. Moreover, we know that $\zeta_{\mu_n} \to \zeta_\mu$ uniformly since $\mu$ has no atoms. Thus
$$|\zeta_{\mu}(\delta^{\mu_n}) - 1| \leq |\zeta_{\mu_n}(\delta^{\mu_n}) - 1| + |\zeta_{\mu}(\delta^{\mu_n}) - \zeta_{\mu_n}(\delta^{\mu_n})| \leq \|\zeta_\mu - \zeta_{\mu_n}\|_\infty$$
so $|\delta^{\mu_n} - \delta^\mu| \to 0$ as $n \to \infty$ as $\zeta_\mu'(\alpha) = \sum_a \int (\log |g_{\mu,a}'|) |g_{\mu,a}'|^{\alpha} \, d\mu$ giving $\zeta' \in [-C,-c]$ for some constants $C,c > 0$. Indeed, if $f \in C^1(\mathbb{R},\mathbb{R})$ be such that $f' \leq -c$, then for any $h > 0$ we have by Taylor's approximation $f(x+h) \leq f(x)- c_0 h$, so that if $f(x_0) \leq \varepsilon$, we know that the zero $z_0$ of $f$ satisfies $z_0 \in [x_0,x_0+ \varepsilon/c_0]$. Indeed, $f(x_0+ \varepsilon/c_0) \leq 0$ by the previous Taylor estimate and so by intermediate value theorem the zero is in the interval. Thus $\delta^\mu$ is continuous in $\mu$.

Finally, we notice that
$$\zeta_\mu(3 \eps) = \int_\R \sum_{a \in \cA} |g_{\mu,a}'(x)|^{3\eps} \, d\mu(x) \geq 3/2$$
if we assume $\eps < 1/25$ as then $|g_{\mu,a}'(x)|^{3\eps} \geq 10^{-3\eps} \geq 3/4$, so $\delta^\mu \geq 3\eps$. \\

\textit{Property (2)}. We need to check $\Psi(\cM_{C,\eps}) \subset \cM_{C,\eps}$ so fix $\mu \in \cM_{C,\eps}$. We know that $\mu$ is a probability measure as $\mu = \cL_{\mu}^*(\mu)$ and $\cL_{\mu}^*(\mu)(\R) = 1$ by the choice of $\delta^\mu$. Fix $h$ supported on $(1,\infty)$, which then gives
$$\int h \, d \cL_{\mu}^*(\mu) = \int \cL_{\mu}(h) \, d\mu= 0$$
as the support of $\cL_{\mu}(h)$ is also in $(1,\infty)$ as for all $x \in (-1,1)$ we have $g_{\mu,i}(x) \in (-1,1)$ and so $\cL_{\mu}(h)(x) = 0$. Hence we know that $\cL_{\mu}^*(\mu) \in \cP([-1,1])$. Let now $I \subset \R$ with $|I| \leq \frac{1}{100}$. Then
$$\cL_{\mu}^*(\mu)(I) = \int \cL_{\mu}(\1_I) \, d\mu.$$
As $|I| \leq \frac{1}{100}$, the interval $I$ can contain only one of the intervals $I_a$, $a \in \cA$, say, $a_0 \in \cA$, so
\begin{align*}
\cL_{\mu}(\1_I) &= \int |g_{\mu,a_0}'(x)|^{\delta^\mu}\1_I(g_{\mu,a_0}(x)) \, d\mu(x) \leq \kappa^{\delta^\mu} \mu(g_{\mu,a_0}^{-1}I) \\
& \leq \kappa^{\delta^\mu} |\kappa_-^{-1} I|^\eps \leq \kappa^{\delta^\mu} \kappa_-^{-\eps} | I|^\eps \leq |I|^\eps
\end{align*}
as long as $\eps > 0$ here is small enough and using $\delta_\mu \geq 1/25$ we noticed earlier in proving (1). Thus indeed $\Psi(\mu) \in \cM_{C,\eps}$. \\

\textit{Property (3)}. To check the continuity of $\Phi$, let us notice that for a continuous $h$ we can write
$$\int h \, d \cL_{\mu}^*(\mu) = \sum_{a \in \cA} \int |g_{\mu,a}'(x)|^{\delta^\mu} h(g_{\mu,a}(x)) \, d\mu(x) $$
and here if $\mu_n \to \mu$ the integrand $|g_{\mu_n,a}'(x)|^{\delta^{\mu_n}} h(g_{\mu_n,a}(x))$ converges uniformly to $|g_{\mu,a}'(x)|^{\delta^\mu} h(g_{\mu,a}(x))$ as $\mu$ has no atoms. This implies that $\int h \, d \cL_{\mu_n}^*(\mu_n) \rightarrow \int h \, d \cL_{\mu}^*(\mu)$.

\textit{Property (4)}. The set $\cM_{C,\eps}$ is convex as for any $t \in (0,1)$ and $\mu,\nu \in \cM_{C,\eps}$ we have for any interval $I$ that
$$(t \mu + (1-t) \nu)(I) \leq t|I|^\eps+(1-t)|I|^\eps = |I|^{\eps}.$$
Finally, the set $\cM_{C,\eps}$ is closed as if $\mu_n\to \mu$, with $\mu_n \in \cM_{C,\eps}$, then for any interval $I$ we also have $\mu_n(I) \to \mu(I) $ as $\mu$ has no atoms implying that $\mu(I) \leq C|I|^\eps$. \end{proof}

\section*{Acknowledgements} We thank Amir Algom, Simon Baker, Amlan Banaji, Jonathan Fraser, Thomas Jordan, Natalia Jurga, Osama Khalil, Fr\'ed\'eric Naud, Tuomas Orponen, Mike Todd and Caroline Wormell for useful discussions during the preparation of this manuscript. We thank in particular Simon Baker for suggesting to consider the Liverani-Saussol-Vaienti parametrisation for the Manneville-Pommeau map for studying the UNI and its generalisation, and discussions on the Lyons conjecture.

\bibliographystyle{plain}
\bibliography{bibliography}

\begin{thebibliography}{10}

\bibitem{AlgomChangWuWu}
Amir {Algom}, Yuanyang {Chang}, Meng {Wu}, and Yu-Liang {Wu}.
\newblock {Van der Corput and metric theorems for geometric progressions for
  self-similar measures}.
\newblock {\em arXiv e-prints}, page arXiv:2401.01120, January 2024.

\bibitem{AlgomHertzWang-Log}
Amir Algom, Federico~Rodriguez Hertz, and Zhiren Wang.
\newblock Logarithmic {F}ourier decay for self conformal measures.
\newblock {\em J. Lond. Math. Soc. (2)}, 106(2):1628--1661, 2022.

\bibitem{AlgomKhalil}
Amir Algom and Osama Khalil.
\newblock $l^2$-flattening of self-similar measures on non-degenerate curves,
  2025.

\bibitem{AlgomOrponen}
Amir Algom and Tuomas Orponen.
\newblock Uniformly perfect measures on strictly convex planar graphs are
  $l^{2}$-flattening, 2025.

\bibitem{AlgO}
Amir Algom, Snir~Ben Ovadia, Federico~Rodriguez Hertz, and Mario Shannon.
\newblock How linear can a non-linear hyperbolic ifs be?, 2024.

\bibitem{AlgomHertzWang-Normality}
Amir Algom, Federico Rodriguez~Hertz, and Zhiren Wang.
\newblock Pointwise normality and {F}ourier decay for self-conformal measures.
\newblock {\em Adv. Math.}, 393:Paper No. 108096, 72, 2021.

\bibitem{AlgomRHWang-Polynomial}
Amir {Algom}, Federico {Rodriguez Hertz}, and Zhiren {Wang}.
\newblock {Polynomial Fourier decay and a cocycle version of Dolgopyat's method
  for self conformal measures}.
\newblock {\em arXiv e-prints}, page arXiv:2306.01275, June 2023.

\bibitem{AV}
V{\'\i}tor Ara{\'u}jo and Paulo Varandas.
\newblock Robust exponential decay of correlations for singular-flows.
\newblock {\em Communications in Mathematical Physics}, 311(1):215--246, 2012.

\bibitem{AvilaLyubichZhang}
Artur Avila, Mikhail Lyubich, and Zhiyuan Zhang.
\newblock To appear.

\bibitem{BakerBanaji}
Simon {Baker} and Amlan {Banaji}.
\newblock {Polynomial Fourier decay for fractal measures and their
  pushforwards}.
\newblock {\em arXiv e-prints}, page arXiv:2401.01241, January 2024.

\bibitem{BakerKhalilSahlsten}
Simon Baker, Osama Khalil, and Tuomas Sahlsten.
\newblock Fourier decay from $l^2$-flattening, 2024.

\bibitem{BKS}
Simon {Baker}, Osama {Khalil}, and Tuomas {Sahlsten}.
\newblock {Fourier Decay from $L^2$-Flattening}.
\newblock {\em arXiv e-prints}, page arXiv:2407.16699, July 2024.

\bibitem{BakerSahlsten}
Simon {Baker} and Tuomas {Sahlsten}.
\newblock {Spectral gaps and Fourier dimension for self-conformal sets with
  overlaps}.
\newblock {\em arXiv e-prints}, page arXiv:2306.01389, June 2023.

\bibitem{baladi2000positive}
Viviane Baladi.
\newblock {\em Positive transfer operators and decay of correlations},
  volume~16.
\newblock World scientific, 2000.

\bibitem{BanajiYu}
Amlan Banaji and Han Yu.
\newblock Fourier transform of nonlinear images of self-similar measures:
  quantitative aspects, 2025.

\bibitem{barany2022typical}
Bal{\'a}zs B{\'a}r{\'a}ny, K{\'a}roly Simon, Boris Solomyak, and Adam
  {\'S}piewak.
\newblock Typical absolute continuity for classes of dynamically defined
  measures.
\newblock {\em Advances in Mathematics}, 399:108258, 2022.

\bibitem{Bourgain-SumProduct}
Jean Bourgain.
\newblock The discretized sum-product and projection theorems.
\newblock {\em J. Anal. Math.}, 112:193--236, 2010.

\bibitem{BourgainDyatlov}
Jean Bourgain and Semyon Dyatlov.
\newblock Fourier dimension and spectral gaps for hyperbolic surfaces.
\newblock {\em Geom. Funct. Anal.}, 27(4):744--771, 2017.

\bibitem{BMT}
Henk Bruin, Ian Melbourne, and Dalia Terhesiu.
\newblock Sharp polynomial bounds on decay of correlations for multidimensional
  nonuniformly hyperbolic systems and billiards.
\newblock {\em Annales Henri Lebesgue}, 4:407--451, 2021.

\bibitem{BruinTodd}
Henk Bruin and Mike Todd.
\newblock Equilibrium states for interval maps: Potentials with sup phi inf phi
  htop f.
\newblock {\em Communications in Mathematical Physics}, 283(3):579--611, 2008.

\bibitem{SuomalaChen}
Changhao Chen, Bing Li, and Ville Suomala.
\newblock Fourier dimension of mandelbrot multiplicative cascades, 2025.

\bibitem{ChenHan}
Xinxin Chen, Yong Han, Yanqi Qiu, and Zipeng Wang.
\newblock Harmonic analysis of mandelbrot cascades -- in the context of
  vector-valued martingales, 2025.

\bibitem{ChernovZhang}
N~Chernov and H-K Zhang.
\newblock Billiards with polynomial mixing rates.
\newblock {\em Nonlinearity}, 18(4):1527, apr 2005.

\bibitem{Leslie}
Tom Contie-Leslie.
\newblock An introduction to intermittency, thesis.

\bibitem{curien2017harmonic}
Nicolas Curien and Jean-Francois Le~Gall.
\newblock The harmonic measure of balls in random trees.
\newblock 2017.

\bibitem{DavenportErdosLeVeque}
Harold Davenport, Paul Erd\Horig{o}s, and William LeVeque.
\newblock On {W}eyl's criterion for uniform distribution.
\newblock {\em Michigan Math. J.}, 10:311--314, 1963.

\bibitem{Todd}
Mark~F. Demers and Mike Todd.
\newblock Slow and fast escape for open intermittent maps.
\newblock {\em Communications in Mathematical Physics}, 351(2):775--835, 2017.

\bibitem{demeter2024szemer}
Ciprian Demeter and Hong Wang.
\newblock Szemer$\backslash$'edi-trotter bounds for tubes and applications.
\newblock {\em arXiv preprint arXiv:2406.06884}, 2024.

\bibitem{LorenzYoungCoding}
Karla D{\'\i}az-Ordaz.
\newblock Decay of correlations for non-h{\"o}lder observables for
  one-dimensional expanding lorenz-like maps.
\newblock {\em Discrete and continuous dynamical systems}, 15(1):159--176,
  2006.

\bibitem{Dihn}
Tien-Cuong Dinh, Lucas Kaufmann, and Hao Wu.
\newblock Decay of fourier coefficients for furstenberg measures.
\newblock {\em Transactions of the American Mathematical Society},
  376(10):6873--6926, 2023.

\bibitem{Dolgopyat}
Dmitry Dolgopyat.
\newblock On decay of correlations in {A}nosov flows.
\newblock {\em Ann. of Math. (2)}, 147(2):357--390, 1998.

\bibitem{FalconerJin2019}
Kenneth Falconer and Xiong Jin.
\newblock {Exact dimensionality and projection properties of Gaussian
  multiplicative chaos measures}.
\newblock {\em Trans. Am. Math. Soc.}, 372(4):2921--2957, 2019.

\bibitem{FengLau}
De-Jun Feng and Ka-Sing Lau.
\newblock Multifractal formalism for self-similar measures with weak separation
  condition.
\newblock {\em J. Math. Pures Appl. (9)}, 92(4):407--428, 2009.

\bibitem{fraserspectrum}
Jonathan~M. Fraser.
\newblock The fourier spectrum and sumset type problems.
\newblock {\em Mathematische Annalen}, 390(3):3891--3930, 2024.

\bibitem{GaVargas}
Christophe Garban and Vincent Vargas.
\newblock Harmonic analysis of gaussian multiplicative chaos on the circle,
  2024.

\bibitem{GuillarmouMazzeo}
Colin Guillarmou and Rafe Mazzeo.
\newblock Resolvent of the laplacian on geometrically finite hyperbolic
  manifolds.
\newblock {\em Inventiones mathematicae}, 187(1):99--144, 2012.

\bibitem{Hochman-Annals}
Michael Hochman.
\newblock On self-similar sets with overlaps and inverse theorems for entropy.
\newblock {\em Ann. of Math. (2)}, 180(2):773--822, 2014.

\bibitem{JorSah}
Thomas Jordan and Tuomas Sahlsten.
\newblock Fourier transforms of {G}ibbs measures for the {G}auss map.
\newblock {\em Math. Ann.}, 364(3-4):983--1023, 2016.

\bibitem{Kaufman2}
Robert Kaufman.
\newblock Continued fractions and {F}ourier transforms.
\newblock {\em Mathematika}, 27(2):262--267, 1980.

\bibitem{Khalil-Mixing}
Osama {Khalil}.
\newblock {Exponential Mixing Via Additive Combinatorics}.
\newblock {\em arXiv e-prints}, page arXiv:2305.00527, April 2023.

\bibitem{lai2025fourier}
Chun-Kit Lai and Cheuk~Yin Lee.
\newblock Fourier dimension of the graph of fractional brownian motion with $h
  \geq 1/2$.
\newblock {\em arXiv preprint arXiv:2502.09032}, 2025.

\bibitem{LeclercPreprint}
Ga\'{e}tan Leclerc.
\newblock Fourier decay of equilibrium states and the fibonacci hamiltonian,
  preprint arxiv:2507.23731 (2025).

\bibitem{Leclerc-JuliaSets}
Ga\'{e}tan Leclerc.
\newblock Julia sets of hyperbolic rational maps have positive {F}ourier
  dimension.
\newblock {\em Comm. Math. Phys.}, 397(2):503--546, 2023.

\bibitem{leclerc2024nonlinearityfractalsfourierdecay}
Gaétan Leclerc.
\newblock Nonlinearity, fractals, fourier decay -- harmonic analysis of
  equilibrium states for hyperbolic dynamical systems, 2024.

\bibitem{lee2025fourierdimensionfractionalbrownian}
Cheuk~Yin Lee and Samy Tindel.
\newblock On the fourier dimension of fractional brownian graphs, 2025.

\bibitem{LiENS}
Jialun Li.
\newblock Fourier decay, renewal theorem and spectral gaps for random walks on
  split semisimple lie groups.
\newblock {\em Annales scientifiques de l'Ecole normale supérieure}, 6(55),
  2022.

\bibitem{LiNaudPan}
Jialun Li, Fr\'{e}d\'{e}ric Naud, and Wenyu Pan.
\newblock Kleinian {S}chottky groups, {P}atterson-{S}ullivan measures, and
  {F}ourier decay.
\newblock {\em Duke Math. J.}, 170(4):775--825, 2021.
\newblock With an appendix by Li.

\bibitem{LiPan}
Jialun Li and Wenyu Pan.
\newblock Exponential mixing of geodesic flows for geometrically finite
  hyperbolic manifolds with cusps.
\newblock {\em Invent. Math.}, 231(3):931--1021, 2023.

\bibitem{LindenstraussVarju}
Elon Lindenstrauss and P\'eter Varj\'u.
\newblock Spectral gap in the group of affine transformations over prime
  fields.
\newblock {\em Ann. Fac. Sci. Toulouse Math. (6)}, 25(5):969--993, 2016.

\bibitem{Lindenstrauss_2016}
Elon Lindenstrauss and Péter~P. Varjú.
\newblock Random walks in the group of euclidean isometries and self-similar
  measures.
\newblock {\em Duke Mathematical Journal}, 165(6), April 2016.

\bibitem{LSV}
CARLANGELO LIVERANI, BENOÎT SAUSSOL, and SANDRO VAIENTI.
\newblock A probabilistic approach to intermittency.
\newblock {\em Ergodic Theory and Dynamical Systems}, 19(3):671–685, 1999.

\bibitem{Lorenz}
Stefano Luzzatto, Ian Melbourne, and Frederic Paccaut.
\newblock The lorenz attractor is mixing.
\newblock {\em Communications in Mathematical Physics}, 260(2):393--401, 2005.

\bibitem{Lyons}
Russell Lyons.
\newblock Singularity of some random continued fractions.
\newblock {\em Journal of Theoretical Probability}, 13(2):535--545, 2000.

\bibitem{lyons1995ergodic}
Russell Lyons, Robin Pemantle, and Yuval Peres.
\newblock Ergodic theory on galton—watson trees: speed of random walk and
  dimension of harmonic measure.
\newblock {\em Ergodic Theory and Dynamical Systems}, 15(3):593--619, 1995.

\bibitem{Lyons_Peres_2017}
Russell Lyons and Yuval Peres.
\newblock {\em Probability on Trees and Networks}.
\newblock Cambridge Series in Statistical and Probabilistic Mathematics.
  Cambridge University Press, 2017.

\bibitem{Markarian}
ROBERTO MARKARIAN.
\newblock Billiards with polynomial decay of correlations.
\newblock {\em Ergodic Theory and Dynamical Systems}, 24(1):177–197, 2004.

\bibitem{mihailescu2016random}
Eugen Mihailescu and Mariusz Urba{\'n}ski.
\newblock Random countable iterated function systems with overlaps and
  applications.
\newblock {\em Advances in Mathematics}, 298:726--758, 2016.

\bibitem{MosqueraShmerkin}
Carolina Mosquera and Pablo Shmerkin.
\newblock Self-similar measures: asymptotic bounds for the dimension and
  {F}ourier decay of smooth images.
\newblock {\em Ann. Acad. Sci. Fenn. Math.}, 43(2):823--834, 2018.

\bibitem{orponenAdditive2023}
Tuomas Orponen.
\newblock Additive properties of fractal sets on the parabola.
\newblock {\em Annales Fennici Mathematici}, 48:113--139, 01 2023.

\bibitem{OdSS}
Tuomas Orponen, Nicolas de~Saxcé, and Pablo Shmerkin.
\newblock On the fourier decay of multiplicative convolutions, 2024.

\bibitem{orponen2025fourier}
Tuomas Orponen, Carmelo Puliatti, and Aleksi Py{\"o}r{\"a}l{\"a}.
\newblock On fourier transforms of fractal measures on the parabola.
\newblock {\em Transactions of the American Mathematical Society},
  378(10):7429--7450, 2025.

\bibitem{ParryPollicott}
William Parry and Mark Pollicott.
\newblock Zeta functions and the periodic orbit structure of hyperbolic
  dynamics.
\newblock {\em Ast\'{e}risque}, 187-188:268, 1990.

\bibitem{Par}
Emanuel Parzen.
\newblock {\em Modern Probability Theory and Its Applications}.
\newblock John Wiley \& Sons Canada, Limited, 1960.

\bibitem{PVZZ}
Andrew Pollington, Sanju Velani, Agamemnon Zafeiropoulos, and Evgeniy Zorin.
\newblock Inhomogeneous diophantine approximation on m0 -sets with restricted
  denominators.
\newblock {\em International Mathematics Research Notices}, 2022, 12 2020.

\bibitem{MP}
Yves Pomeau and Paul Manneville.
\newblock Intermittent transition to turbulence in dissipative dynamical
  systems.
\newblock {\em Communications in Mathematical Physics}, 74(2):189--197, 1980.

\bibitem{Sahlsten-survey}
Tuomas {Sahlsten}.
\newblock {Fourier transforms and iterated function systems}.
\newblock {\em arXiv e-prints}, page arXiv:2311.00585, November 2023.

\bibitem{SahlstenStevens}
Tuomas Sahlsten and Connor Stevens.
\newblock Fourier transform and expanding maps on cantor sets.
\newblock {\em American Journal of Mathematics}, pages 1--32, 2023.

\bibitem{Sarig2}
Omri Sarig.
\newblock Existence of gibbs measures for countable markov shifts.
\newblock {\em Proceedings of the American Mathematical Society},
  131(6):1751--1758, 2003.

\bibitem{Sarig1}
Omri~M. Sarig.
\newblock Phase transitions for countable markov shifts.
\newblock {\em Communications in Mathematical Physics}, 217(3):555--577, 2001.

\bibitem{Shmerkin-AbsContBernoulliConv}
Pablo Shmerkin.
\newblock On the exceptional set for absolute continuity of {B}ernoulli
  convolutions.
\newblock {\em Geom. Funct. Anal.}, 24(3):946--958, 2014.

\bibitem{SuomalaShmerkin}
Pablo Shmerkin and Ville Suomala.
\newblock Spatially independent martingales, intersections, and applications.
\newblock {\em Memoirs of the American Mathematical Society}, 251, 09 2014.

\bibitem{simon2001invariant}
K{\'a}roly Simon, Boris Solomyak, and Mariusz Urba{\'n}ski.
\newblock Invariant measures for parabolic ifs with overlaps and random
  continued fractions.
\newblock {\em Transactions of the American Mathematical Society},
  353(12):5145--5164, 2001.

\bibitem{solomyak1998non}
B~Solomyak.
\newblock Non-linear iterated function systems with overlaps.
\newblock {\em Periodica Mathematica Hungarica}, 37(1-3):127--141, 1998.

\bibitem{Solom}
Boris Solomyak.
\newblock On the random series $\sum \pm \lambda^n$ (an erd\"os problem).
\newblock {\em Annals of Mathematics}, 142(3):611--625, 1995.

\bibitem{solomyak2001q}
Boris Solomyak and Mariusz Urba{\'n}ski.
\newblock L q densities for measures associated with parabolic ifs with
  overlaps.
\newblock {\em Indiana University mathematics journal}, pages 1845--1866, 2001.

\bibitem{TsujiiZhang}
Masato Tsujii and Zhiyuan Zhang.
\newblock Smooth mixing {A}nosov flows in dimension three are exponentially
  mixing.
\newblock {\em Ann. of Math. (2)}, 197(1):65--158, 2023.

\bibitem{Varj__2019}
Péter Varjú.
\newblock On the dimension of bernoulli convolutions for all transcendental
  parameters.
\newblock {\em Annals of Mathematics}, 189(3), May 2019.

\bibitem{yi2024bounded}
Guangzeng Yi.
\newblock On bounded energy of convolution of fractal measures.
\newblock {\em arXiv preprint arXiv:2410.23080}, 2024.

\end{thebibliography}
\end{document}